\documentclass[12pt]{article}
\usepackage{amsmath,amssymb,fullpage,amsthm,bbm,comment}
\usepackage[numbers]{natbib}

\newtheorem{theorem}{Theorem}[section]
\newtheorem{lemma}[theorem]{Lemma}
\newtheorem{corollary}[theorem]{Corollary}
\newtheorem{remark}[theorem]{Remark}
\newtheorem{proposition}[theorem]{Proposition}
\theoremstyle{definition}
\newtheorem{definition}{Definition}[section]
\newtheorem{assumption}{Assumption}

\numberwithin{equation}{section}

\begin{document}
\begin{center}
\large{\textbf{Adaptive Supremum Norm Posterior Contraction: Wavelet Spike-and-Slab and Anisotropic Besov Spaces}}
\end{center}

\begin{center}
\textbf{William Weimin Yoo$^1$, Vincent Rivoirard$^2$ and Judith Rousseau$^3$}\\
\textit{Leiden University$^1$, Universit\'{e} Paris Dauphine$^2$, Oxford University$^3$}
\end{center}

\begin{abstract}
Supremum norm loss is intuitively more meaningful to quantify function estimation error in statistics. In the context of multivariate nonparametric regression with unknown error, we propose a Bayesian procedure based on spike-and-slab prior and wavelet projections to estimate the regression function and all its mixed partial derivatives. We show that their posterior distributions contract to the truth optimally and adaptively under supremum-norm loss. The master theorem through tests with exponential errors used in Bayesian nonparametrics was not adequate to deal with this problem, and we developed a new idea such that posterior under the regression model is systematically reduced to a posterior arising from some quasi-white noise model, where the latter model greatly simplifies our rate calculations. Hence, this paper takes the first step in showing explicitly how one can translate results from white noise to regression model in a Bayesian setting.
\end{abstract}

\noindent\textbf{Keywords:} Supremum norm, Adaptive posterior contraction, Nonparametric regression, Spike-and-Slab, Wavelet tensor products, Anisotropic Besov space, Type II error.\\

\noindent\textbf{MSC2010 classifications:} Primary 62G08; secondary 62G05, 62G10, 62G20

\section{Introduction}
Consider the nonparametric multivariate regression model
\begin{align}\label{eq:model}
Y_i=f(\boldsymbol{X}_i)+\varepsilon_i,\qquad i=1,\dotsc,n,
\end{align}
where $Y_i$ is a response variable, $\boldsymbol{X}_i$ is covariate, and $\varepsilon_1,\dotsc,\varepsilon_n$ are independent and identically distributed (i.i.d.) as $\mathrm{N}(0,\sigma^2)$ with unknown $0<\sigma<\infty$. Each $\boldsymbol{X}_i$ takes values in some rectangular region in $\mathbbm{R}^d$, which is assumed to be $[0,1]^d$ without loss of generality. The covariates can be deterministic or are sampled from a uniform distribution on $[0,1]^d$ independent of $\varepsilon_i$. There is some freedom in choosing the locations of the fixed covariates, as long as its empirical distribution can be approximated by a uniform distribution with an error of at most $n^{-1}$.

Suppose we observe $(Y_i,\boldsymbol{X}_i),i=1,\dotsc,n$, then our main problem is to recover or estimate the unknown $f$ and its mixed partial derivatives. In the literature, recovery is performed by minimizing certain loss functions, with the $L_2$ or integrated mean square error being the most common. However, other choices of loss, especially the supremum norm or $L_\infty$ is also of interest. Unlike the $L_2$-loss, the $L_\infty$-loss is intuitively more meaningful and hence a more natural distance to use to quantify the ``difference" between two functions. Moreover, $L_\infty$-distance is used to construct simultaneous credible bands, which are visually more interpretable in one dimension than $L_2$-credible sets. Also, it can be used to solve other problems such as function mode estimation discussed in Yoo and Ghosal \cite{yoo2017}.

Adaptive $L_2$-posterior contraction is a well-studied topic in Bayesian nonparametrics, where optimal procedures have been proposed for white noise models, inverse problems, nonparametric regression and density estimation (see Belitser and Ghosal \citep{belitser2003}, Ray \cite{ray2013}, Shen and Ghosal \cite{ShenandGhosal:DensityRegression}, van der Vaart and van Zanten \cite{vaart2009}). Results on $L_\infty$-contraction are much more limited. In the non-adaptive case, Gin\'{e} and Nickl \citep{nickl2011} studied contraction rates in $L_r$-metric, $1\leq r\leq\infty$, and obtained optimal rate using conjugacy for the Gaussian white noise model, and a rate for density estimation based on random wavelet series and Dirichlet process mixture, by using a testing approach based on concentration inequalities. In the same context, Castillo \citep{castillosupnorm} introduced techniques based on semiparametric Bernstein-von Misses theorems to obtain optimal $L_\infty$-contraction rates. Scricciolo \citep{scricciolo2014} applied the techniques of Gin\'{e} and Nickl \citep{nickl2011} to obtain $L_\infty$-rates using Gaussian kernel mixtures prior for analytic true densities. Using B-splines tensor product with Gaussian coefficients and by conjugacy arguments, Yoo and Ghosal \citep{yoo2015} established optimal $L_\infty$-posterior contraction rates for estimating multivariate regression function and its mixed partial derivatives.

To the best of our knowledge, there are only two papers on the adaptive case. In Hoffmann et al.~\citep{adaptsupnorm}, the authors established optimal $L_\infty$-contraction rate for the Gaussian white noise model and gave an existential result for density estimation; while in Chapter 3 of the thesis by Sniekers \citep{sniekers}, a near optimal rate is given implicitly through a result on credible bands for regression models. For models beyond the white noise, the first aforementioned paper used an abstract sieve prior construction to prove the existence of a Bayesian procedure, which is not readily implementable in practice; while the latter paper, which is based on a scaled Brownian motion prior, can only adapt up to H\"{o}lder smoothness of order two.

In this paper, we study a concrete hierarchical Bayesian method to estimate $f$ and its mixed partial derivatives adaptively under the $L_\infty$-loss for nonparametric multivariate regression models as in \eqref{eq:model}. We first represent $f$ as a finite combination of tensor product wavelet bases, and endow the basis coefficients with a spike-and-slab prior. We further endow the error variance $\sigma^2$ with a continuous prior density with support on $(0,\infty)$. Spike-and-slab is one of the most widely used prior in Bayesian statistics, particularly in connection to model selection (cf.~Mitchell and Beauchamp \cite{mitchell1988}, George and McCulloch \cite{george1993}, Ishwaran and Rao \cite{ishwaran2005}) and high-dimensional regression (cf.~Li and Zhang \cite{fan2010}, Castillo and van der Vaart \cite{castillo2012}, Castillo et al.~\cite{castillo2015}). When used together with wavelet bases, it can denoise noisy signals and compress images (see Chipman et al.~\citep{chipman1997}, Abramovich and Silverman \cite{abramovich1998}, Johnstone and Silverman \cite{johnstone2005}).

In the literature, an enormous amount of theoretical investigation has been devoted to the Gaussian white noise model, with only vague references to Le Cam's asymptotic equivalence theory claiming that results in this setting translate to more practical statistical problems such as regression considered in this paper. However in the present Bayesian setting, it is unclear whether such results still hold, in view of the fact that our design points are not exactly discrete uniform, but are close to uniform up to some error. In particular, we are not aware of any master theorem that makes this translation explicit in Bayesian asymptotics.

In this paper (Section \ref{sec:proof2}), we show how one can make the idea of asymptotic equivalence explicit in Bayeisan posterior rate computations. In our approach, we bound the posterior under the regression model with a posterior arising from a quasi-white noise model, with ``quasi" refers to the use of a scaling based on the wavelet basis Gram matrix rather than the standard $n^{-1/2}$ in white noise models. In a series of steps and intersecting with appropriately chosen events, this is achieved by reducing the regression likelihood to a likelihood that resembles and retains the component-wise structure of a white noise model, where the latter likelihood structure greatly simplifies our calculations and thus giving our method a certain Bayesian  ``asymptotic equivalence" flavor. We applied this new technique to establish optimal $L_\infty$-posterior contraction rates for the regression model of \eqref{eq:model} under the spike-and-slab tensor product wavelet prior, by reducing its posterior distribution to a quasi-white noise counterpart. Once in this simpler setup, we can then adapt proof techniques of Hoffmann et al.~\citep{adaptsupnorm}, which is based on the standard white noise model to prove our main result and translate them back to the more practical regression setting.

Our main result shows that spike-and-slab priors with appropriate weights can estimate $f$ and all its mixed partial derivatives optimally and adaptively under $L_\infty$-loss, in the sense that the resulting sup-norm posterior contraction rates match with the minimax rates for this problem. The scope of our result is quite general, in that we require only the slab prior density to be bounded from above and bounded from below on some interval, and this encompasses (nonconjugate) distributions such as Gaussian, sub-Gaussian, Laplace, uniform and most $t$-distributions. The Gaussian assumption of our errors in \eqref{eq:model} is simply a working model to derive expressions for the posterior, and our results will hold even if the model is misspecified and the actual data generation mechanism is sub-Gaussian.

The main challenge of this new approach is the handling of discretization inherent in regression models, as many convenient wavelet properties are lost when working in the discrete domain with finite data. As an example, the wavelet basis matrix constructed from wavelets evaluated at the design points is not orthogonal and this complicates analysis. We solve this problem by approximating discrete quantities or sums by its continuous or integral versions, and thus incurring approximation errors that we propagate throughout our calculations, while at the same time keeping them under control so as not to overwhelm stochastic and truncation errors (bias) in other parts of the problem. Another generalization we considered is to allow $f_0$ to be anisotropic, i.e., different smoothness in different dimensions, and we introduce a version of the anisotropic Besov space (see Definition \ref{definition:holder} below) suited for our analysis and we assume that the true regression function $f_0$ belongs to this space.

One might question the need of this new approach in Bayesian nonparametrics, as there is a state-of-the-art technique in the form of a master theorem to establish posterior contraction rates (cf.~Ghosal et al.~\citep{ghosal2000}, Shen and Wasserman \cite{shen2001}, Ghosal and van der Vaart \cite{converge2007}). One of the main criterion of this theorem is the existence of tests for the hypotheses $H_0:f=f_0$ against $H_1:f\in\{f:\|f-f_0\|_\infty>M\epsilon_n\}$ that have Type I error approaching zero and exponentially decreasing Type II error, where $\epsilon_n$ is the minimax rate for the regression problem under consideration. However, we show that this is impossible to achieve in general for sup-norm alternatives, as any such test has Type II error decreasing at least polynomially in $n$. For exponential error test to exists, we show that the null and the alternative hypotheses must be further separated in $L_\infty$-norm, and this increase in separation results in the contraction rate being inflated by the same (polynomial) factor (see also Hoffmann et al.~\citep{adaptsupnorm} for a related discussion on this sub-optimality issue). The proposed approach circumvents this problem and gives optimal rates.


The paper is organized as follows. The next section introduces notations. Section \ref{sec:prior} describes the prior and the assumptions used in this paper. The main result on adaptive $L_\infty$-posterior contraction, for $f$ and its mixed partial derivatives are presented in Section \ref{sec:result}, and this is followed by a discussion on the lower limit of adaptation. The inadequacy of the master theorem is detailed in Section \ref{sec:test}. Section \ref{sec:proof} contains proofs of all main results and is further divided into three subsections. The proof of our $L_\infty$-contraction result is in Section \ref{sec:proof1}, we introduce our posterior bounding technique in Section \ref{sec:proof2} and the rest of the proofs are gathered in Section \ref{sec:testproof}. The last Section \ref{sec:proof3} contains technical lemmas used throughout the proofs, where some results such as continuous approximations to discrete objects and $L_2$-contraction for spike-and-slab priors are of independent interests.

\section{Notations}
Given two numerical sequences $a_n$ and $b_n$, $a_n=O(b_n)$ or $a_n\lesssim b_n$ means $a_n/b_n$ is bounded, while $a_n=o(b_n)$ of $a_n\ll b_n$ means $a_n/b_n\rightarrow0$. If $a_n\asymp b_n$, then we have $a_n=O(b_n)$ and $b_n=O(a_n)$. For stochastic sequence $Z_n$, $Z_n=O_P(a_n)$ means $Z_n/a_n$ is bounded in $P$-probability, while $Z_n=o_P(a_n)$ means $Z_n/a_n$ converges to $0$ in $P$-probability. Define $\mathbb{N}=\{1,2,\dotsc\}$ to be the set of natural numbers and $\mathbb{N}_0=\mathbb{N}\cup\{0\}$.

Define $\|\boldsymbol{x}\|_p=(\sum_{k=1}^d |x_k|^p)^{1/p}$, $1\le p<\infty$, $\|\boldsymbol{x}\|_\infty=\max_{1\leq k\leq d}|x_k|$ and write $\|\boldsymbol{x}\|$ for $\|\boldsymbol{x}\|_2$ the Euclidean norm. For $f:U\rightarrow\mathbb{R}$ on some bounded set $U\subseteq\mathbb{R}^d$ with interior points, let $\|f\|_p$ be the $L_p$-norm, and $\|f\|_\infty=\sup_{x\in U}|f(x)|$. For $\boldsymbol{r}=(r_1,\dotsc,r_d)^T\in\mathbb{N}_0^d$ and $|\boldsymbol{r}|:=\sum_{k=1}^dr_k$, let $D^{\boldsymbol{r}}$ be the partial derivative operator $\partial^{|\boldsymbol{r}|}/\partial x_1^{r_1}\cdots\partial x_d^{r_d}$.  For a set $\mathcal{A}$, let $\mathbbm{1}_{\mathcal{A}}$ be the indicator function on $\mathcal{A}$. For a vector $\boldsymbol{x}$, we write $x_{\boldsymbol{j}}$ to be its $\boldsymbol{j}$th component with $\boldsymbol{j}$ possibly be multi-index $(j_1,\dotsc,j_d)^T$, and in that case we let the entries be ordered lexicographically.

\section{Wavelet series with spike-and-slab prior}\label{sec:prior}
Since our domain of interest is bounded i.e., $[0,1]^d$, we will use the boundary corrected wavelets introduced by Cohen-Daubechies-Vial (CDV) in Section 4 of \citep{waveletb}. At each dimension $l=1,\dotsc,d$, the CDV wavelets are constructed from the usual Daubechies wavelet system on $\mathbb{R}$, by retaining wavelets supported in the interior of $[0,1]$ and replacing the wavelets near $\{0,1\}$ with boundary corrected versions, such that the entire system still generates a multiresolution analysis on $[0,1]$ and is orthonormal. For some $N_l$ to be chosen below, we write the system's father and mother wavelets as $\varphi_{N_l,m_l}(x)=2^{N_l/2}\varphi_{m_l}(2^{N_l}x)$ and $\psi_{j_l,k_l}(x)=2^{j_l/2}\psi_{k_l}(2^{j_l}x)$ where $0\leq m_l\leq2^{N_l}-1,j_l\geq N_l$ and $0\leq k_l\leq2^{j_l}-1$. For the interior wavelets, $\varphi_{m_l}(x)=\varphi(x-m_l)$ and $\psi_{k_l}(x)=\psi(x-k_l)$ are the translated original Daubechies system and for the boundary wavelets, $\varphi_{m_l}$ and $\psi_{k_l}$ are some linear combinations of this system. We take the CDV wavelets to be $\eta$-regular at each direction (see Definition 4.2.14 of \citep{nickl2016}) such that the derivatives $\varphi_{m_l}^{(r_l)},\psi_{k_l}^{(r_l)}$ are uniformly bounded for $r_l<\eta+1$.

For $\boldsymbol{x}=(x_1,\dotsc,x_d)$, we construct tensor products of the CDV father and mother wavelets as $\varphi_{\boldsymbol{N},\boldsymbol{m}}(\boldsymbol{x})=\prod_{l=1}^d\varphi_{N_l,m_l}(x_l)$ and $\psi_{\boldsymbol{j},\boldsymbol{k}}(\boldsymbol{x})=\prod_{l=1}^d\psi_{j_l,k_l}(x_l)$ respectively, where $\boldsymbol{m}=(m_1,\dotsc,m_d), \boldsymbol{N}=(N_1,\dotsc,N_d),\boldsymbol{j}=(j_1,\dotsc,j_d)$ and $\boldsymbol{k}=(k_1,\dotsc,k_d)$. Since the CDV wavelets are unconditional $L_2$-bases, we can expand $f$ in the multivariate regression model of \eqref{eq:model} using these bases, and this leads us to consider the following hierarchical priors to study sup-norm posterior contraction:
\begin{gather}
f(\boldsymbol{x})=\sum_{m_1=0}^{2^{N_1}-1}\dotsc\sum_{m_d=0}^{2^{N_d}-1}\vartheta_{\boldsymbol{m}}\varphi_{\boldsymbol{N},\boldsymbol{m}}(\boldsymbol{x})
+\sum_{j_1=N_1}^{J_{n,1}-1}\sum_{k_1=0}^{2^{j_1}-1}\cdots\sum_{j_d=N_d}^{J_{n,d}-1}\sum_{k_d=0}^{2^{j_d}-1}\theta_{\boldsymbol{j},\boldsymbol{k}}\psi_{\boldsymbol{j},\boldsymbol{k}}(\boldsymbol{x}),\nonumber\\
\vartheta_{\boldsymbol{m}}\overset{i.i.d.}{\sim}p(\cdot)\nonumber\\
\theta_{\boldsymbol{j},\boldsymbol{k}}\overset{i.i.d.}{\sim}(1-\omega_{j_1,\dotsc j_d,n})\delta_0(\cdot)+\omega_{j_1,\dotsc,j_d,n}p(\cdot),\nonumber\\
\sigma\sim\pi_{\sigma}.\label{eq:prior}
\end{gather}
The prior on the mother coefficient is called a spike-and-slab, with the spike part corresponding to the point mass at $0$ ($\delta_0$ is the Dirac function) and the slab part some density $p(\cdot)$ on $\mathbb{R}$. With appropriate chosen weights $\omega_{j_1,\dotsc,j_d}$, it does a form of model selection by zeroing ``unimportant" coefficients. Observe that we only assign spike-and-slab priors on the mother wavelet coefficients, this is done to prevent overly sparse models by allowing father coefficients to capture global structures of $f$. The truncation point $J_{n,l}$ at some fixed $l=1,\dotsc,d$ is a sequence of positive integers increasing with $n$, such that $\prod_{l=1}^d2^{J_{n,l}}=\sqrt{n/\log{n}}$ for both fixed and random design points, where the division by $\log{n}$ is a technical requirement. The presence of a square root here is due to the method of our proof, and it relates to the fact that we can only reduce a regression likelihood to a corresponding white noise version, when there is a lower limit imposed on the true function smoothness that we can adapt to (see Section \ref{sec:lower} for more details). Therefore for true functions that are sufficiently smooth, our theory suggests that it suffices to take $\sqrt{n/\log{n}}$ for the regression model as opposed to $n$ used in \citep{adaptsupnorm} for the white noise model, and hence our lower truncation point speeds up wavelet computations by reducing the size of the candidate coefficients. Here $N_l$ is a positive integer such that $2^{N_l}\geq2\eta$. Also, we assume that the priors on $\{\vartheta_{\boldsymbol{m}}\}, \{\theta_{\boldsymbol{j},\boldsymbol{k}}\}, \sigma^2$ are mutually independent with each other. For the spike-and-slab weights, we let $n^{-\lambda}\leq\omega_{j_1,\dotsc,j_d,n}\leq\min\{2^{-\sum_{l=1}^dj_l(1+\mu_l)},1/2\}$ for some $\lambda>0$ and $\mu_l>1/2,l=1,\dotsc,d$. Here, $p(\cdot)$ is such that $p_{\mathrm{max}}=\sup_{x\in\mathbb{R}}p(x)<\infty$ and for some $R_0>0$,
\begin{align}\label{eq:density}
\inf_{x\in[-R_0,R_0]}p(x)=p_{\mathrm{min}}>0.
\end{align}
Examples of $p(\cdot)$ include the Gaussian, sub-Gaussian, Laplace, the uniform $[-R_0,R_0]$, $t$-distributions and most commonly used parametric families. We let $\pi_{\sigma}$ be a positive and continuous prior density with support on $(0,\infty)$, e.g., inverse gamma distribution.

If the covariates $\boldsymbol{X}_i=(X_{i1},\dotsc,X_{id})^T$ for $i=1,\dotsc,n$ are fixed, we assume they are chosen such that
\begin{align}\label{eq:cdf}
\sup_{\boldsymbol{x}\in[0,1]^d}|G_n(\boldsymbol{x})-U(\boldsymbol{x})|=O\left(\frac{1}{n}\right),
\end{align}
where $U(\boldsymbol{x})$ is the cumulative distribution function of a uniform on $[0,1]^d$, and $G_n(\boldsymbol{x})$ is the empirical cumulative distribution function of $\{\boldsymbol{X}_i,i=1,\dotsc,n\}$, that is, $G_n(\boldsymbol{x})=n^{-1}\sum_{i=1}^n\mathbbm{1}_{\prod_{l=1}^d[0,X_{il}]}(\boldsymbol{x})$. This requirement can be fulfilled if we used a discrete uniform design, that is for $n=m^d$ for some $m\in\mathbb{N}$, $\boldsymbol{X}_i\in\{(j-1)/(m-1):j=1,\dotsc,m\}^d$ with $i=1,\dotsc,n$. We will mainly discuss and prove results based on fixed design points, and we make brief remarks concerning the random case.

\begin{remark}
To be technically precise, there should be indices to indicate the fact that combinations of both father and mother wavelet components are used to construct $\psi_{\boldsymbol{j},\boldsymbol{k}}$. In particular, let $\mathcal{I}$ be the set of $2^d-1$ sequences of the form $(i_1,\dotsc,i_d)$, such that each $i_l$ can be $0$ or $1$, but excluding the case where $i_l=0$ for all $l$. Then $\psi_{\boldsymbol{j},\boldsymbol{k}}$ is augmented to $\psi_{\boldsymbol{j},\boldsymbol{k}}^{\boldsymbol{i}}=\prod_{l=1}^d\psi_{j_l,k_l}^{i_l},\boldsymbol{i}\in\mathcal{I}$ such that $\psi_{j_l,k_l}^0=\varphi_{N_l,m_l}$ and $\psi_{j_l,k_l}^1=\psi_{j_l,k_l}$. However, since $\boldsymbol{i}\in\mathcal{I}$ are simply identification indices to dictate which tensor product component is a father or mother wavelet, and coupled with the fact that $\sum_{\boldsymbol{i}\in\mathcal{I}}=2^d-1$ does not grow with $n$, we drop this identification in this paper and simply work with $\psi_{\boldsymbol{j},\boldsymbol{k}}$. This is to improve readability and help readers focus on the main ideas instead of the technicalities of working in multiple dimensions.
\end{remark}

To study $L_\infty$-posterior contraction for mixed partial derivatives, we apply the differential operator $D^{\boldsymbol{r}}$ on both sides of the wavelet expansion in \eqref{eq:prior} to yield
\begin{align*}
D^{\boldsymbol{r}}f=\sum_{\boldsymbol{m}}\vartheta_{\boldsymbol{m}}D^{\boldsymbol{r}}\varphi_{\boldsymbol{N},\boldsymbol{m}}+
\sum_{\boldsymbol{j},\boldsymbol{k}}\theta_{\boldsymbol{j},\boldsymbol{k}}D^{\boldsymbol{r}}\psi_{\boldsymbol{j},\boldsymbol{k}},
\end{align*}
where the priors on $\vartheta_{\boldsymbol{m}}$ and $\theta_{\boldsymbol{j},\boldsymbol{k}}$ are the same as in \eqref{eq:prior}. To study both $f$ and its derivatives in the same framework, we adopt the convention $D^{\boldsymbol{0}}f\equiv f$. We note that objects such as $D^{\boldsymbol{r}}\varphi_{\boldsymbol{N},\boldsymbol{m}}$ and $D^{\boldsymbol{r}}\psi_{\boldsymbol{j},\boldsymbol{k}}$ are called vaguelet tensor products (see \citep{cai2002}).

To study frequentist properties and derive contraction rates for our posterior, we assume the existence of an underlying true function $f_0$, such that $f_0$ belongs to an anisotropic Besov space as defined below. Let us first denote $\alpha^{*}$ to be the harmonic mean of $\boldsymbol{\alpha}=(\alpha_1,\dotsc,\alpha_d)^T$, i.e., $(\alpha^{*})^{-1}=d^{-1}\sum_{l=1}^d\alpha_l^{-1}$.

\begin{definition}[Anisotropic Besov space]\label{definition:holder}
The anisotropic Besov function space $\mathcal{B}^{\boldsymbol{\alpha}}_{p,q}$ for $\boldsymbol{\alpha}=(\alpha_1,\dotsc,\alpha_d)^T$ such that $0<\alpha_l<\eta+1,l=1,\dotsc,d$ and $1\leq p,q\leq\infty$ is given as
\begin{align}
\mathcal{B}^{\boldsymbol{\alpha}}_{p,q}\equiv\begin{cases}\{f\in L_p([0,1]^d):\|f\|_{\mathcal{B}^{\boldsymbol{\alpha}}_{p,q}}<\infty\},&1\leq p<\infty,\\
\{f\in C_u([0,1]^d):\|f\|_{\mathcal{B}^{\boldsymbol{\alpha}}_{p,q}}<\infty\},&p=\infty\end{cases}
\end{align}
with $C_u([0,1]^d)$ the space of uniformly continuous functions on $[0,1]^d$, and the anisotropic Besov norm $\|f\|_{\mathcal{B}^{\boldsymbol{\alpha}}_{p,q}}$ is $\left(\sum_{\boldsymbol{m}}|\langle f, \varphi_{\boldsymbol{N},\boldsymbol{m}}\rangle|^p\right)^{\frac{1}{p}}$ plus
\begin{align}\label{eq:besov}
\begin{cases}
\left[\sum\limits_{\boldsymbol{j}}2^{q\sum_{l=1}^d\alpha_lj_l\left(\frac{1}{d}+\frac{1}{2\alpha^{*}}-\frac{1}{\alpha^{*}p}\right)}\left(\sum\limits_{\boldsymbol{k}}|\langle f,
\psi_{\boldsymbol{j},\boldsymbol{k}}\rangle|^p\right)^{\frac{q}{p}}\right]^{\frac{1}{q}},&1\leq q<\infty,\\
\sup\limits_{\boldsymbol{j}}2^{\sum_{l=1}^d\alpha_lj_l\left(\frac{1}{d}+\frac{1}{2\alpha^{*}}-\frac{1}{\alpha^{*}p}\right)}\left(\sum\limits_{\boldsymbol{k}}|\langle f,
\psi_{\boldsymbol{j},\boldsymbol{k}}\rangle|^p\right)^{\frac{1}{p}},&q=\infty,\end{cases}
\end{align}
where we replace the $\ell_p$-sequence norm with the $\|\cdot\|_\infty$-norm when $p=\infty$.
\begin{remark}\label{rem:b0}
If we set $\alpha_l=\alpha$ and take $\alpha\rightarrow0$, then we can define the Besov spaces $\mathcal{B}^{\boldsymbol{0}}_{p,q}$, which is the multivariate and anisotropic generalization of its univariate counterpart discussed in Section 4.3.2 of \citep{nickl2016}. In this case, we replace $\sum_{l=1}^d\alpha_lj_l(\frac{1}{d}+\frac{1}{2\alpha^{*}}-\frac{1}{\alpha^{*}p})$ in the exponent by $\sum_{l=1}^dj_l/2$.
\end{remark}

Let $K_{\boldsymbol{W}}(\boldsymbol{x},\boldsymbol{y})=\sum_{\boldsymbol{m}}\varphi_{\boldsymbol{W},\boldsymbol{m}}(\boldsymbol{x})\varphi_{\boldsymbol{W},\boldsymbol{m}}(\boldsymbol{y})$, and define the operator $K_{\boldsymbol{W}}$ on $L_p$ such that $K_{\boldsymbol{W}}(g)(\boldsymbol{x})=\int K_{\boldsymbol{W}}(\boldsymbol{x},\boldsymbol{y})g(\boldsymbol{y})d\boldsymbol{y}$ for $g\in L_p$. Thus, observe that $K_{\boldsymbol{W}}(g)$ is the $L_2$-projection of $g$ onto the subspace spanned by $\{\varphi_{N,\boldsymbol{m}}:0\leq m_l\leq2^{N_l}-1\}\cup\{\psi_{\boldsymbol{j},\boldsymbol{k}}:N_l\leq j_l\leq W_l-1,0\leq k_l\leq2^{j_l}-1\},l=1,\dotsc,d$. That is, $K_{\boldsymbol{W}}(g)$ has wavelet expansion as in \eqref{eq:prior} but truncated at levels $\boldsymbol{W}=(W_1,\dotsc,W_d)^T$. The proposition below then tells us how well these anisotropic wavelet projections approximate functions in $\mathcal{B}^{\boldsymbol{\alpha}}_{p,q}$.
\begin{proposition}\label{prop:wavelet}
Let $0<\alpha_l<\eta+1, l=1,\dotsc,d$. For any $g\in\mathcal{B}^{\boldsymbol{\alpha}}_{p,q}$, let $K_{\boldsymbol{W}}(g)$ be its projected version at level $\boldsymbol{W}$ as described above, then there exists constant $C>0$ depending on the wavelets used such that
\begin{align}
\|K_{\boldsymbol{W}}(g)-g\|_p\leq C\|g\|_{\mathcal{B}^{\boldsymbol{\alpha}}_{p,q}}\sum_{l=1}^d2^{-\alpha_lW_l}.
\end{align}
\end{proposition}
The proof of this proposition can be found in Section \ref{sec:proof1}. We are now ready to introduce the assumptions on the underlying true model for \eqref{eq:model}.
\end{definition}
\begin{assumption}\label{assump:assumption}
Under the true distribution $P_0$, $Y_i=f_0(\boldsymbol{X}_i)+\varepsilon_i, i=1,\dotsc,n$, where $f_0\in\mathcal{B}^{\boldsymbol{\alpha}}_{\infty,\infty}$ and $\varepsilon_i$ are i.i.d.~Gaussian with mean $0$ and finite variance $\sigma_0^2>0$ for $i=1,\dotsc,n$. Here, $\boldsymbol{\alpha}=(\alpha_1,\dotsc,\alpha_d)^T\in(0,\infty)^d$ is unknown.
\end{assumption}

\begin{remark}
Inspection of the main proof shows that we can actually relax the assumption on errors so that they are sub-Gaussian, and hence allowing the model to be possibly misspecified. However, we would need to use the misspecified version of the master theorem (Theorem 4.1 of \citep{miss}) to prove $L_2$-contraction for the mother wavelet coefficients as part of the overall proof. We refrain from doing this because this will add extra technicalities that are a distraction for the main $L_\infty$-task at hand.
\end{remark}

We define $\vartheta_{\boldsymbol{m}}^0=\langle f_0, \varphi_{\boldsymbol{N},\boldsymbol{m}}\rangle$ and $\theta_{\boldsymbol{j},\boldsymbol{k}}^0=\langle f_0, \psi_{\boldsymbol{j},\boldsymbol{k}}\rangle$ to be the true wavelet coefficients. We denote $\mathrm{E}_0(\cdot)$ as the expectation operator taken with respect to $P_0$ and write $\boldsymbol{Y}=(Y_1,\dotsc,Y_n)^T$. Moreover, we write Besov ball of radius $R>0$ as $\mathcal{B}^{\boldsymbol{\alpha}}_{p,q}(R):=\{f:\|f\|_{\mathcal{B}^{\boldsymbol{\alpha}}_{p,q}}\leq R\}$.

\section{Adaptive posterior contraction}\label{sec:result}
Before establishing supremum norm contraction rate for $f$, a preliminary key step is to show that the posterior distribution of $\sigma$ is consistent under the hierarchical priors of \eqref{eq:prior}. We therefore begin with the following proposition whose proof is given in Section \ref{sec:proof1}.

\begin{proposition}\label{prop:consistent}
Under Assumption \ref{assump:assumption}, we can conclude that for any prior on $\sigma$ with positive and continuous density, the posterior distribution of $\sigma$ is consistent, uniformly over $f_0\in\mathcal{B}^{\boldsymbol{\alpha}}_{\infty,\infty}(R)$ for any $R>0$ and for any $\boldsymbol{\alpha}$ such that $0<\alpha_l<\eta+1$ where $\eta$ is the regularity of the wavelet bases.
\end{proposition}

Using wavelet expansions such as \eqref{eq:prior}, we can work with wavelet coefficients instead of $f$ and treat them as component-wise signals we are trying to recover. In all our calculations, the threshold $\gamma\sqrt{\log{n}/n}$ with some appropriately chosen constant $\gamma>0$ is of crucial importance as it serves as a cutoff to determine statistically which signal is considered ``large" or ``small". The speed of which the posterior will contract to the truth in $L_\infty$-norm is then dictated by these two conditions:

\begin{enumerate}
\item Signal detection errors, i.e., $\theta_{\boldsymbol{j},\boldsymbol{k}}=0$ but the true signal is ``large" $|\theta_{\boldsymbol{j},\boldsymbol{k}}^0|>\gamma\sqrt{\log{n}/n}$ for some large enough $\gamma$ and vice versa are unlikely to occur under the posterior.
\item The posterior concentrates within a $\gamma\sqrt{\log{n}/n}$-neighborhood (for some large enough $\gamma$) around large detectable signals.
\end{enumerate}

Asymptotically, this implies that the spike-and-slab posterior behaves like a local wavelet thresholding operator, which does coefficient-wise thresholding with $\gamma\sqrt{\log{n}/n}$ as threshold. During the course of establishing these conditions, we have to deal with discrete approximation errors as encoded in \eqref{eq:cdf}, finite truncation error of Proposition \ref{prop:wavelet} and stochastic error in our model \eqref{eq:model}. This requires a very delicate balancing of these opposing errors that we propagate throughout our calculations, and we arrive at our results by ensuring that no single source of error will dominate the others.

In many applications such as model selection and high-dimensional regression, the weights $\omega_{\boldsymbol{j},n}$ are typically endowed with another layer of hyper-prior or estimated using empirical Bayes (e.g.,\citep{ishwaran2005,johnstone2005,castillo2015}). However for sup-norm posterior contraction, it suffices to choose them fixed beforehand as was done in Section \ref{sec:prior}. This is because in order to reduce signal detection error as alluded in $1$.~above, $\theta_{\boldsymbol{j},\boldsymbol{k}}$ needs to decay in similar manner as $\theta_{\boldsymbol{j},\boldsymbol{k}}^0$ does, i.e., in the same form as \eqref{eq:infbesov}. To ensure this using the spike-and-slab prior, it is then enough to set $\omega_{\boldsymbol{j},n}\leq2^{-\sum_{l=1}^dj_l(1+\mu_l)}$. The following main results show that by selecting coefficients using these fixed weights, the posteriors of $f$ and its mixed partial derivatives contract adaptively in $L_\infty$ at the optimal rate to the truth. Note here that the same weights can be used for $f$ and all orders of its mixed partial derivatives. Clearly, one cannot adapt at each dimension beyond the regularity $\eta$ of the wavelets, but it will be seen that there is a lower limit of adaptation present that prevents us to adapt arbitrarily close to $0$ (see Section \ref{sec:lower} for a more thorough discussion). Therefore, this lead us to formulate our range of adaptation as
\begin{align}\label{eq:adaptive}
\mathbb{A}_{\boldsymbol{r}}&=\left\{\boldsymbol{\alpha}:\frac{2(r_l+1)\alpha^{*}d}{2\alpha^{*}+d}<\alpha_l<\eta+1, l=1,\dotsc,d\right\},
\end{align}
and if $\boldsymbol{r}=\boldsymbol{0}$, we simply write $\mathbb{A}_{\boldsymbol{0}}$ as $\mathbb{A}$.

\begin{theorem}(Adaptive $L_\infty$-contraction)\label{th:adapt}\\
(a) For the regression function:\\
For any $0<R\leq R_0-1/2$ and some constants $\xi,M>0$,
\begin{align*}
\sup_{\boldsymbol{\alpha}\in\mathbb{A}}\sup_{f_0\in\mathcal{B}^{\boldsymbol{\alpha}}_{\infty,\infty}(R)}
\mathrm{E}_0\Pi\left(\|f-f_0\|_\infty>M\left(n/\log{n}\right)^{-\frac{\alpha^{*}}{2\alpha^{*}+d}}\middle|\boldsymbol{Y}\right)\leq\frac{(\log{n})^d}{n^\xi}.
\end{align*}

\noindent(b) For mixed partial derivatives:\\
Let $\boldsymbol{r}\geq\boldsymbol{0}$ such that $\boldsymbol{r}\neq\boldsymbol{0}$ and $r_l<\eta+1,l=1,\dotsc,d$. Then for any $0<R\leq R_0-1/2$ and some constants $\xi,M>0$, we have uniformly over $\boldsymbol{\alpha}\in\mathbb{A}_{\boldsymbol{r}}$ and $f_0\in\mathcal{B}^{\boldsymbol{\alpha}}_{\infty,\infty}(R)$ that
\begin{align*}
\mathrm{E}_0\Pi\left(\|D^{\boldsymbol{r}}f-D^{\boldsymbol{r}}f_0\|_\infty>M(n/\log{n})^{-\frac{\alpha^{*}\{1-\sum_{l=1}^d(r_l/\alpha_l)\}}{2\alpha^{*}+d}}\middle|\boldsymbol{Y}\right)\leq\frac{(\log{n})^d}{n^\xi}.
\end{align*}
\end{theorem}

\begin{remark}\label{rem:dim}
For the isotropic case where $\alpha_l=\alpha,l=1,\dotsc,d$, $\mathbb{A}_{\boldsymbol{r}}$ is defined through $\max\{d/2,\sum_{l=1}^dr_l\}<\alpha<\eta+1$.
\end{remark}

The proof of Theorem \ref{th:adapt} is given in Section \ref{sec:proof1}, and it has important implication in frequentist statistics. In particular, the posterior mean as an adaptive point estimator converges uniformly to $f_0$ at the same rate.

\begin{corollary}\label{cor:adapt}
Let $\boldsymbol{r}\geq\boldsymbol{0}$ such that $\boldsymbol{r}\neq\boldsymbol{0}$ and $r_l<\eta+1,l=1,\dotsc,d$, then for any $0<R\leq R_0-1/2$,
\begin{align*}
\sup_{\boldsymbol{\alpha}\in\mathbb{A}_{\boldsymbol{r}}}\sup_{f_0\in\mathcal{B}^{\boldsymbol{\alpha}}_{\infty,\infty}(R)}
\mathrm{E}_0\|\mathrm{E}(D^{\boldsymbol{r}}f|\boldsymbol{Y})-D^{\boldsymbol{r}}f_0\|_\infty\lesssim(n/\log{n})^{-\frac{\alpha^{*}\{1-\sum_{l=1}^d(r_l/\alpha_l)\}}{2\alpha^{*}+d}}.
\end{align*}
\end{corollary}
We note here that the $\|\cdot\|_\infty$-norm is not bounded, and hence the usual route of deriving this result from Theorem \ref{th:adapt} based on convex, bounded loss through Jensen's inequality is not applicable. To proceed, we compute the expectation on slicing of the function space and add these polynomially decaying terms to bound the expectation of interest on the entire space. For more details, see the proof at the end of Section \ref{sec:proof2}.

\begin{remark}
For the random case, we assume that $\boldsymbol{X}_1,\dotsc,\boldsymbol{X}_n\overset{i.i.d.}{\sim}U(\boldsymbol{x})$, where $U$ is the cumulative distribution function of a uniform distribution on $[0,1]^d$. Note that random design points do not satisfy \eqref{eq:cdf} because by Donsker's theorem, we will have $\sup_{\boldsymbol{x}\in[0,1]^d}|G_n(\boldsymbol{x})-U(\boldsymbol{x})|=O_P(n^{-1/2})$. However, by following similar calculations performed for the fixed design case, it turns out that we will get the exact same posterior contraction rates as in Theorem \ref{th:adapt} for random uniform designs, and the effect of the extra $\sqrt{n}$-factor shows up through a slightly larger lower bound in $\mathbb{A}_{\boldsymbol{r}}$.
\end{remark}

\subsection{Discussion on \eqref{eq:adaptive}: limits on range of adaptation}\label{sec:lower}
Theorem \ref{th:adapt} in particular shows that there is a certain lower limit in the range of smoothness that we can adapt to, and this limit is increasing with $d$ the ambient dimension. To see this point, take $\boldsymbol{r}=\boldsymbol{0}$, rearrange the lower bound in \eqref{eq:adaptive} to $1/d+1/(2\alpha^{*})>1/\alpha_l$, and sum both sides across $l=1,\dotsc,d$ to get $\alpha^{*}>d/2$. This lower limit implies that our range of adaptation shrinks and we can only adapt to smoother functions in higher dimensions.

In our approach, this limit arises when we try to reduce the regression posterior to a quasi-white noise version, by forcing the regression likelihood based on sums of squares into component-wise fashion like those encountered in white noise models. The success of this reduction depends on the truncation point $2^{\sum_{l=1}^dJ_{n,l}}$ and also the lower bound on $\alpha^{*}$ we are willing to tolerate. Now suppose $2^{\sum_{l=1}^dJ_{n,l}}=(n/\log{n})^m$ for some $m\leq1$, then Lemma \ref{lem:betabound} in Section \ref{sec:proof2} below shows that we will be able to perform this reduction and hence establish our main results in the previous section, if for all $\boldsymbol{\alpha}$,
\begin{align*}
\alpha^{*}>d\max\left\{\frac{(m-1)+\sqrt{(m-1)^2+8m^2}}{4m},\quad\frac{1}{2}\left(\frac{1}{m}-1\right)\right\}.
\end{align*}
Note that the first term inside the max operation is increasing $m$ while the second term is decreasing. Thus the optimal $m$ can be found by equating these two antagonistic terms and this will yield $m=1/2$ giving the smallest lower bound $\alpha^{*}>d/2$, which is ensured by letting $\boldsymbol{\alpha}\in\mathbb{A}_{\boldsymbol{r}}\subseteq\mathbb{A}$ in \eqref{eq:adaptive} for any $\boldsymbol{r}\geq\boldsymbol{0}$. Moreover, this also explains why we chose $2^{\sum_{l=1}^dJ_{n,l}}=\sqrt{n/\log{n}}$ as our truncation point in Section \ref{sec:prior}. From this perspective, \eqref{eq:adaptive} arises due to our method of proof.

Interestingly, such lower limit has been observed in the one-dimensional case and in other settings (see \citep{nickl2011,adaptsupnorm,sniekers,castillosupnorm}). Based on the current known literature so far, this limit appears when one tries to establish sup-norm posterior contraction rates for models beyond the Gaussian white noise, e.g., nonparametric regression and density estimation. On the other hand, if we look at frequentist procedures such as Lepski's method and wavelet thresholding, such lower limit apparently do not exists at least for regression and density estimation problems, and one can adapt the function smoothness arbitrarily close to $0$.

By taking into account both perspectives, we do not know whether this is due to artefacts of proof methods, or to some deeper reasons such as the incompatibility of fully Bayesian procedures (which is usually based on intrinsic $L_2$-metric) to the desired $L_\infty$-loss. However since the aforementioned papers and our proposed method all arrived at some lower limits, despite using different proof techniques and priors, we conjecture that the former reason is unlikely. Admittedly, this is far from conclusive and further research is needed to verify these claims.



\section{The master theorem of Bayesian nonparametrics}\label{sec:test}
In this section, we explain in detail why we have to develop a new method of deriving contraction rates by comparing regression posterior with a corresponding quasi-white noise version. In Bayesian nonparametrics, the current state-of-the-art method in calculating posterior contraction rates is the master theorem developed by \citep{ghosal2000,shen2001,converge2007}. As its name suggests, this master theorem consists of sufficient conditions that are designed to be applicable to general classes of models and prior distributions. Let $\epsilon_n$ be the minimax rate and $\Pi$ a general prior distribution, not necessarily the spike-and-slab prior considered in previous sections. In its most basic version adapted to our present regression model, these conditions are ($C_1,C_2,C_3>0$ are some constants):
\begin{enumerate}
\item The existence of a sequence of tests $\phi_n$ for the hypotheses $H_0:f=f_0$ against $H_1:f\in\{\mathcal{F}_n:\|f-f_0\|_\infty>M\epsilon_n\}$ with $\mathcal{F}_n$ some appropriately chosen sequence of sieve sets of the function parameter space, such that its Type I error goes to $0$ while its Type II error decreases like $e^{-C_1n\epsilon_n^2}$,
\item The prior $\Pi$ puts at least $e^{-C_2n\epsilon_n^2}$ mass on certain Kullback-Leibler neighborhoods around $f_0$ of radius $\epsilon_n$,
\item The prior $\Pi$ puts most of its mass in the sieve sets such that $\Pi(\mathcal{F}_n^c)\leq e^{-C_3n\epsilon_n^2}$.
\end{enumerate}
Recently however, research in this area has discovered cases that do not fall within the scope of this master theorem. In our context of $L_\infty$-contraction, the works by \citep{nickl2011} and \citep{adaptsupnorm} found that the master theorem, which corresponds to verifying the 3 conditions above, can produce suboptimal contraction rates. In the following, we will investigate this issue in more depth and give a more thorough explanation in higher dimensions.

Throughout this section, we take $\sigma=\sigma_0$ to be known since it does not play a role in explaining this suboptimality issue, this is done to streamline the proofs and help readers better understand the cause of this problem, which is driven by the nonparametric part i.e., the regression function $f$ of the model.

The root of this problem is the first testing criterion, which requires us to find some sequence of test functions $\phi_n$ with exponentially decreasing Type II errors, or more precisely, $\sup_{f\in\mathcal{F}_n:\|f-f_0\|_\infty>M\epsilon_n}\mathrm{E}_f(1-\phi_n)\leq e^{-C_1n\epsilon_n^2}$. However, the proposition below shows that for any test, there exists a function under $H_1$ with Type II error that decreases at least polynomially in $n$. Therefore one cannot achieve $e^{-C_1n\epsilon_n^2}$, or exponential-type decrease in general if the null and alternative are separated apart by (a constant multiple of) $\epsilon_n$ in sup-norm.
\begin{proposition}\label{prop:test}
Let $\epsilon_n=(n/\log{n})^{-\alpha^{*}/(2\alpha^{*}+d)}$. Consider the hypotheses $H_0:f=f_0$ against $H_1:f\in\{\mathcal{B}^{\boldsymbol{\alpha}}_{\infty,\infty}(R):\|f-f_0\|_\infty>M\epsilon_n\}$ with $f_0\in\mathcal{B}^{\boldsymbol{\alpha}}_{\infty,\infty}(R)$ for any $\boldsymbol{\alpha}\in\mathbb{A}$. Let $\phi_n(\boldsymbol{X}_1,\dotsc,\boldsymbol{X}_n;f_0)\rightarrow\{0,1\}$ be any test function for this problem such that $\mathrm{E}_0\phi_n\rightarrow0$, then there exists a constant $Q>0$ such that
\begin{align}\label{eq:test}
\sup_{f\in\mathcal{B}^{\boldsymbol{\alpha}}_{\infty,\infty}(R):\|f-f_0\|_\infty>M\epsilon_n}\mathrm{E}_f(1-\phi_n)\gtrsim n^{-Q}.
\end{align}
\end{proposition}
As the likelihood ratio test is uniformly most powerful, its Type II error lower bounds those from any other tests. Since it is based on the $L_2$-metric due to Gaussian distributed observations, we can always find functions under $H_1$ such that they are far from $f_0$ in $L_\infty$-norm (at least $M\epsilon_n)$, but are close to $f_0$ (within $\sqrt{\log/n}$) when the same distance is measured using the intrinsic $L_2$-norm of the likelihood ratio test. This discrepancy in measured distance caused by these functions is what give rise to polynomially decreasing Type II errors (see Section \ref{sec:testproof} for a complete proof). In \citep{adaptsupnorm}, the authors reached similar conclusions in a different way through minimax theory, by formulating the aforementioned discrepancy into a modulus of continuity relating the intrinsic $L_2$-norm with the desired $L_\infty$-distance.

Polynomial rates are not unique to testing problems with $L_\infty$-separation in the alternative hypothesis. In fact, posterior probabilities on shrinking $L_\infty$ or point-wise $\epsilon_n$-neighborhoods around $f_0$ tend to $0$ at polynomial rates (up to a logarithmic factor), e.g., see Theorem \ref{th:adapt} and Lemmas \ref{lem:lemma1}, \ref{lem:lemma2}. Exponential rates are only possible for weaker losses, with decay of the type $e^{-C_1n\epsilon_n^2}$ corresponding to the $L_2$-loss and its equivalent metrics.

The previous proposition implies that for exponential error tests to exist, the null and the alternative hypotheses must be further separated in $L_\infty$-norm. To that end, let us introduce a separation factor $r_n\to\infty$ as $n\to\infty$ in the alternative $H_1:f\in\{\mathcal{B}^{\boldsymbol{\alpha}}_{\infty,\infty}(R):\|f-f_0\|_\infty>Mr_n\epsilon_n\}$. It is then instructive to ask how large $r_n$ should be so that tests with exponential Type II error start to exist. The following proposition says that $r_n$ must be greater than $\epsilon_n^{-d/(2\alpha^{*}+d)}$, and this increase in separation results in the contraction rate being inflated by the same factor.

\begin{proposition}\label{prop:testing}
For $\epsilon_n=(n/\log{n})^{-\alpha^{*}/(2\alpha^{*}+d)}$, let $\rho_n=\epsilon_n^{-d/(2\alpha^{*}+d)}$ and consider the hypotheses $H_0:f=f_0$ versus $H_1:f\in\{\mathcal{B}^{\boldsymbol{\alpha}}_{\infty,\infty}(R):\|f-f_0\|_\infty>Mr_n\epsilon_n\}$ for any $f_0\in\mathcal{B}^{\boldsymbol{\alpha}}_{\infty,\infty}(R)$ with $\boldsymbol{\alpha}\in\mathbb{A}$. Then for all $r_n=o(\rho_n)$, we have
\begin{align*}
\sup_{f\in\mathcal{B}^{\boldsymbol{\alpha}}_{\infty,\infty}(R):\|f-f_0\|_\infty>Mr_n\epsilon_n}\mathrm{E}_f(1-\phi_n)\gtrsim n^{-Q},
\end{align*}
for any test $\phi_n$ with $\mathrm{E}_0\phi_n\to0$ and some constant $Q>0$. However for $r_n=\rho_n$, there exists a test $\Phi_n$ such that for some constants $C_I,C_{II}>0$,
\begin{align*}
\mathrm{E}_0\Phi_n\leq e^{-C_In\epsilon_n^2},\qquad\sup_{f\in\mathcal{B}^{\boldsymbol{\alpha}}_{\infty,\infty}(R):\|f-f_0\|_\infty>M\rho_n\epsilon_n}\mathrm{E}_f(1-\Phi_n)\leq e^{-C_{II}n\epsilon_n^2}.
\end{align*}
Consequently, if we used the master theorem to prove the first assertion of Theorem \ref{th:adapt}, then there exists a constant $M>0$ such that as $n\rightarrow\infty$,
\begin{align*}
\sup_{\boldsymbol{\alpha}\in\mathbb{A}}\sup_{f_0\in\mathcal{B}^{\boldsymbol{\alpha}}_{\infty,\infty}(R)}\mathrm{E}_0\Pi(\|f-f_0\|_{\infty}>M\epsilon_n^{2\alpha^{*}/(2\alpha^{*}+d)}|\boldsymbol{Y})\rightarrow0.
\end{align*}
\end{proposition}
The proof of this proposition is given in Section \ref{sec:testproof}. Since $\epsilon_n^{2\alpha^{*}/(2\alpha^{*}+d)}\gg\epsilon_n$ and $\epsilon_n$ is the optimal contraction rate given in Theorem \ref{th:adapt}, we incur an extra polynomial factor by utilizing the master theorem. Here we use plug-in test in the form of $\Phi_n=\mathbbm{1}\{\|\widehat{f}_n-f_0\|_\infty\gtrsim\rho_n\epsilon_n\}$ where $\widehat{f}_n$ is the least squares estimator of $f_0$, and exponential Type I and II errors are established using techniques of concentration inequalities introduced by \citep{nickl2011}, with Talagrand's inequality replaced by the Borell's inequality. In \citep{nickl2011} and \citep{adaptsupnorm}, the authors obtained suboptimal rate of the form $\epsilon_n^{1-d/(2\alpha^{*})}$ when generalized to $d$-dimensions, and this is strictly greater than our suboptimal rate $\epsilon_n^{2\alpha^{*}/(2\alpha^{*}+d)}$. The reason for this is that the authors in the aforementioned papers used the truncation point $2^{J_{n,l}(\boldsymbol{\alpha})}=\epsilon_n^{-1/\alpha_l}$ in their calculations while we used a slightly smaller truncation $2^{h_{n,l}(\boldsymbol{\alpha})}=\epsilon_n^{-2\alpha^{*}/\{\alpha_l(2\alpha^{*}+d)\}}$ to construct $\Phi_n$. The former balances the variance and bias when there is no separation factor in the rate, while the latter balances these two quantities at the presence of $\rho_n$, and hence gives a slightly better but albeit suboptimal rate.

Regardless of these suboptimal rates, the preceding discussion and results show the fundamental limitation of the master theorem, and in particular, the method of using tests with exponential errors. In the context of sup-norm posterior contraction, this further suggests that we need alternative methods of proof to get the correct adaptive rates. It is with these thoughts in mind that we developed the technique of reducing regression posterior to its quasi-white noise version, which we will describe in detail in Section \ref{sec:proof2} below. As this technique depends on our ability to isolate coefficients and control the basis Gram matrix, it will also hold for a wider class of basis functions such as Fourier or B-splines and also for other norms such as $L_2$. Moreover, as many models in statistics can be related to the Gaussian white noise model through asymptotic equivalence theory, this new technique should be applicable to more complex problems such as density estimation. We believe that this technique will provide statisticians a powerful tool to prove posterior contraction rates in a simpler setting and translate results back to more realistic models, and hence it would be interesting to explore these extensions in future research.

\section{Proofs}\label{sec:proof}
\subsection{Proof of main results up to Section \ref{sec:lower}}\label{sec:proof1}

\begin{proof}[Proof of Proposition \ref{prop:wavelet}]
An $\boldsymbol{m}=(m_1,\dotsc,m_d)^T$-order tensor product polynomial is a linear combination of $\{x_1^{i_1-1}\cdots x_d^{i_d-1}\}$ for $1\leq i_l\leq m_l,l=1,\dotsc,d$. Recall the wavelet projection operator $K_{\boldsymbol{W}}(f)(\boldsymbol{x})=\int K_{\boldsymbol{W}}(\boldsymbol{x},\boldsymbol{y})f(\boldsymbol{y})d\boldsymbol{y}$ with $K_{\boldsymbol{W}}(\boldsymbol{x},\boldsymbol{y})=\sum_{\boldsymbol{m}}\varphi_{\boldsymbol{W},\boldsymbol{m}}(\boldsymbol{x})\varphi_{\boldsymbol{W},\boldsymbol{m}}(\boldsymbol{y})$. We will be using two important properties of $K_{\boldsymbol{W}}$. The first is
\begin{align}\label{eq:wavelet1}
K_{\boldsymbol{W}}(P)=P,
\end{align}
for any tensor product polynomial $P$ with order less than or equal to $(\eta+1,\dotsc,\eta+1)^T$, the regularity of wavelet used at each dimension (see Theorem 4 of \citep{waveletyves}). The second is
\begin{align}\label{eq:wavelet2}
\|K_{\boldsymbol{W}}(f)\|_p\leq C_1\|f\|_p,
\end{align}
for some constant $C_1>0$ and for any $f\in L_p$. This inequality follows from using the argument discussed in Section 3.1.1 of \citep{nickl2011}. As a result, $K_{\boldsymbol{W}}$ is bounded on $L_p$ and it reproduces polynomials.

Define hypercubes $\mathcal{I}_{\boldsymbol{k}}=\prod_{l=1}^d[k_l2^{-W_l},(k_l+1)2^{-W_l}]$ for $0\leq k_l\leq2^{W_l}-1$ and note that the unit cube $[0,1]^d$ is the sum of these smaller cubes over all $\boldsymbol{k}$. Let $f|_{\mathcal{I}_{\boldsymbol{k}}}$ be the restriction of $f$ onto $\mathcal{I}_{\boldsymbol{k}}$. By Theorem 13.18 of \citep{lschumaker}, we know that there exists a tensor product Taylor's polynomial $p_{\boldsymbol{k}}$ such that
\begin{align*}
\|(g-p_{\boldsymbol{k}})|_{\mathcal{I}_{\boldsymbol{k}}}\|_p\leq C_2\sum_{l=1}^d2^{-\alpha_lW_l}\left\|\frac{\partial^{\alpha_l}}{\partial x_l^{\alpha_l}}g\middle|_{\mathcal{I}_{\boldsymbol{k}}}\right\|_p,
\end{align*}
for some constant $C_2>0$. Then using \eqref{eq:wavelet1}, \eqref{eq:wavelet2} and the triangle inequality,
\begin{align*}
\|\left[K_{\boldsymbol{W}}(g)-g\right]|_{\mathcal{I}_{\boldsymbol{k}}}\|_p&\leq\|(g-p_{\boldsymbol{k}})|_{\mathcal{I}_{\boldsymbol{k}}}\|_p+\|K_{\boldsymbol{W}}(g-p_{\boldsymbol{k}})|_{\mathcal{I}_{\boldsymbol{k}}}\|_p\\
&\lesssim\|(g-p_{\boldsymbol{k}})|_{\mathcal{I}_{\boldsymbol{k}}}\|_p\leq C\sum_{l=1}^d2^{-\alpha_lW_l}\left\|\frac{\partial^{\alpha_l}}{\partial x_l^{\alpha_l}}g\middle|_{\mathcal{I}_{\boldsymbol{k}}}\right\|_p,
\end{align*}
for some constant $C>0$ depending on the wavelets used. The result follows by summing both sides over $0\leq k_l\leq2^{W_l}-1,l=1,\dotsc,d$ and applying Proposition 4.3.8 of \citep{nickl2016}, in view of \eqref{eq:wavelet2}.
\end{proof}

\begin{proof}[Proof of Proposition \ref{prop:consistent}]
The result is a consequence of Lemma \ref{lem:l2contract}, where the posterior of $\sigma$ contracts to $\sigma_0$ at rate $(n/\log{n})^{-\alpha^{*}/(2\alpha^{*}+d)}$.
\end{proof}

For Theorem \ref{th:adapt}, we will only prove the mixed partial derivatives case, as regression function is a special case by setting $\boldsymbol{r}=\boldsymbol{0}$ and interpreting $D^{\boldsymbol{0}}f\equiv f$. The proof is a multivariate generalization of the proof of Theorem 3.1 in \citep{adaptsupnorm}, with some extra new steps to deal with anisotropic smoothness and the fact that we are also considering mixed partial derivatives.

\begin{proof}[Proof of Theorem \ref{th:adapt}]
Since the CDV wavelets are compactly supported and their derivatives are uniformly bounded, it follows that for $\boldsymbol{r}=(r_1,\dotsc,r_d)^T$ with $0\leq r_l<\eta+1,l=1,\dotsc,d$,
\begin{align}\label{eq:regular}
\left\|\sum_{\boldsymbol{m}}|D^{\boldsymbol{r}}\varphi_{\boldsymbol{N},\boldsymbol{m}}|\right\|_\infty=O(1),\quad\left\|\sum_{\boldsymbol{k}}
|D^{\boldsymbol{r}}\psi_{\boldsymbol{j},\boldsymbol{k}}|\right\|_\infty\lesssim\prod_{l=1}^d2^{(1/2+r_l)j_l}.
\end{align}
If $f_0\in\mathcal{B}^{\boldsymbol{\alpha}}_{\infty,\infty}(R)$, then by using the wavelet characterization of \eqref{eq:besov}, it follows that for $\alpha_l<\eta+1,l=1,\dotsc,d$,
\begin{align}\label{eq:infbesov}
\|\boldsymbol{\vartheta}_0\|_\infty\leq R,\qquad\|\boldsymbol{\theta}_{\boldsymbol{j}}^0\|_\infty\leq R2^{-\sum_{l=1}^d\alpha_lj_l\left(\frac{1}{d}+\frac{1}{2\alpha^{*}}\right)},
\end{align}
for any $\boldsymbol{j}$. Note that \eqref{eq:prior} implicitly implies that $\theta_{\boldsymbol{j},\boldsymbol{k}}=0$ when $j_l>J_{n,l}$ for some $l=1,\dotsc,d$. Denote $\mathcal{P}=\{(\boldsymbol{j},\boldsymbol{k}):\theta_{\boldsymbol{j},\boldsymbol{k}}\neq0\}$ the set of nonzero wavelet coefficients. In view of \eqref{eq:infbesov} above, we define for some constant $\gamma>0$ the set
\begin{align}\label{eq:Jngamma}
\mathcal{J}_n(\gamma)&=\left\{(\boldsymbol{j},\boldsymbol{k}):|\theta_{\boldsymbol{j},\boldsymbol{k}}^{0}|>\prod_{l=1}^d\min\left\{2^{-\alpha_lj_l\left(\frac{1}{d}+\frac{1}{2\alpha^{*}}\right)}, \gamma\left(\frac{\log{n}}{n}\right)^{\frac{1}{2d}}\right\}\right\},
\end{align}
and for some constants $0<\underline{\gamma}<\overline{\gamma}<\infty$, the events
\begin{align}\label{eq:abc}
\mathcal{A}&:=\left[\sup_{(\boldsymbol{j},\boldsymbol{k})\in\mathcal{J}_n(\underline{\gamma})}|\theta_{\boldsymbol{j},\boldsymbol{k}}-\theta_{\boldsymbol{j},\boldsymbol{k}}^{0}|\leq\overline{\gamma}\sqrt{\frac{\log{n}}{n}}\right],\nonumber\\
\mathcal{B}&:=[\mathcal{P}\cap\mathcal{J}_n(\underline{\gamma})^c=\emptyset]=\bigcap_{(\boldsymbol{j},\boldsymbol{k})\in\mathcal{J}_n(\underline{\gamma})^c}[\theta_{\boldsymbol{j},\boldsymbol{k}}=0],\nonumber\\
\mathcal{C}&:=[\mathcal{P}^c\cap\mathcal{J}_n(\overline{\gamma})=\emptyset]=\bigcap_{(\boldsymbol{j},\boldsymbol{k})\in\mathcal{J}_n(\overline{\gamma})}[\theta_{\boldsymbol{j},\boldsymbol{k}}\neq0].
\end{align}
As discussed in points $1$ and $2$ in Section \ref{sec:result}, getting the correct sup-norm rate involves showing that $\mathcal{A}$ occurs with posterior probability tending to $1$, and we do not make any signal detection errors as represented by events $\mathcal{B}$ and $\mathcal{C}$.

Let $\mathcal{V}_n$ be a shrinking neighborhood of $\sigma_0$ such that $\mathrm{E}_0\Pi(\sigma\in\mathcal{V}_n|\boldsymbol{Y})\rightarrow1$. Observe that for $\epsilon_{n,\boldsymbol{r}}:=(n/\log{n})^{-\alpha^{*}\{1-\sum_{l=1}^d(r_l/\alpha_l)\}/(2\alpha^{*}+d)}$ and some large enough constant $M>0$ to be specified below, $\mathrm{E}_0\Pi(\|D^{\boldsymbol{r}}f-D^{\boldsymbol{r}}f_0\|_\infty>M\epsilon_{n,\boldsymbol{r}}|\boldsymbol{Y})$ is bounded above by
\begin{align}\label{eq:main}
&\mathrm{E}_0\sup_{\sigma\in\mathcal{V}_n}\Pi(\|D^{\boldsymbol{r}}f-D^{\boldsymbol{r}}f_0\|_\infty>M\epsilon_{n,\boldsymbol{r}}|\boldsymbol{Y},\sigma)
+\mathrm{E}_0\Pi(\sigma\notin\mathcal{V}_n|\boldsymbol{Y})\nonumber\\
&\leq\mathrm{E}_0\sup_{\sigma\in\mathcal{V}_n}\Pi([\|D^{\boldsymbol{r}}f-D^{\boldsymbol{r}}f_0\|_\infty>M\epsilon_{n,\boldsymbol{r}}]\cap\mathcal{A}\cap\mathcal{B}|\boldsymbol{Y},\sigma)\nonumber\\
&\qquad+\mathrm{E}_0\Pi(\sigma\notin\mathcal{V}_n|\boldsymbol{Y})+\mathrm{E}_0\sup_{\sigma\in\mathcal{V}_n}\Pi(\mathcal{B}^c|\boldsymbol{Y},\sigma)\nonumber\\
&\qquad+\mathrm{E}_0\sup_{\sigma\in\mathcal{V}_n}\Pi(\mathcal{C}^c|\boldsymbol{Y},\sigma)
+\mathrm{E}_0\sup_{\sigma\in\mathcal{V}_n}\Pi(\mathcal{A}^c\cap\mathcal{C}|\boldsymbol{Y},\sigma).
\end{align}

By Proposition \ref{prop:consistent}, the second term tends to $0$. By Lemmas \ref{lem:lemma1} and \ref{lem:lemma2} in Section \ref{sec:proof2} below, the last three terms tend to $0$. We then proceed to show that the first term on the right hand side of \eqref{eq:main} approaches $0$ as $n\rightarrow\infty$. Now for any $\boldsymbol{x}\in[0,1]^d$,
\begin{align}\label{eq:coefbound}
|D^{\boldsymbol{r}}f(\boldsymbol{x})-D^{\boldsymbol{r}}f_0(\boldsymbol{x})|&\leq\sum_{m_1=0}^{2^{N_1}-1}\cdots\sum_{m_d=0}^{2^{N_d}-1}|\vartheta_{\boldsymbol{m}}-\vartheta_{\boldsymbol{m}}^0||D^{\boldsymbol{r}}\varphi_{\boldsymbol{N},\boldsymbol{m}}(\boldsymbol{x})|\nonumber\\
&\quad+\sum_{j_1=N_1}^\infty\sum_{k_1=0}^{2^{j_1}-1}\cdots\sum_{j_d=N_d}^\infty\sum_{k_d=0}^{2^{j_d}-1}
|\theta_{\boldsymbol{j},\boldsymbol{k}}-\theta_{\boldsymbol{j},\boldsymbol{k}}^0||D^{\boldsymbol{r}}\psi_{\boldsymbol{j},\boldsymbol{k}}(\boldsymbol{x})|.
\end{align}
Writing $\boldsymbol{\vartheta}=\{\vartheta_{\boldsymbol{m}}:0\leq m_l\leq 2^{N_l}-1,1\leq l\leq d\}$ and using the fact that $\|\boldsymbol{x}\|_\infty\leq\|\boldsymbol{x}\|$ for any $\boldsymbol{x}\in\mathbb{R}^n$, the first sum above is bounded by
\begin{align}\label{eq:sumfirst}
\|\boldsymbol{\vartheta}-\boldsymbol{\vartheta}_0\|_\infty\left\|\sum_{m_1=0}^{2^{N_1}-1}\cdots\sum_{m_d=0}^{2^{N_d}-1}|D^{\boldsymbol{r}}\varphi_{\boldsymbol{N},\boldsymbol{m}}|\right\|_\infty
\lesssim\|\boldsymbol{\vartheta}-\boldsymbol{\vartheta}_0\|\lesssim\epsilon_{n,\boldsymbol{r}},
\end{align}
where the last inequality follows from \eqref{eq:regular} and Corollary \ref{cor:l2contract}.

To bound the second sum, we first choose $J_{n,l}(\boldsymbol{\alpha}),l=1,\dotsc,d$ such that $2^{J_{n,l}(\boldsymbol{\alpha})}\asymp(n/\log{n})^{\alpha^{*}/\{\alpha_l(2\alpha^{*}+d)\}}$. By \eqref{eq:infbesov} with $\boldsymbol{j}=(J_{n,1}(\boldsymbol{\alpha}),\dotsc,J_{n,d}(\boldsymbol{\alpha}))^T$, we have $|\theta_{\boldsymbol{j},\boldsymbol{k}}^0|\leq\|\boldsymbol{\theta}_{\boldsymbol{j}}^0\|_\infty\leq C(\log{n}/n)^{1/2}$ for some constant $C>0$. Therefore, if we choose $\underline{\gamma}$ small enough, we will have  $|\theta_{\boldsymbol{j},\boldsymbol{k}}^0|>\underline{\gamma}(\log{n}/n)^{1/2}$ for $j_l\leq J_{n,l}(\boldsymbol{\alpha})-1,l=1,\dotsc,d$. In other words,
\begin{align}\label{eq:important}
\mathcal{J}_n(\underline{\gamma})\subset\mathcal{I}_n(\boldsymbol{\alpha}):=\{(\boldsymbol{j},\boldsymbol{k}):j_l<J_{n,l}(\boldsymbol{\alpha}),k_l<2^{j_l},l=1,\dotsc,d\},
\end{align}
for sufficiently small $\underline{\gamma}$. Write the sum $\sum_{(\boldsymbol{j},\boldsymbol{k})}$ as an abbreviation of $\sum_{j_1}\sum_{k_1}\cdots\sum_{j_d}\sum_{k_d}$. Using $\mathcal{I}_n(\boldsymbol{\alpha})$ and its complement, the second sum on the right hand side of \eqref{eq:coefbound} is
\begin{align}\label{eq:sumJ}
\left(\sum_{(\boldsymbol{j},\boldsymbol{k})\in\mathcal{I}_n(\boldsymbol{\alpha})}
+\sum_{(\boldsymbol{j},\boldsymbol{k})\in\mathcal{I}_n(\boldsymbol{\alpha})^c}\right)|\theta_{\boldsymbol{j},\boldsymbol{k}}-\theta_{\boldsymbol{j},\boldsymbol{k}}^0||D^{\boldsymbol{r}}\psi_{\boldsymbol{j},\boldsymbol{k}}(\boldsymbol{x})|.
\end{align}
We first bound the second term with summation indices in $\mathcal{I}_n(\boldsymbol{\alpha})^c$, which can be further decomposed as
\begin{align*}
\sum_{(\boldsymbol{j},\boldsymbol{k})\in\mathcal{I}_n(\boldsymbol{\alpha})^c\cap\mathcal{P}^c}|\theta_{\boldsymbol{j},\boldsymbol{k}}^0||D^{\boldsymbol{r}}\psi_{\boldsymbol{j},\boldsymbol{k}}(\boldsymbol{x})|
+\sum_{(\boldsymbol{j},\boldsymbol{k})\in\mathcal{I}_n(\boldsymbol{\alpha})^c\cap\mathcal{P}}|\theta_{\boldsymbol{j},\boldsymbol{k}}-\theta_{\boldsymbol{j},\boldsymbol{k}}^0||D^{\boldsymbol{r}}\psi_{\boldsymbol{j},\boldsymbol{k}}(\boldsymbol{x})|.
\end{align*}
Taking complements on both sides of \eqref{eq:important}, we have $\mathcal{I}_n(\boldsymbol{\alpha})^c\subset\mathcal{J}_n(\underline{\gamma})^c$. Thus, the second sum above is bounded by $\sum_{\mathcal{J}_n(\underline{\gamma})^c\cap\mathcal{P}}|\theta_{\boldsymbol{j},\boldsymbol{k}}-\theta_{\boldsymbol{j},\boldsymbol{k}}^0||D^{\boldsymbol{r}}\psi_{\boldsymbol{j},\boldsymbol{k}}(\boldsymbol{x})|$. This is zero with posterior probability tending to $1$ under event $\mathcal{B}$. In view of \eqref{eq:Jngamma}, the first sum with summation indices in $(\boldsymbol{j},\boldsymbol{k})\in\mathcal{I}_n(\boldsymbol{\alpha})^c\cap\mathcal{P}^c\subset\mathcal{J}_n(\underline{\gamma})^c$ is bounded above by
\begin{align}\label{eq:sumJ2}
\max\{R,\underline{\gamma}\}\sum_{(\boldsymbol{j},\boldsymbol{k})\in\mathcal{I}_n(\boldsymbol{\alpha})^c}\prod_{l=1}^d
U_{j_l,n}|D^{\boldsymbol{r}}\psi_{\boldsymbol{j},\boldsymbol{k}}(\boldsymbol{x})|.
\end{align}
where $U_{j_l,n}:=\min\left\{2^{-\alpha_lj_l\left(\frac{1}{d}+\frac{1}{2\alpha^{*}}\right)},\left(\log{n}/n\right)^{1/(2d)}\right\}$. If we define sets $\mathcal{Q}_l,l=1,\dotsc,d$, where $\mathcal{Q}_l$ can be $\{N_l\leq j_l\leq J_{n,l}(\boldsymbol{\alpha})-1\}$ or $\{j_l\geq J_{n,l}(\boldsymbol{\alpha})\}$, but with the constraint that not all $\mathcal{Q}_l$'s are $\{N_l\leq j_l\leq J_{n,l}(\boldsymbol{\alpha})-1\}$. Then the summation in \eqref{eq:sumJ2} is over $(\boldsymbol{j},\boldsymbol{k})$ such that $\boldsymbol{j}$ takes on all $2^d-1$ possible combinations of the $\mathcal{Q}_l$'s, and each combination has the form
\begin{align*}
&\sum_{j_1\in\mathcal{Q}_1}\sum_{k_1=0}^{2^{j_1}-1}\cdots\sum_{j_d\in\mathcal{Q}_d}\sum_{k_d=0}^{2^{j_d}-1}\prod_{l=1}^d
U_{j_l,n}
|D^{\boldsymbol{r}}\psi_{\boldsymbol{j},\boldsymbol{k}}(\boldsymbol{x})|\nonumber\\
&\qquad\leq\prod_{l=1}^d\sum_{j_l\in\mathcal{Q}_l}U_{j_l,n}
\left\|\sum_{k_1=0}^{2^{j_1}-1}\cdots\sum_{k_d=0}^{2^{j_d}-1}|D^{\boldsymbol{r}}\psi_{\boldsymbol{j},\boldsymbol{k}}|\right\|_\infty
\lesssim\prod_{l=1}^d\sum_{j_l\in\mathcal{Q}_l}2^{(r_l+1/2)j_l}U_{j_l,n}
\end{align*}
where the last inequality follows from \eqref{eq:regular}. The two expressions inside the minimum function of $U_{j_l,n}$ will have the same order if $j_l=J_{n,l}(\boldsymbol{\alpha})$ and $2^{-\alpha_lj_l\left(\frac{1}{d}+\frac{1}{2\alpha^{*}}\right)}$ will have a larger order when $j_l<J_{n,l}(\boldsymbol{\alpha})$, while $(\log{n}/n)^{1/(2d)}$ will dominate if $j_l\geq J_{n,l}(\boldsymbol{\alpha})$. Therefore under the regime $\mathcal{Q}_l=\{N_l\leq j_l\leq J_{n,l}(\boldsymbol{\alpha})-1\}$, we have that
\begin{align*}
&\sum_{N_l\leq j_l<J_{n,l}(\boldsymbol{\alpha})}2^{(r_l+1/2)j_l}\min\left\{2^{-\alpha_lj_l\left(\frac{1}{d}+\frac{1}{2\alpha^{*}}\right)},\left(\frac{\log{n}}{n}\right)^{1/(2d)}\right\}\nonumber\\
&\lesssim2^{(r_l+1/2)J_{n,l}(\boldsymbol{\alpha})}\left(\log{n}/n\right)^{1/(2d)}
\lesssim(n/\log{n})^{-\frac{\alpha^{*}}{2\alpha^{*}+d}\left(\frac{1}{d}+\frac{1}{2\alpha^{*}}-\frac{1}{2\alpha_l}-\frac{r_l}{\alpha_l}\right)};
\end{align*}
while under the regime $\mathcal{Q}_l=\{j_l\geq J_{n,l}(\boldsymbol{\alpha})\}$, we will have
\begin{align*}
&\sum_{j_l\geq J_{n,l}(\boldsymbol{\alpha})}2^{(r_l+1/2)j_l}\min\left\{2^{-\alpha_lj_l\left(\frac{1}{d}+\frac{1}{2\alpha^{*}}\right)},\left(\frac{\log{n}}{n}\right)^{1/(2d)}\right\}\nonumber\\
&\lesssim2^{-\left[\alpha_l\left(\frac{1}{d}+\frac{1}{2\alpha^{*}}\right)-r_l-1/2\right]J_{n,l}(\boldsymbol{\alpha})}
\lesssim(n/\log{n})^{-\frac{\alpha^{*}}{2\alpha^{*}+d}\left(\frac{1}{d}+\frac{1}{2\alpha^{*}}-\frac{1}{2\alpha_l}-\frac{r_l}{\alpha_l}\right)},
\end{align*}
where the first inequality above is justified since $\alpha_l\left(\frac{1}{d}+\frac{1}{2\alpha^{*}}\right)-r_l-1/2>0$ for $l=1,\dotsc,d$ on $\boldsymbol{\alpha}\in\mathbb{A}_{\boldsymbol{r}}$. Putting this bound back and using the fact that there are only $2^d-1$ combinations of the $\mathcal{Q}_l$'s, it then follows that the right hand side of \eqref{eq:sumJ2} is
\begin{align*}
O\left((n/\log{n})^{-\frac{\alpha^{*}}{2\alpha^{*}+d}\sum_{l=1}^d\left(\frac{1}{d}+\frac{1}{2\alpha^{*}}-\frac{1}{2\alpha_l}-\frac{r_l}{\alpha_l}\right)}\right)=O(\epsilon_{n,\boldsymbol{r}}).
\end{align*}
Using a similar decomposition as before, the first sum with summation indices in $\mathcal{I}_n(\boldsymbol{\alpha})=\{(\boldsymbol{j},\boldsymbol{k}):N_l\leq j_l\leq J_{n,l}(\boldsymbol{\alpha})-1,0\leq k_l\leq2^{j_l}-1\}$ of \eqref{eq:sumJ} can be decomposed into 3 parts:
\begin{align*}
&\sum_{(\boldsymbol{j},\boldsymbol{k})\in\mathcal{I}_n(\boldsymbol{\alpha})\cap\mathcal{J}_n(\underline{\gamma})}
+\sum_{(\boldsymbol{j},\boldsymbol{k})\in\mathcal{I}_n(\boldsymbol{\alpha})\cap\mathcal{J}_n(\underline{\gamma})^c\cap\mathcal{P}}
+\sum_{(\boldsymbol{j},\boldsymbol{k})\in\mathcal{I}_n(\boldsymbol{\alpha})\cap\mathcal{J}_n(\underline{\gamma})^c\cap\mathcal{P}^c}.
\intertext{Recalling $\mathcal{J}_n(\underline{\gamma})\subset\mathcal{I}_n(\boldsymbol{\alpha})$ for the first sum above, and the second sum in the decomposition vanishes under event $\mathcal{B}$, the first sum of \eqref{eq:sumJ} reduces to}
&\sum_{(\boldsymbol{j},\boldsymbol{k})\in\mathcal{J}_n(\underline{\gamma})}|\theta_{\boldsymbol{j},\boldsymbol{k}}-\theta_{\boldsymbol{j},\boldsymbol{k}}^0||D^{\boldsymbol{r}}\psi_{\boldsymbol{j},\boldsymbol{k}}(\boldsymbol{x})|
+\sum_{(\boldsymbol{j},\boldsymbol{k})\in\mathcal{I}_n(\boldsymbol{\alpha})\cap\mathcal{J}_n(\underline{\gamma})^c\cap\mathcal{P}^c}|\theta_{\boldsymbol{j},\boldsymbol{k}}^0||D^{\boldsymbol{r}}\psi_{\boldsymbol{j},\boldsymbol{k}}(\boldsymbol{x})|.
\end{align*}
By intersecting with $\mathcal{A}$, the right hand side above is further bounded by
\begin{align*}
&\max_{(\boldsymbol{j},\boldsymbol{k})\in\mathcal{J}_n(\underline{\gamma})}|\theta_{\boldsymbol{j},\boldsymbol{k}}-\theta_{\boldsymbol{j},\boldsymbol{k}}^0|\sum_{j_1=N_1}^{J_{n,1}(\boldsymbol{\alpha})-1}\cdots\sum_{j_d=N_d}^{J_{n,d}(\boldsymbol{\alpha})-1}\sum_{k_1=0}^{2^{j_1}-1}\cdots\sum_{k_d=0}^{2^{j_d}-1}|D^{\boldsymbol{r}}\psi_{\boldsymbol{j},\boldsymbol{k}}(\boldsymbol{x})|\nonumber\\
&\qquad+\sum_{(\boldsymbol{j},\boldsymbol{k})\in\mathcal{I}_n(\boldsymbol{\alpha})\cap\mathcal{J}_n(\underline{\gamma})^c}\underline{\gamma}\sqrt{\frac{\log{n}}{n}}|D^{\boldsymbol{r}}\psi_{\boldsymbol{j},\boldsymbol{k}}(\boldsymbol{x})|\nonumber\\
&\lesssim\prod_{l=1}^d\sum_{N_l\leq j_l<J_{n,l}(\boldsymbol{\alpha})}2^{(1/2+r_l)j_l}\sqrt{\frac{\log{n}}{n}}\lesssim\epsilon_{n,\boldsymbol{r}},
\end{align*}
where the last inequality follows from \eqref{eq:regular}. Now, combining the above with bounds established in \eqref{eq:sumfirst} and \eqref{eq:sumJ2} into \eqref{eq:coefbound}, we conclude that $\|D^{\boldsymbol{r}}f-D^{\boldsymbol{r}}f_0\|_\infty\mathbbm{1}_{\mathcal{A}\cap\mathcal{B}}\leq M\epsilon_{n,\boldsymbol{r}}$ for some sufficiently large constant $M>0$ under the posterior distribution. Using this fact with Lemmas \ref{lem:lemma1} and \ref{lem:lemma2}, it follows that the right hand side of \eqref{eq:main} approaches $0$ as $n\rightarrow\infty$.
\end{proof}
It now remains to show that the last three terms in \eqref{eq:main} approach $0$ asymptotically, and this is detailed in Section \ref{sec:proof2} below.

\subsection{Regression to quasi-white noise posterior}\label{sec:proof2}
In view of \eqref{eq:abc} above, it is clear that we need to bound posterior probabilities of events involving only individual coefficient $\theta_{\boldsymbol{j},\boldsymbol{k}}$. To accomplish this, we bound posterior of $\theta_{\boldsymbol{j},\boldsymbol{k}}$ under the regression model by posterior of $\theta_{\boldsymbol{j},\boldsymbol{k}}$ arising from some quasi-white noise model, where the latter model greatly simplifies calculations through its component-wise structure.

We first define notations. If the rows and columns of a matrix are each indexed by $d$-dimensional multi-indices, we assume that these multi-indices are arranged in the lexicographic order. Let $\boldsymbol{i}=(i_1,\dotsc,i_d)^T$ and $\boldsymbol{j}=(j_1,\dotsc,j_d)^T$. For a matrix $\boldsymbol{A}$ indexed by $2d$-dimensional indices, we write $a_{\boldsymbol{i},\boldsymbol{j}}$ or $\boldsymbol{A}_{\boldsymbol{i},\boldsymbol{j}}$ to be the $(\boldsymbol{i},\boldsymbol{j})$th entry, $\boldsymbol{A}_{\boldsymbol{j},\cdot}$ to be the $\boldsymbol{j}$th row of $\boldsymbol{A}$, $\boldsymbol{A}_{\boldsymbol{j},-\boldsymbol{j}}$ to be the $\boldsymbol{j}$th row of $\boldsymbol{A}$ such that the $\boldsymbol{j}$th entry of that row is excluded, and $\boldsymbol{A}_{-\boldsymbol{j},-\boldsymbol{j}}$ to be a matrix created as a result of deleting the $\boldsymbol{j}$th row and $\boldsymbol{j}$th column of $\boldsymbol{A}$. For a vector $\boldsymbol{x}$, we write $x_{\boldsymbol{j}}$ to be its $\boldsymbol{j}$th component, and $\boldsymbol{x}_{-\boldsymbol{j}}$ be a vector created from $\boldsymbol{x}$ such that its $\boldsymbol{j}$th component is excluded.

Given observations $\{\boldsymbol{X}_1,\dotsc,\boldsymbol{X}_n\}$, we construct the father wavelet matrix $\boldsymbol{B}$ such that its $(h,\boldsymbol{m})$th entry is $\varphi_{\boldsymbol{N},\boldsymbol{m}}(\boldsymbol{X}_h)$, for $1\leq h\leq n$ and $0\leq m_l\leq2^{N_l}-1,l=1,\dotsc,d$. In addition, we define the mother wavelet matrix $\boldsymbol{\Psi}_{\boldsymbol{j}}$ such that its $(h,\boldsymbol{k})$th entry is $\psi_{\boldsymbol{j},\boldsymbol{k}}(\boldsymbol{X}_h)$, for $1\leq h\leq n$ and $0\leq k_l\leq2^{j_l}-1,l=1,\dotsc,d$.

Observe that for $d\times1$ vectors $\boldsymbol{a},\boldsymbol{b},\boldsymbol{c},\boldsymbol{e}$, $\boldsymbol{\Psi}_{\boldsymbol{a}}^T\boldsymbol{\Psi}_{\boldsymbol{b}}$ is a matrix indexed by $2d$-dimensional indices, such that $(\boldsymbol{\Psi}_{\boldsymbol{a}}^T\boldsymbol{\Psi}_{\boldsymbol{b}})_{\boldsymbol{c},\boldsymbol{e}}=\sum_{i=1}^n\psi_{\boldsymbol{a},\boldsymbol{c}}(\boldsymbol{X}_i)\psi_{\boldsymbol{b},\boldsymbol{e}}(\boldsymbol{X}_i)$. Similarly, we also have $(\boldsymbol{\Psi}_{\boldsymbol{a}}^T\boldsymbol{B})_{\boldsymbol{c},\boldsymbol{m}}=\sum_{i=1}^n\psi_{\boldsymbol{a},\boldsymbol{c}}(\boldsymbol{X}_i)\varphi_{\boldsymbol{N},\boldsymbol{m}}(\boldsymbol{X}_i)$. Recall that $\boldsymbol{\vartheta}=\{\vartheta_{\boldsymbol{m}}:0\leq m_l\leq2^{N_l}-1,l=1,\dotsc,d\}$ and define $\boldsymbol{\theta}_{\boldsymbol{j}}=\{\theta_{\boldsymbol{j},\boldsymbol{k}}:0\leq k_l\leq2^{j_l}-1,l=1,\dotsc,d\}$ for a fixed $\boldsymbol{j}$. Let $\boldsymbol{\theta}_{-(\boldsymbol{j},\boldsymbol{k})}=\{\boldsymbol{\theta}:\theta_{\boldsymbol{j},\boldsymbol{k}} \text{ is excluded}\}$, where $\boldsymbol{\theta}=\{\boldsymbol{\theta}_{\boldsymbol{j}}: N_l\leq j_l\leq J_{n,l}-1, l=1,\dotsc,d\}$. We write $\boldsymbol{\varepsilon}=(\varepsilon_1,\dotsc,\varepsilon_n)^T$ and the truncation error $\boldsymbol{\xi}=\boldsymbol{F}_0-\sum_{j_1=N_1}^{J_{n,1}-1}\cdots\sum_{j_d=N_1}^{J_{n,d}-1}\boldsymbol{\Psi}_{\boldsymbol{j}}\boldsymbol{\theta}_{\boldsymbol{j}}^0$, where $\boldsymbol{F}_0=(f_0(\boldsymbol{X}_1),\dotsc,f_0(\boldsymbol{X}_n))^T$.

Write \eqref{eq:model} as $\boldsymbol{Y}=\boldsymbol{B\vartheta}_0+\sum_{j_1=N_1}^{J_{n,1}-1}\cdots\sum_{j_d=N_d}^{J_{n,d}-1}\boldsymbol{\Psi}_{\boldsymbol{j}}\boldsymbol{\theta}_{\boldsymbol{j}}^{0}+\boldsymbol{\xi}+\boldsymbol{\varepsilon}$ under the true distribution $P_0$. Then $\|\boldsymbol{Y}-\boldsymbol{B\vartheta}-\sum_{j_1=N_1}^{J_{n,1}-1}\cdots\sum_{j_d=N_d}^{J_{n,d}-1}\boldsymbol{\Psi}_{\boldsymbol{j}}\boldsymbol{\theta}_{\boldsymbol{j}}\|^2$ is
\begin{align}\label{eq:nozero}
(\theta_{\boldsymbol{j},\boldsymbol{k}}-\theta_{\boldsymbol{j},\boldsymbol{k}}^{0})^2(\boldsymbol{\Psi}_{\boldsymbol{j}}^T\boldsymbol{\Psi}_{\boldsymbol{j}})_{\boldsymbol{k},\boldsymbol{k}}+2(\theta_{\boldsymbol{j},\boldsymbol{k}}-\theta_{\boldsymbol{j},\boldsymbol{k}}^{0})\beta_n(\widetilde{\Theta})
+H_n(\widetilde{\Theta}),
\end{align}
where we have separated the $(\boldsymbol{j},\boldsymbol{k})$th component out from the rest such that
\begin{align*}
\beta_n(\widetilde{\Theta})&:=(\boldsymbol{\Psi}_{\boldsymbol{j}}^T\boldsymbol{\Psi}_{\boldsymbol{j}})_{\boldsymbol{k},-\boldsymbol{k}}(\boldsymbol{\theta}_{\boldsymbol{j}}
-\boldsymbol{\theta}_{\boldsymbol{j}}^0)_{-\boldsymbol{k}}
+\sum_{\boldsymbol{a}\neq\boldsymbol{j}}(\boldsymbol{\Psi}_{\boldsymbol{j}}^T\boldsymbol{\Psi}_{\boldsymbol{a}})_{\boldsymbol{k},\cdot}(\boldsymbol{\theta}_{\boldsymbol{a}}-\boldsymbol{\theta}_{\boldsymbol{a}}^0)\\
&\qquad+(\boldsymbol{\Psi}_{\boldsymbol{j}}^T\boldsymbol{B})_{\boldsymbol{k},\cdot}(\boldsymbol{\vartheta}-\boldsymbol{\vartheta}_0)
-(\boldsymbol{\xi}^T\boldsymbol{\Psi}_{\boldsymbol{j}})_{\boldsymbol{k}}-(\boldsymbol{\varepsilon}^T\boldsymbol{\Psi}_{\boldsymbol{j}})_{\boldsymbol{k}},
\end{align*}
for $\widetilde{\Theta}:=(\boldsymbol{\vartheta},\boldsymbol{\theta}_{-(\boldsymbol{j},\boldsymbol{k})})$ and $H_n(\widetilde{\Theta})$ is
\begin{align*}
&(\boldsymbol{\theta}_{\boldsymbol{j}}-\boldsymbol{\theta}_{\boldsymbol{j}}^{0})_{-\boldsymbol{k}}^T(\boldsymbol{\Psi}_{\boldsymbol{j}}^T\boldsymbol{\Psi}_{\boldsymbol{j}})_{-\boldsymbol{k},-\boldsymbol{k}}
(\boldsymbol{\theta}_{\boldsymbol{j}}-\boldsymbol{\theta}_{\boldsymbol{j}}^{0})_{-\boldsymbol{k}}\\
&\quad+2\sum_{\boldsymbol{a}\neq\boldsymbol{j}}(\boldsymbol{\theta}_{\boldsymbol{j}}-\boldsymbol{\theta}_{\boldsymbol{j}}^{0})_{-\boldsymbol{k}}^T(\boldsymbol{\Psi}_{\boldsymbol{j}}^T
\boldsymbol{\Psi}_{\boldsymbol{a}})_{-\boldsymbol{k},\cdot}(\boldsymbol{\theta}_{\boldsymbol{a}}-\boldsymbol{\theta}_{\boldsymbol{a}}^{0})
+(\boldsymbol{\vartheta}-\boldsymbol{\vartheta}_0)^T\boldsymbol{B}^T\boldsymbol{B}(\boldsymbol{\vartheta}-\boldsymbol{\vartheta}_0)\\
&\quad+\sum_{\boldsymbol{a}\neq\boldsymbol{j}}\sum_{\boldsymbol{b}\neq\boldsymbol{j}}(\boldsymbol{\theta}_{\boldsymbol{a}}-\boldsymbol{\theta}_{\boldsymbol{a}}^{0})^T
\boldsymbol{\Psi}_{\boldsymbol{a}}^T\boldsymbol{\Psi}_{\boldsymbol{b}}(\boldsymbol{\theta}_{\boldsymbol{b}}-\boldsymbol{\theta}_{\boldsymbol{b}}^{0})
-2\boldsymbol{\xi}^T\boldsymbol{B}(\boldsymbol{\vartheta}-\boldsymbol{\vartheta}_0)-2\boldsymbol{\varepsilon}^T\boldsymbol{B}(\boldsymbol{\vartheta}-\boldsymbol{\vartheta}_0)\\
&\quad+2\sum_{\boldsymbol{a}\neq\boldsymbol{j}}(\boldsymbol{\vartheta}-\boldsymbol{\vartheta}_0)^T\boldsymbol{B}^T\boldsymbol{\Psi}_{\boldsymbol{a}}
(\boldsymbol{\theta}_{\boldsymbol{a}}-\boldsymbol{\theta}_{\boldsymbol{a}}^0)
+2(\boldsymbol{\vartheta}-\boldsymbol{\vartheta}_0)^T(\boldsymbol{B}^T\boldsymbol{\Psi}_{\boldsymbol{j}})_{\cdot,-{\boldsymbol{k}}}(\boldsymbol{\theta}_{\boldsymbol{j}}-\boldsymbol{\theta}_{\boldsymbol{j}}^0)_{-{\boldsymbol{k}}}\\
&\quad-2(\boldsymbol{\xi}^T\boldsymbol{\Psi}_{\boldsymbol{j}})_{-{\boldsymbol{k}}}(\boldsymbol{\theta}_{\boldsymbol{j}}-\boldsymbol{\theta}_{\boldsymbol{j}}^{0})_{-{\boldsymbol{k}}}
-2\sum_{\boldsymbol{a}\neq\boldsymbol{j}}\boldsymbol{\xi}^T\boldsymbol{\Psi}_{\boldsymbol{a}}(\boldsymbol{\theta}_{\boldsymbol{a}}-\boldsymbol{\theta}_{\boldsymbol{a}}^{0})\\
&\quad-2(\boldsymbol{\varepsilon}^T\boldsymbol{\Psi}_{\boldsymbol{j}})_{-{\boldsymbol{k}}}(\boldsymbol{\theta}_{\boldsymbol{j}}-\boldsymbol{\theta}_{\boldsymbol{j}}^{0})_{-{\boldsymbol{k}}}
-2\sum_{\boldsymbol{a}\neq\boldsymbol{j}}\boldsymbol{\varepsilon}^T\boldsymbol{\Psi}_{\boldsymbol{a}}(\boldsymbol{\theta}_{\boldsymbol{a}}-\boldsymbol{\theta}_{\boldsymbol{a}}^{0})+\|\boldsymbol{\xi}+\boldsymbol{\varepsilon}\|^2.
\end{align*}
The prior density of $\widetilde{\Theta}$ is $d\Pi(\widetilde{\Theta})=d\Pi(\boldsymbol{\vartheta})d\Pi(\boldsymbol{\theta}_{-(\boldsymbol{j},\boldsymbol{k})})$ where $d\Pi(\boldsymbol{\vartheta})=\prod_{\boldsymbol{m}}p(\vartheta_{\boldsymbol{m}})$ and
\begin{align*}
d\Pi(\boldsymbol{\theta}_{-(\boldsymbol{j},\boldsymbol{k})})=\prod_{(\boldsymbol{x,y})\neq(\boldsymbol{j},\boldsymbol{k})}
[(1-\omega_{\boldsymbol{x},n})d\delta_0(\theta_{\boldsymbol{x,y}})+\omega_{\boldsymbol{x},n}p(\theta_{\boldsymbol{x,y}})d\theta_{\boldsymbol{x,y}}].
\end{align*}
Let $\mathcal{U}$ and $\mathcal{W}$ be two measurable sets on the parameter space of $(\boldsymbol{\vartheta},\boldsymbol{\theta})$, and $\Omega$ be an event on $\boldsymbol{\varepsilon}$ or equivalently on $\boldsymbol{Y}$. Then in view of \eqref{eq:nozero} and by completing the squares, $\Pi(\theta_{\boldsymbol{j},\boldsymbol{k}}\in\mathcal{U},(\boldsymbol{\vartheta},\boldsymbol{\theta}_{-(\boldsymbol{j},\boldsymbol{k})})\in\mathcal{W}|\boldsymbol{Y},\sigma)\mathbbm{1}_{\Omega}$ is
\begin{align*}
&\frac{\int_{\mathcal{W}}\int_{\mathcal{U}} \exp\left\{-\|\boldsymbol{Y}-\boldsymbol{B\vartheta}-\sum_{\boldsymbol{r}}\boldsymbol{\Psi}_{\boldsymbol{r}}\boldsymbol{\theta}_{\boldsymbol{r}}\|^2/(2\sigma^2)\right\}d\Pi(\boldsymbol{\vartheta},\boldsymbol{\theta})}
{\int_{\mathbb{R}^q} \exp\left\{-\|\boldsymbol{Y}-\boldsymbol{B\vartheta}-\sum_{\boldsymbol{r}}\boldsymbol{\Psi}_{\boldsymbol{r}}\boldsymbol{\theta}_{\boldsymbol{r}}\|^2/(2\sigma^2)\right\}d\Pi(\boldsymbol{\vartheta},\boldsymbol{\theta})}\mathbbm{1}_{\Omega}\nonumber\\
&\leq\frac{\int_{\mathcal{W}}I_n(\mathcal{U},\widetilde{\Theta})K_n(\widetilde{\Theta})d\Pi(\widetilde{\Theta})}{\int_{\mathcal{W}}I_n(\mathbb{R},\widetilde{\Theta})K_n(\widetilde{\Theta})d\Pi(\widetilde{\Theta})}\mathbbm{1}_{\Omega},
\end{align*}
where $q$ is the dimension of $(\boldsymbol{\vartheta},\boldsymbol{\theta})$ and
\begin{align}
K_n(\widetilde{\Theta})&:=\exp\left\{\beta_n(\widetilde{\Theta})^2/[2\sigma^2(\boldsymbol{\Psi}_{\boldsymbol{j}}^T\boldsymbol{\Psi}_{\boldsymbol{j}})_{\boldsymbol{k},\boldsymbol{k}}]-H_n(\widetilde{\Theta})/(2\sigma^2)\right\},\nonumber\\ I_n(\mathcal{U},\widetilde{\Theta})&:=\int_{\mathcal{U}}\exp\left\{-\frac{(\boldsymbol{\Psi}_{\boldsymbol{j}}^T\boldsymbol{\Psi}_{\boldsymbol{j}})_{\boldsymbol{k},\boldsymbol{k}}}{2\sigma^2}
\left[\theta-\theta_{\boldsymbol{j},\boldsymbol{k}}^0+\frac{\beta_n(\widetilde{\Theta})}{(\boldsymbol{\Psi}_{\boldsymbol{j}}^T\boldsymbol{\Psi}_{\boldsymbol{j}})_{\boldsymbol{k},\boldsymbol{k}}}\right]^2\right\}d\pi(\theta).\label{eq:pratio}
\end{align}
Therefore, if $\mathcal{W}$ and $\Omega$ are both chosen such that $|\beta_n(\widetilde{\Theta})|$ has sharp upper bound of the correct order uniformly over $\widetilde{\Theta}$, then we can untangle the exponential factor in $I_n(\mathcal{U},\widetilde{\Theta})$ with $K_n(\widetilde{\Theta})$, and the ratio will look like a posterior from a sequence white noise model, with $\int_{\mathcal{W}}K_n(\widetilde{\Theta})d\Pi(\widetilde{\Theta})$ cancelling out each other. Hence, we reduce our posterior to one where we only have to compare the $(\boldsymbol{j},\boldsymbol{k})$th component at the top and bottom, just like in the case when we have quasi-white noise model of the form $Y_{\boldsymbol{j},\boldsymbol{k}}=\theta_{\boldsymbol{j},\boldsymbol{k}}+\sigma(\boldsymbol{\Psi}_{\boldsymbol{j}}^T\boldsymbol{\Psi}_{\boldsymbol{j}})_{\boldsymbol{k},\boldsymbol{k}}^{-1/2}\varepsilon_{\boldsymbol{j},\boldsymbol{k}}$.

The optimal choices of $\mathcal{W}$ and $\Omega$ depend on the statistical problem at hand, the assumed function space for $f_0$, and also implicitly depend on our choice of basis functions through the entries of $\boldsymbol{\Psi}_{\boldsymbol{j}}^T\boldsymbol{\Psi}_{\boldsymbol{j}}$. For orthonormal basis such as the wavelets used in this paper, the diagonal entries are typically of the order $n$ under some conditions on the truncation level $J_{n,l},l=1,\dotsc,d$ (see Lemma \ref{lem:phiphi} below). Let us denote $\widetilde{\beta}_n(\widetilde{\Theta})=\beta_n(\widetilde{\Theta})+(\boldsymbol{\varepsilon}^T\boldsymbol{\Psi}_{\boldsymbol{j}})_{\boldsymbol{k}}$. As we will show below, the appropriate $\mathcal{W}$ for our wavelet regression model is
\begin{align}\label{eq:wn}
\mathcal{W}_n=\{\widetilde{\Theta}:|\widetilde{\beta}_n(\widetilde{\Theta})|\leq \tau_n\sqrt{n\log{n}}\},
\end{align}
with $\tau_n\rightarrow0$ given in \eqref{eq:tau} of Lemma \ref{lem:betabound}; while $\Omega\equiv\Omega_n(c)$ has the form
\begin{align}\label{eq:sigma2bound}
\bigcap_{\substack{N_l\leq j_l\leq J_{n,l}-1\\0\leq k_l\leq2^{j_l}-1\\l=1,\dotsc,d}}\left\{\frac{|\beta_n(\widetilde{\Theta})-\widetilde{\beta}_n(\widetilde{\Theta})|}
{\sqrt{\sigma_0^2(\boldsymbol{\Psi}_{\boldsymbol{j}}^T\boldsymbol{\Psi}_{\boldsymbol{j}})_{\boldsymbol{k},\boldsymbol{k}}}}
\leq\left(2\log{\prod_{l=1}^d2^{j_l}}+c\log{n}\right)^{1/2}\right\}.
\end{align}
The extent that $\mathcal{W}_n$ holds with high posterior probability depends on the choice of the wavelet truncation $2^{\sum_{l=1}^dJ_{n,l}}$, and also on the lower bound on $\alpha^{*}$ that we are able or willing to impose. The lemma below makes this statement explicit.

\begin{lemma}\label{lem:betabound}
Let us take $2^{\sum_{l=1}^dJ_{n,l}}=(n/\log{n})^m$ for some $m\leq1$. Then for any $N_l\leq j_l\leq J_{n,l}-1$ and $0\leq k_l\leq2^{j_l}-1, l=1,\dotsc,d$,
\begin{align*}
\sup_{f_0\in\mathcal{B}^{\boldsymbol{\alpha}}_{\infty,\infty}(R)}\mathrm{E}_0\Pi\left(|\widetilde{\beta}_n(\widetilde{\Theta})|>\tau_n\sqrt{n\log{n}}\middle|\boldsymbol{Y}\right)\leq n^{-P_4},
\end{align*}
for all $\boldsymbol{\alpha}$ such that
\begin{align*}
\alpha^{*}>d\max\left\{\frac{(m-1)+\sqrt{(m-1)^2+8m^2}}{4m},\quad\frac{1}{2}\left(\frac{1}{m}-1\right)\right\},
\end{align*}
with $\tau_n\to0$ given in \eqref{eq:tau} below and $P_4>0$ the same constant as in Lemma \ref{lem:l2contract}. In particular, $m=1/2$ (the default choice of the present paper) minimizes the lower bound on the right hand side to yield $\alpha^{*}>d/2$.
\end{lemma}
\begin{proof}
Let $2^{\sum_{l=1}^dJ_{n,l}}=(n/\log{n})^m$ for some $m\leq1$. We then need to determine the optimal $m$ such that the statement in the lemma is true and the lower limit on $\alpha^{*}$ (if there is any) is as small as possible.

Now by the triangle inequality, $|\widetilde{\beta}_n(\widetilde{\Theta})|$ is bounded above by
\begin{align*}
&|(\boldsymbol{\Psi}_{\boldsymbol{j}}^T\boldsymbol{\Psi}_{\boldsymbol{j}})_{\boldsymbol{k},-\boldsymbol{k}}(\boldsymbol{\theta}_{\boldsymbol{j}}
-\boldsymbol{\theta}_{\boldsymbol{j}}^0)_{-\boldsymbol{k}}|+
\sum_{\boldsymbol{a}\neq\boldsymbol{j}}|(\boldsymbol{\Psi}_{\boldsymbol{j}}^T\boldsymbol{\Psi}_{\boldsymbol{a}})_{\boldsymbol{k},\cdot}(\boldsymbol{\theta}_{\boldsymbol{a}}
-\boldsymbol{\theta}_{\boldsymbol{a}}^0)|\\
&\qquad+|(\boldsymbol{\Psi}_{\boldsymbol{j}}^T\boldsymbol{B})_{\boldsymbol{k},\cdot}(\boldsymbol{\vartheta}-\boldsymbol{\vartheta}_0)|+|(\boldsymbol{\xi}^T\boldsymbol{\Psi}_{\boldsymbol{j}})_{\boldsymbol{k}}|.
\end{align*}
Let $\epsilon_n=(n/\log{n})^{-\alpha^{*}/(2\alpha^{*}+d)}$. By the Cauchy-Schwarz inequality, the first term on the right hand side is bounded by $\|(\boldsymbol{\Psi}_{\boldsymbol{j}}^T\boldsymbol{\Psi}_{\boldsymbol{j}})_{\boldsymbol{k},-\boldsymbol{k}}\|\|\boldsymbol{\theta}-\boldsymbol{\theta}_0\|$. By Lemma \ref{lem:phiphi}, $\|(\boldsymbol{\Psi}_{\boldsymbol{j}}^T\boldsymbol{\Psi}_{\boldsymbol{j}})_{\boldsymbol{k},-\boldsymbol{k}}\|\lesssim\prod_{l=1}^d2^{3j_l/2}$ and hence by Corollary \ref{cor:l2contract}, the first term is $O_P(2^{3\sum_{l=1}^dj_l/2}\epsilon_n)$. Similarly, we can bound the third term using the same lemma and corollary by $\|(\boldsymbol{\Psi}_{\boldsymbol{j}}^T\boldsymbol{B})_{\boldsymbol{k},\cdot}\|\|\boldsymbol{\vartheta}-\boldsymbol{\vartheta}_0\|
\leq C_3\|\boldsymbol{\vartheta}-\boldsymbol{\vartheta}_0\|\prod_{l=1}^d2^{N_l/2}2^{j_l}=O_P(2^{\sum_{l=1}^dj_l}\epsilon_n)$.

To bound the last term, observe that $|(\boldsymbol{\xi}^T\boldsymbol{\Psi}_{\boldsymbol{j}})_k|\leq\|\boldsymbol{\xi}\|_\infty\sum_{i=1}^n|\psi_{\boldsymbol{j},\boldsymbol{k}}(\boldsymbol{X}_i)|$. By Proposition \ref{prop:wavelet}, we have that $\|\boldsymbol{\xi}\|_\infty\lesssim\sum_{l=1}^d2^{-\alpha_lJ_{n,l}}$ for $f_0\in\mathcal{B}^{\boldsymbol{\alpha}}_{\infty,\infty}(R)$. By \eqref{eq:phiphi3} in Lemma \ref{lem:phiphi}, $\sum_{i=1}^n|\psi_{\boldsymbol{j},\boldsymbol{k}}(\boldsymbol{X}_i)|\lesssim n\prod_{l=1}^d2^{-j_l/2}$ since $2^{\sum_{l=1}^dj_l}=o(n)$. Thus we conclude that $|(\boldsymbol{\xi}^T\boldsymbol{\Psi}_{\boldsymbol{j}})_{\boldsymbol{k}}|\lesssim n\prod_{l=1}^d2^{-j_l/2}\sum_{l=1}^d2^{-\alpha_lJ_{n,l}}$.

What remains now is to bound the second term. By another application of the Cauchy-Schwarz inequality, we can bound this term from above by $\|\boldsymbol{\theta}-\boldsymbol{\theta}_0\|\sum_{\boldsymbol{a}\neq\boldsymbol{j}}\|(\boldsymbol{\Psi}_{\boldsymbol{j}}^T\boldsymbol{\Psi}_{\boldsymbol{a}})_{\boldsymbol{k},\cdot}\|$. Then by Lemma \ref{lem:phiphi},
\begin{align*}
\sum_{\boldsymbol{a}\neq\boldsymbol{j}}\|(\boldsymbol{\Psi}_{\boldsymbol{j}}^T\boldsymbol{\Psi}_{\boldsymbol{a}})_{\boldsymbol{k},\cdot}\|&
\lesssim\sum_{\boldsymbol{a}\neq\boldsymbol{j}}\prod_{l=1}^d2^{a_l+j_l/2}.
\end{align*}
Define sets $\mathcal{T}_l,l=1,\dotsc,d$ such that $\mathcal{T}_l$ can be $\{a_l=j_l\}$ or $\{a_l\neq j_l\}$, but with the constraint that not all $\mathcal{T}_l$'s are $\{a_l=j_l\}$. Then the sum on the right hand side above consists of $2^d-1$ terms of the form
\begin{align*}
\sum_{a_1\in\mathcal{T}_1}\cdots\sum_{a_d\in\mathcal{T}_d}\prod_{l=1}^d2^{a_l+j_l/2}=\prod_{l=1}^d\sum_{a_l\in\mathcal{T}_l}2^{a_l+j_l/2}.
\end{align*}
If $\mathcal{T}_l=\{a_l=j_l\}$, then $\sum_{a_l=j_l}2^{a_l+j_l/2}=2^{3j_l/2}$; and if $\mathcal{T}_l=\{a_l\neq j_l\}$, then $\sum_{a_l\neq j_l}^{J_{n,l}}2^{a_l+j_l/2}\lesssim 2^{J_{n,l}+j_l/2}$. It then follows from Corollary \ref{cor:l2contract} that the second term in the upper bound of $|\widetilde{\beta}_n(\widetilde{\Theta})|$ above is $O_P(2^{\sum_{l=1}^d(j_l/2+J_{n,l})}\epsilon_n)$. Combining the bounds obtained, $|\widetilde{\beta}_n(\widetilde{\Theta})|$ is for any $N_l\leq j_l\leq J_{n,l}-1,l=1,\dotsc,d$,
\begin{align*}
O_P\left(2^{\sum_{l=1}^d(j_l/2+J_{n,l})}\epsilon_n\right)+O\left(\frac{n}{\prod_{l=1}^d2^{j_l/2}}\sum_{l=1}^d2^{-\alpha_lJ_{n,l}}\right).
\end{align*}
To optimize this upper bound with respect to $j_l,l=1,\dotsc,d$, we will take $2^{\sum_{l=1}^dj_l}=n2^{-\sum_{l=1}^dJ_{n,l}}\epsilon_n^{-1}\sum_{l=1}^d2^{-\alpha_lJ_{n,l}}$ to balance the two antagonistic terms and conclude that
\begin{align*}
|\widetilde{\beta}_n(\widetilde{\Theta})|=O_P\left(\sqrt{n}2^{\sum_{l=1}^dJ_{n,l}/2}\epsilon_n^{1/2}\sum_{l=1}^d2^{-\alpha_lJ_{n,l}/2}\right).
\end{align*}
To proceed, we set $2^{J_{n,l}}=2^{J/\alpha_l},l=1,\dotsc,d$ for some integer $J$, and it follows that $2^{\sum_{l=1}^dJ_{n,l}}=2^{dJ/\alpha^{*}}$ and $\sum_{l=1}^d2^{-\alpha_lJ_{n,l}/2}=d2^{-J/2}$. Then the above is with posterior probability tending to $1$ bounded above by
\begin{align}\label{eq:tau}
d\sqrt{n}2^{\sum_{l=1}^dJ_{n,l}\left(\frac{1}{2}-\frac{\alpha^{*}}{2d}\right)}\epsilon_n^{1/2}\lesssim\underbrace{n^{-\frac{2m(\alpha^{*})^2-d(m-1)\alpha^{*}-md^2}{2d(2\alpha^{*}+d)}}(\log{n})^{\kappa}}_{\tau_n(m)}\sqrt{n\log{n}},
\end{align}
where $\kappa$ is some constant not depending on $n$. For the lemma's statement to hold, the numerator in the exponent of $n$, i.e., $2m(\alpha^{*})^2-d(m-1)\alpha^{*}-md^2$ must be strictly greater than $0$. At the same time, we know from minimax theory and also from the proof of Theorem \ref{th:adapt} that $2^{J_{n,l}(\boldsymbol{\alpha})}\asymp(n/\log{n})^{\alpha^{*}/\{\alpha_l(2\alpha^{*}+d)\}}$ is the optimal number of wavelets at each $l=1,\dotsc,d$ and hence $2^{\sum_{l=1}^dJ_{n,l}(\boldsymbol{\alpha})}=(n/\log{n})^{d/(2\alpha^{*}+d)}$ must be less than $2^{\sum_{l=1}^dJ_{n,l}}=(n/\log{n})^m$, or equivalently $\alpha^{*}>(d/2)(m^{-1}-1)$. By combining the two lower bound constraints on $\alpha^{*}$, we will have
\begin{align*}
\alpha^{*}>d\max\left\{\frac{(m-1)+\sqrt{(m-1)^2+8m^2}}{4m},\frac{1}{2}\left(\frac{1}{m}-1\right)\right\}.
\end{align*}
Note that the first term inside the max operation is increasing in $m$ while the second term is decreasing. Therefore, the optimal $m$ that minimizes the right hand side can be found by equating these two opposing terms, and the solution to this ``minimax" problem on the right hand side is $m=1/2$ giving the smallest lower bound $\alpha^{*}>d/2$. This smallest lower bound is ensured through letting $\boldsymbol{\alpha}\in\mathbb{A}_{\boldsymbol{r}}\subseteq\mathbb{A}$ for any $\boldsymbol{r}\geq\boldsymbol{0}$ such that $1/d+1/(2\alpha^{*})>1/\alpha_l,l=1,\dotsc,d$ and hence $\alpha^{*}>d/2$ by summing both sides. Therefore, we can take $\tau_n$ to be $\tau_n(1/2)$ in \eqref{eq:tau} and $\tau_n\to0$ by virtue of the established lower bound on $\alpha^{*}$.
\end{proof}

Using the reduction technique discussed and in view of Lemma \ref{lem:betabound} above, we proceed to show that the last three terms in \eqref{eq:main} are negligible under the posterior. We would like to remind the reader the definitions of $\mathcal{W}_n$ in \eqref{eq:wn} and $\Omega_n(c)$ in \eqref{eq:sigma2bound} since they will be used repeatedly in the proofs below.

\begin{lemma}\label{lem:lemma1}
For small enough $\underline{\gamma}$ and large enough $\overline{\gamma}$, there exist constants $P_1,P_2>0$ such that uniformly in $f_0\in\mathcal{B}^{\boldsymbol{\alpha}}_{\infty,\infty}(R)$ with $\boldsymbol{\alpha}\in\mathbb{A}$ and any $0<R\leq R_0-1/2$,
\begin{align}
&\mathrm{E}_0\sup_{\sigma\in\mathcal{V}_n}\Pi(\mathcal{B}^c|\boldsymbol{Y},\sigma)\leq\frac{(\log{n})^d}{n^{P_1}}\label{eq:lemma1first},\\
&\mathrm{E}_0\sup_{\sigma\in\mathcal{V}_n}\Pi(\mathcal{C}^c|\boldsymbol{Y},\sigma)\leq\frac{(\log{n})^d}{n^{P_2}}\label{eq:lemma1second}.
\end{align}
\end{lemma}

\begin{proof}[\textit{Proof of Lemma \ref{lem:lemma1}}]
We first prove \eqref{eq:lemma1first}. By \eqref{eq:abc}, we can write $\mathcal{B}^c=[\mathcal{P}\cap\mathcal{J}_n(\underline{\gamma})^c\neq\emptyset]=\cup_{(\boldsymbol{j},\boldsymbol{k})\in\mathcal{J}_n(\underline{\gamma})^c}[\theta_{\boldsymbol{j},\boldsymbol{k}}\neq0]$.  Recall that $\mathcal{V}_n$ is a shrinking neighborhood of $\sigma_0$, and $\sigma\in\mathcal{V}_n$ implies that $\sigma^2=\sigma^2_0+o(1)$. Using the fact that the posterior probability is bounded by $1$, we have $\mathrm{E}_0\sup_{\sigma\in\mathcal{V}_n}\Pi(\mathcal{P}\cap\mathcal{J}_n(\underline{\gamma})^c\neq\emptyset|\boldsymbol{Y},\sigma)$ is bounded above by
\begin{align}\label{eq:theta0}
\mathrm{E}_0\sup_{\sigma\in\mathcal{V}_n}\sum_{(\boldsymbol{j},\boldsymbol{k})\in\mathcal{J}_n(\underline{\gamma})^c}\Pi(\theta_{\boldsymbol{j},\boldsymbol{k}}\neq0|\boldsymbol{Y},\sigma)\mathbbm{1}_{\Omega_n(\underline{\gamma})}
+P_0[\Omega_n(\underline{\gamma})^c].
\end{align}
To bound the last term, observe that $\widetilde{\beta}_n(\widetilde{\Theta})-\beta_n(\widetilde{\Theta})=(\boldsymbol{\varepsilon}^T\boldsymbol{\Psi}_{\boldsymbol{j}})_k$ is Gaussian under $P_0$, with mean $0$ and variance $\sigma_0^2(\boldsymbol{\Psi}_{\boldsymbol{j}}^T\boldsymbol{\Psi}_{\boldsymbol{j}})_{\boldsymbol{k},\boldsymbol{k}}$. Thus, using the fact that $P(|\sum_i\varepsilon_ia_i|>z)\leq2e^{-z^2/(2\sigma^2\|\boldsymbol{a}\|^2)}$ for $z>0$ and constants $\boldsymbol{a}=(a_1,\dotsc,a_n)^T$ when $\varepsilon_i,i=1,\dotsc,n$ are i.i.d.~Gaussian with mean $0$ and variance $\sigma^2$, we have $P_0[\Omega_n(\underline{\gamma})^c]$ is bounded above by
\begin{align}\label{eq:comega}
&\sum_{j_1=N_1}^{J_{n,1}-1}\sum_{k_1=0}^{2^{j_1}-1}\cdots\sum_{j_d=N_d}^{J_{n,d}-1}\sum_{k_d=0}^{2^{j_d}-1}
P_0\left[\frac{|(\boldsymbol{\varepsilon}^T\boldsymbol{\Psi}_{\boldsymbol{j}})_{\boldsymbol{k}}|}{\sigma_0(\boldsymbol{\Psi}_{\boldsymbol{j}}^T\boldsymbol{\Psi}_{\boldsymbol{j}})^{1/2}_{\boldsymbol{k},\boldsymbol{k}}}
>(2\log{2^{\sum_{l=1}^dj_l}}+\underline{\gamma}\log{n})^{1/2}\right]\nonumber\\
&\qquad\leq2(2^d-1)n^{-\underline{\gamma}/2}\prod_{l=1}^d(J_{n,l}-N_l)\lesssim n^{-\underline{\gamma}/2}(\log{n})^{d}.
\end{align}
The right hand side above approaches $0$ as $n\rightarrow\infty$ for any $\underline{\gamma}>0$. Recall that in \eqref{eq:pratio} above, we defined
\begin{align*}
I_n(\mathcal{U},\widetilde{\Theta}):=\int_{\mathcal{U}}\exp\left\{-\frac{(\boldsymbol{\Psi}_{\boldsymbol{j}}^T\boldsymbol{\Psi}_{\boldsymbol{j}})_{\boldsymbol{k},\boldsymbol{k}}}{2\sigma^2}
\left[\theta-\theta_{\boldsymbol{j},\boldsymbol{k}}^0+\frac{\beta_n(\widetilde{\Theta})}{(\boldsymbol{\Psi}_{\boldsymbol{j}}^T\boldsymbol{\Psi}_{\boldsymbol{j}})_{\boldsymbol{k},\boldsymbol{k}}}\right]^2\right\}d\pi(\theta).
\end{align*}
To bound the first term in \eqref{eq:theta0}, observe that for $(\boldsymbol{j},\boldsymbol{k})\in\mathcal{J}_n(\underline{\gamma})^c$, we can upper bound $\Pi(\theta_{\boldsymbol{j},\boldsymbol{k}}\neq0|\boldsymbol{Y},\sigma)\mathbbm{1}_{\Omega_n(\underline{\gamma})}$ by
\begin{align}\label{eq:QU}
&\Pi(\theta_{\boldsymbol{j},\boldsymbol{k}}\neq0,(\boldsymbol{\vartheta},\boldsymbol{\theta}_{-(\boldsymbol{j},\boldsymbol{k})})\in\mathcal{W}_n|\boldsymbol{Y},\sigma)\mathbbm{1}_{\Omega_n(\underline{\gamma})}
+\Pi(\mathcal{W}_n^c|\boldsymbol{Y},\sigma)\nonumber\\
&\qquad\leq\frac{\int_{\mathcal{W}_n}I_n([\theta_{\boldsymbol{j},\boldsymbol{k}}\neq0],\widetilde{\Theta})K_n(\widetilde{\Theta})d\Pi(\widetilde{\Theta})}{\int_{\mathcal{W}_n}I_n([\theta_{\boldsymbol{j},\boldsymbol{k}}=0],\widetilde{\Theta})K_n(\widetilde{\Theta})d\Pi(\widetilde{\Theta})}\mathbbm{1}_{\Omega_n(\underline{\gamma})}+O_{P_0}\left(n^{-B}\right),
\end{align}
where $O_{P_0}\left(n^{-B}\right)$ for some constant $B>0$ follows from Lemma \ref{lem:betabound}. When $\theta_{\boldsymbol{j},\boldsymbol{k}}\neq0$, then $d\pi(\theta_{\boldsymbol{j},\boldsymbol{k}})=\omega_{\boldsymbol{j},n}p(\theta_{\boldsymbol{j},\boldsymbol{k}})d\theta_{\boldsymbol{j},\boldsymbol{k}}$. Since $p_{\mathrm{max}}=\sup_{x\in\mathbb{R}}p(x)<\infty$ by assumption, we can upper bound $I_n([\theta_{\boldsymbol{j},\boldsymbol{k}}\neq0],\widetilde{\Theta})$ by $\omega_{\boldsymbol{j},n}p_{\mathrm{max}}$ times
\begin{align*}
\int_{-\infty}^\infty\exp\left\{-\frac{(\boldsymbol{\Psi}_{\boldsymbol{j}}^T\boldsymbol{\Psi}_{\boldsymbol{j}})_{\boldsymbol{k},\boldsymbol{k}}}{2\sigma^2}\left[x+\frac{\beta_n(\widetilde{\Theta})}{(\boldsymbol{\Psi}_{\boldsymbol{j}}^T\boldsymbol{\Psi}_{\boldsymbol{j}})_{\boldsymbol{k},\boldsymbol{k}}}\right]^2\right\}dx
=\frac{\sqrt{2\pi\sigma^2}}{\sqrt{(\boldsymbol{\Psi}_{\boldsymbol{j}}^T\boldsymbol{\Psi}_{\boldsymbol{j}})_{\boldsymbol{k},\boldsymbol{k}}}}.
\end{align*}
Therefore, in view of the lower bound in Lemma \ref{lem:phiphi},
\begin{align}\label{eq:Ujk}
\int_{\mathcal{W}_n}I_n([\theta_{\boldsymbol{j},\boldsymbol{k}}\neq0],\widetilde{\Theta})K_n(\widetilde{\Theta})d\Pi(\widetilde{\Theta})\lesssim
\omega_{\boldsymbol{j},n}n^{-1/2}\sigma\int_{\mathcal{W}_n}
K_n(\widetilde{\Theta})d\Pi(\widetilde{\Theta}).
\end{align}
Now, $\int_{\mathcal{W}_n}I_n([\theta_{\boldsymbol{j},\boldsymbol{k}}=0],\widetilde{\Theta})K_n(\widetilde{\Theta})d\Pi(\widetilde{\Theta})$ is $1-\omega_{\boldsymbol{j},n}$ times
\begin{align}\label{eq:Qjk}
&\int_{\mathcal{W}_n}\exp\left\{-\frac{(\boldsymbol{\Psi}_{\boldsymbol{j}}^T\boldsymbol{\Psi}_{\boldsymbol{j}})_{\boldsymbol{k},\boldsymbol{k}}}{2\sigma^2}\left[\theta_{\boldsymbol{j},\boldsymbol{k}}^0-\frac{\beta_n(\widetilde{\Theta})}{(\boldsymbol{\Psi}_{\boldsymbol{j}}^T\boldsymbol{\Psi}_{\boldsymbol{j}})_{\boldsymbol{k},\boldsymbol{k}}}\right]^2\right\}
K_n(\widetilde{\Theta})d\Pi(\widetilde{\Theta}).
\end{align}
To lower bound the expression above, we proceed by lower bounding the first exponential factor. By definition, we have $|\theta_{\boldsymbol{j},\boldsymbol{k}}^0|\leq\underline{\gamma}\sqrt{\log{n}/n}$ for $(\boldsymbol{j},\boldsymbol{k})\in\mathcal{J}_n(\underline{\gamma})^c$. Under $\mathcal{W}_n$ and $\Omega_n(\underline{\gamma})$, we can use the triangle inequality to bound $\sqrt{(\boldsymbol{\Psi}_{\boldsymbol{j}}^T\boldsymbol{\Psi}_{\boldsymbol{j}})_{\boldsymbol{k},\boldsymbol{k}}}\left|\theta_{\boldsymbol{j},\boldsymbol{k}}^0-\frac{\beta_n(\widetilde{\Theta})}{(\boldsymbol{\Psi}_{\boldsymbol{j}}^T\boldsymbol{\Psi}_{\boldsymbol{j}})_{\boldsymbol{k},\boldsymbol{k}}}\right|$ from above by
\begin{align*}
&\sqrt{(\boldsymbol{\Psi}_{\boldsymbol{j}}^T\boldsymbol{\Psi}_{\boldsymbol{j}})_{\boldsymbol{k},\boldsymbol{k}}}|\theta_{\boldsymbol{j},\boldsymbol{k}}^0|+\frac{|\widetilde{\beta}_n(\widetilde{\Theta})|}{\sqrt{(\boldsymbol{\Psi}_{\boldsymbol{j}}^T\boldsymbol{\Psi}_{\boldsymbol{j}})_{\boldsymbol{k},\boldsymbol{k}}}}
+\frac{|\widetilde{\beta}_n(\widetilde{\Theta})-\beta_n(\widetilde{\Theta})|}{\sqrt{(\boldsymbol{\Psi}_{\boldsymbol{j}}^T\boldsymbol{\Psi}_{\boldsymbol{j}})_{\boldsymbol{k},\boldsymbol{k}}}}\\
&\qquad\leq(\sqrt{C_2}\underline{\gamma}+\tau_n/\sqrt{C_1})\sqrt{\log{n}}+\sigma_0\sqrt{2\log{2^{\sum_{l=1}^dj_l}}+\underline{\gamma}\log{n}}\\
&\qquad\leq2\sqrt{C_2}\underline{\gamma}\sqrt{\log{n}}+\sigma_0(2\log{2^{\sum_{l=1}^dj_l}}+\underline{\gamma}\log{n})^{1/2},
\end{align*}
for large enough $n$ because $\tau_n\rightarrow0$ as $n\rightarrow\infty$ (from \eqref{eq:tau}). Note also that by Lemma \ref{lem:phiphi}, $\sqrt{C_1n}\leq(\boldsymbol{\Psi}_{\boldsymbol{j}}^T\boldsymbol{\Psi}_{\boldsymbol{j}})_{\boldsymbol{k},\boldsymbol{k}}^{1/2}\leq\sqrt{C_2n}$ for some constants $C_1,C_2>0$, because $2^{\sum_{l=1}^dj_l}\leq2^{\sum_{l=1}^dJ_{n,l}}=\sqrt{n/\log{n}}=o(n)$ by assumption. By squaring both sides, $(\boldsymbol{\Psi}_{\boldsymbol{j}}^T\boldsymbol{\Psi}_{\boldsymbol{j}})_{\boldsymbol{k},\boldsymbol{k}}\left[\theta_{\boldsymbol{j},\boldsymbol{k}}^0-\frac{\beta_n(\widetilde{\Theta})}{(\boldsymbol{\Psi}_{\boldsymbol{j}}^T\boldsymbol{\Psi}_{\boldsymbol{j}})_{\boldsymbol{k},\boldsymbol{k}}}\right]^2$ is bounded above by
\begin{align*}
&4C_2\underline{\gamma}^2\log{n}+\sigma_0^2(2\log{2^{\sum_{l=1}^dj_l}}+\underline{\gamma}\log{n})\\
&\qquad+4\sigma_0\underline{\gamma}\sqrt{C_2\log{n}
(2\log{2^{\sum_{l=1}^dj_l}}+\underline{\gamma}\log{n})}\leq\kappa(\underline{\gamma})\log{n}+2\sigma_0^2\log{2^{\sum_{l=1}^dj_l}},
\end{align*}
where $\kappa(\underline{\gamma})=4C_2\underline{\gamma}^2+4\sigma_0\sqrt{C_2}\underline{\gamma}^{3/2}+(4\sigma_0\sqrt{2C_2}+\sigma_0^2)\underline{\gamma}$, and the last inequality follows from $\sqrt{a+b}\leq\sqrt{a}+\sqrt{b}$. Thus, $\int_{\mathcal{W}_n}I_n([\theta_{\boldsymbol{j},\boldsymbol{k}}=0],\widetilde{\Theta})K_n(\widetilde{\Theta})d\Pi(\widetilde{\Theta})$ is bounded below by
\begin{align}\label{eq:lowerzero}
(1-\omega_{\boldsymbol{j},n})\exp{\left(-\kappa(\underline{\gamma})\frac{\log{n}}{2\sigma^2}-\frac{\sigma_0^2}{\sigma^2}\log{2^{\sum_{l=1}^dj_l}}\right)}
\int_{\mathcal{W}_n}K_n(\widetilde{\Theta})d\Pi(\widetilde{\Theta}).
\end{align}
By assumption, $\omega_{\boldsymbol{j},n}\leq\min\{2^{-\sum_{l=1}^dj_l(1+\mu_l)},1/2\}$ and $\mu_{\mathrm{min}}:=\min_{1\leq l\leq d}\mu_l>1/2$. Using the fact that $x/(1-x)\leq2x$ for $0\leq x\leq0.5$ with the upper and lower bounds of \eqref{eq:Ujk} and \eqref{eq:lowerzero}, we can upper bound the first term of \eqref{eq:theta0} up to some constant multiple by $n^{\frac{1}{2\sigma^2}\kappa(\underline{\gamma})-\frac{1}{2}}\sigma$ times
\begin{align*}
\sum_{j_1=N_1}^{J_{n,1}}\sum_{k_1=0}^{2^{j_1}-1}\cdots\sum_{j_d=N_d}^{J_{n,d}}\sum_{k_d=0}^{2^{j_d}-1}
2^{\frac{\sigma_0^2}{\sigma^2}\sum_{l=1}^dj_l}2\omega_{\boldsymbol{j},n}\lesssim n^{\frac{1}{2\sigma^2}\kappa(\underline{\gamma})-\frac{1}{2}}\sigma\prod_{l=1}^d\sum_{j_l=N_l}^{J_{n,l}}2^{j_l\left(\frac{\sigma_0^2}{\sigma^2}-\mu_l\right)}.
\end{align*}
Now if $\sigma_0^2/\sigma^2>\mu_l$, we have $\sum_{j_l=N_l}^{J_{n,l}}2^{j_l(\sigma_0^2/\sigma^2-\mu_l)}\lesssim2^{J_{n,l}(\sigma_0^2/\sigma^2-\mu_l)}$;  while for $\sigma_0^2/\sigma^2\leq\mu_l$, this sum is $O(1)$. Therefore, if $\sigma_0^2/\sigma^2\leq\mu_l$ for all $l=1,\dotsc,d$, the right hand side above is $O(n^{\kappa(\underline{\gamma})/[2\sigma_0^2+o(1)]-1/2})$ after uniformizing over $\sigma\in\mathcal{V}_n$, and it will tend to $0$ if $\underline{\gamma}$ is small enough. On the other hand, if $\sigma_0^2/\sigma^2>\mu_l$ for at least one $l=1,\dotsc,d$, then the right hand side is bounded above up to some constant by
\begin{align}\label{eq:ratioboundmu}
n^{\frac{1}{2\sigma^2}\kappa(\underline{\gamma})-1/2}\sigma\prod_{l=1}^d2^{J_{n,l}(\sigma_0^2/\sigma^2-\mu_l)}
\lesssim n^{-\mu_{\mathrm{min}}-1/2+\frac{\sigma_0^2}{\sigma^2}+\frac{1}{2\sigma^2}\kappa(\underline{\gamma})}\sigma.
\end{align}
By uniformizing over $\sigma\in\mathcal{V}_n$, the right hand side above is $[\sigma_0+o(1)]n^{1/2-\mu_{\mathrm{min}}+o(1)+\kappa(\underline{\gamma})/[2\sigma_0^2+o(1)]}$ and it will approach $0$ as $n\rightarrow\infty$ if $\underline{\gamma}$ is chosen small enough, since $\mu_{\mathrm{min}}>1/2$ by our prior assumption.

We now prove the second assertion \eqref{eq:lemma1second}. By \eqref{eq:important}, $(\boldsymbol{j},\boldsymbol{k})\in\mathcal{J}_n(\overline{\gamma})\subset\mathcal{J}_n(\underline{\gamma})\subset\mathcal{I}_n(\boldsymbol{\alpha})$, then $2^{-\alpha_lj_l[d^{-1}+(2\alpha^{*})^{-1}]}$ dominates $\overline{\gamma}(\log{n}/n)^{1/(2d)}$ inside the minimum function of \eqref{eq:Jngamma} since $j_l\leq J_{n,l}(\boldsymbol{\alpha})-1,l=1,\dotsc,d$. Therefore, the definition of event $\mathcal{C}$ in \eqref{eq:abc} can be reduced to
\begin{align}\label{eq:c}
\mathcal{C}:=\bigcap_{\{|\theta_{\boldsymbol{j},\boldsymbol{k}}^0|>\overline{\gamma}\sqrt{\log{n}/n}\}}[\theta_{\boldsymbol{j},\boldsymbol{k}}\neq0],
\end{align}
for large enough $\overline{\gamma}>0$. Taking complements and using the same decomposition as in \eqref{eq:theta0}, $\mathrm{E}_0\sup_{\sigma\in\mathcal{V}_n}\Pi(\mathcal{P}^c\cap\mathcal{J}_n(\overline{\gamma})\neq\emptyset|\boldsymbol{Y},\sigma)$ is bounded above by
\begin{align}\label{eq:theta}
&\mathrm{E}_0\sup_{\sigma\in\mathcal{V}_n}\sum_{\{|\theta_{\boldsymbol{j},\boldsymbol{k}}^0|>\overline{\gamma}\sqrt{\log{n}/n}\}}\Pi(\theta_{\boldsymbol{j},\boldsymbol{k}}=0|\boldsymbol{Y},\sigma)\mathbbm{1}_{\Omega_n(1)}
+P_0[\Omega_n(1)^c].
\end{align}
Using the same argument leading to \eqref{eq:comega} by substituting $\underline{\gamma}$ with $1$,
\begin{align}\label{eq:omega}
P_0[\Omega_n(1)^c]\lesssim n^{-1/2}(\log{n})^d\rightarrow0
\end{align}
as $n\rightarrow\infty$. To bound the first term, note that in the present case $\Pi(\theta_{\boldsymbol{j},\boldsymbol{k}}=0|\boldsymbol{Y},\sigma)\mathbbm{1}_{\Omega_n(1)}$ is bounded above by
\begin{align*}
\frac{\int_{\mathcal{W}_n}I_n([\theta_{\boldsymbol{j},\boldsymbol{k}}=0],\widetilde{\Theta})K_n(\widetilde{\Theta})d\Pi(\widetilde{\Theta})}{\int_{\mathcal{W}_n}I_n([\theta_{\boldsymbol{j},\boldsymbol{k}}\neq0],\widetilde{\Theta})K_n(\widetilde{\Theta})d\Pi(\widetilde{\Theta})}\mathbbm{1}_{\Omega_n(1)}+
\Pi(\mathcal{W}_n^c|\boldsymbol{Y},\sigma),
\end{align*}
and the second term above is $O_{P_0}\left(n^{-B}\right)$ for some constant $B>0$ by Lemma \ref{lem:betabound}. To upper bound $\int_{\mathcal{W}_n}I_n([\theta_{\boldsymbol{j},\boldsymbol{k}}=0],\widetilde{\Theta})K_n(\widetilde{\Theta})d\Pi(\widetilde{\Theta})$, we need to upper bound the first exponential factor in \eqref{eq:Qjk}. Now, $(\boldsymbol{\Psi}_{\boldsymbol{j}}^T\boldsymbol{\Psi}_{\boldsymbol{j}})_{\boldsymbol{k},\boldsymbol{k}}^{1/2}\geq\sqrt{C_1n}$ by Lemma \ref{lem:phiphi} since $2^{\sum_{l=1}^dj_l}=o(n)$ for $j_l<J_{n,l}(\boldsymbol{\alpha}),l=1,\dotsc,d$. Applying the reverse triangular inequality twice and under $\mathcal{W}_n$ and $\Omega_n(1)$, we have for any $|\theta_{\boldsymbol{j},\boldsymbol{k}}^0|>\overline{\gamma}\sqrt{\log{n}/n}$ that $\sqrt{(\boldsymbol{\Psi}_{\boldsymbol{j}}^T\boldsymbol{\Psi}_{\boldsymbol{j}})_{\boldsymbol{k},\boldsymbol{k}}}\left|\theta_{\boldsymbol{j},\boldsymbol{k}}^0-\frac{\beta_n(\widetilde{\Theta})}{(\boldsymbol{\Psi}_{\boldsymbol{j}}^T\boldsymbol{\Psi}_{\boldsymbol{j}})_{\boldsymbol{k},\boldsymbol{k}}}\right|$ is lower bounded by
\begin{align*}
&\geq\sqrt{(\boldsymbol{\Psi}_{\boldsymbol{j}}^T\boldsymbol{\Psi}_{\boldsymbol{j}})_{\boldsymbol{k},\boldsymbol{k}}}|\theta_{\boldsymbol{j},\boldsymbol{k}}^0|-\frac{|\widetilde{\beta}_n(\widetilde{\Theta})|}{\sqrt{(\boldsymbol{\Psi}_{\boldsymbol{j}}^T\boldsymbol{\Psi}_{\boldsymbol{j}})_{\boldsymbol{k},\boldsymbol{k}}}}
-\frac{|\widetilde{\beta}_n(\widetilde{\Theta})-\beta_n(\widetilde{\Theta})|}{\sqrt{(\boldsymbol{\Psi}_{\boldsymbol{j}}^T\boldsymbol{\Psi}_{\boldsymbol{j}})_{\boldsymbol{k},\boldsymbol{k}}}}\\
&\geq\sqrt{C_1}\overline{\gamma}\sqrt{\log{n}}-(\tau_n/\sqrt{C_1})\sqrt{\log{n}}-\sigma_0\left(2\log{2^{\sum_{l=1}^dj_l}}+\log{n}\right)^{1/2},
\end{align*}
which is greater than $0.5\sqrt{C_1}\overline{\gamma}\sqrt{\log{n}}$ if $\overline{\gamma}$ is large enough since $\tau_n\rightarrow0$ as $n\rightarrow\infty$. Therefore $\int_{\mathcal{W}_n}I_n([\theta_{\boldsymbol{j},\boldsymbol{k}}=0],\widetilde{\Theta})K_n(\widetilde{\Theta})d\Pi(\widetilde{\Theta})$ is bounded above by
\begin{align*}
(1-\omega_{\boldsymbol{j},n})\exp\left(-\frac{C_1\overline{\gamma}^2}{8\sigma^2}\log{n}\right)\int_{\mathcal{W}_n}
K_n(\widetilde{\Theta})d\Pi(\widetilde{\Theta}).
\end{align*}
By \eqref{eq:density}, $p(x)\geq p_{\mathrm{min}}>0$ for $|x|\leq R_0$. Thus, we bound $I_n([\theta_{\boldsymbol{j},\boldsymbol{k}}\neq0],\widetilde{\Theta})$ from below by
\begin{align*}
&=p_{\mathrm{min}}
\int_{-R_0-\theta_{\boldsymbol{j},\boldsymbol{k}}^0}^{R_0-\theta_{\boldsymbol{j},\boldsymbol{k}}^0}\exp\left\{-\frac{(\boldsymbol{\Psi}_{\boldsymbol{j}}^T\boldsymbol{\Psi}_{\boldsymbol{j}})_{\boldsymbol{k},\boldsymbol{k}}}{2\sigma^2}\left[x+\frac{\beta_n(\widetilde{\Theta})}{(\boldsymbol{\Psi}_{\boldsymbol{j}}^T\boldsymbol{\Psi}_{\boldsymbol{j}})_{\boldsymbol{k},\boldsymbol{k}}}\right]^2\right\}dx\\
&\geq \frac{p_{\mathrm{min}}\sqrt{2\pi\sigma^2}}{\sqrt{(\boldsymbol{\Psi}_{\boldsymbol{j}}^T\boldsymbol{\Psi}_{\boldsymbol{j}})_{\boldsymbol{k},\boldsymbol{k}}}}\left[2\Phi\left\{\frac{\sqrt{(\boldsymbol{\Psi}_{\boldsymbol{j}}^T\boldsymbol{\Psi}_{\boldsymbol{j}})_{\boldsymbol{k},\boldsymbol{k}}}}
{\sigma}\left(R_0-\left|\theta_{\boldsymbol{j},\boldsymbol{k}}^0-\frac{\beta_n(\widetilde{\Theta})}{(\boldsymbol{\Psi}_{\boldsymbol{j}}^T\boldsymbol{\Psi}_{\boldsymbol{j}})_{\boldsymbol{k},\boldsymbol{k}}}\right|\right)\right\}-1\right]
\end{align*}
where $\Phi$ is the cumulative distribution function of a standard normal. We proceed by lower bounding the expression inside $\Phi$. By the triangle inequality,
\begin{align}\label{eq:phi}
\left|\theta_{\boldsymbol{j},\boldsymbol{k}}^0-\frac{\beta_n(\widetilde{\Theta})}{(\boldsymbol{\Psi}_{\boldsymbol{j}}^T\boldsymbol{\Psi}_{\boldsymbol{j}})_{\boldsymbol{k},\boldsymbol{k}}}\right|\leq|\theta_{\boldsymbol{j},\boldsymbol{k}}^0|
+\frac{|\beta_n(\widetilde{\Theta})-\widetilde{\beta}_n(\widetilde{\Theta})|}{(\boldsymbol{\Psi}_{\boldsymbol{j}}^T\boldsymbol{\Psi}_{\boldsymbol{j}})_{\boldsymbol{k},\boldsymbol{k}}}
+\frac{|\widetilde{\beta}_n(\widetilde{\Theta})|}{(\boldsymbol{\Psi}_{\boldsymbol{j}}^T\boldsymbol{\Psi}_{\boldsymbol{j}})_{\boldsymbol{k},\boldsymbol{k}}}.
\end{align}
By \eqref{eq:infbesov}, we have $|\theta_{\boldsymbol{j},\boldsymbol{k}}^0|\leq R$. On the event $\Omega_n(1)$, the second term above is $O_{P_0}(\sqrt{\log{n}/n})=o_{P_0}(1)$ since $(\boldsymbol{\Psi}_{\boldsymbol{j}}^T\boldsymbol{\Psi}_{\boldsymbol{j}})_{\boldsymbol{k},\boldsymbol{k}}\geq C_1n$ for any $(\boldsymbol{j},\boldsymbol{k})\in\mathcal{J}_n(\overline{\gamma})$ by Lemma \ref{lem:phiphi}. Under $\mathcal{W}_n$ and applying Lemma \ref{lem:phiphi} again, the third term above is $o(\sqrt{\log{n}/n})=o(1)$. Then under the assumption $R\leq R_0-1/2$, the right hand side of \eqref{eq:phi} is bounded above by $R+\frac{1}{4}\leq R_0-\frac{1}{4}$ for $n$ large enough. Hence, another application of Lemma \ref{lem:phiphi} yields
\begin{align*}
I_n([\theta_{\boldsymbol{j},\boldsymbol{k}}\neq0],\widetilde{\Theta})&\geq p_{\mathrm{min}}\sqrt{\frac{2\pi\sigma^2}{(\boldsymbol{\Psi}_{\boldsymbol{j}}^T\boldsymbol{\Psi}_{\boldsymbol{j}})_{\boldsymbol{k},\boldsymbol{k}}}}
\left[2\Phi\left(\frac{\sqrt{C_1n}}{4\sigma}\right)-1\right].
\end{align*}
Using the fact that $P(|Z|>z)\leq2e^{-z^2/2}$ for $z>0$ and $Z\sim\mathrm{N}(0,1)$, we will obtain for any $\sigma\in\mathcal{V}_n$ and for $n$ large enough,
\begin{align*}
2\Phi\left(\frac{\sqrt{C_1n}}{4\sigma}\right)-1=1-P\left(|Z|>\frac{\sqrt{C_1n}}{4\sigma}\right)\geq1-2e^{-C_1n/(32\sigma^2)}\geq1/\sqrt{2}.
\end{align*}
Consequently in view of Lemma \ref{lem:phiphi}, we have for large enough $n$,
\begin{align}\label{eq:lowerboundU}
\int_{\mathcal{W}_n}I_n([\theta_{\boldsymbol{j},\boldsymbol{k}}\neq0],\widetilde{\Theta})K_n(\widetilde{\Theta})d\Pi(\widetilde{\Theta})\gtrsim\omega_{\boldsymbol{j},n}n^{-1/2}\sigma\int_{\mathcal{W}_n}
K_n(\widetilde{\Theta})d\Pi(\widetilde{\Theta}).
\end{align}
By assumption, $\omega_{\boldsymbol{j},n}\geq n^{-\lambda}$ for some constant $\lambda>0$. Thus, the first term on the right hand side of \eqref{eq:theta} is bounded above up to a constant by
\begin{align*}
\sup_{\sigma\in\mathcal{V}_n}\sum_{\{|\theta_{\boldsymbol{j},\boldsymbol{k}}^0|>\overline{\gamma}\sqrt{\log{n}/n}\}}\frac{1-\omega_{\boldsymbol{j},n}}{\sigma\omega_{\boldsymbol{j},n}}n^{\frac{1}{2}-\frac{C_1\overline{\gamma}^2}{8\sigma^2}}
\lesssim\frac{1}{\sigma_0+o(1)}n^{-\left[\frac{C_1}{8(\sigma_0^2+o(1))}\overline{\gamma}^2-\lambda-1\right]},
\end{align*}
where we upper bounded $\sum_{\{|\theta_{\boldsymbol{j},\boldsymbol{k}}^0|>\overline{\gamma}\sqrt{\log{n}/n}\}}$ by $\prod_{l=1}^d\sum_{j_l=N_l}^{J_{n,l}-1}\sum_{k_l=0}^{2^{j_l}-1}\leq n^{1/2}$ for $n>2$. The right hand side will approach $0$ if $\overline{\gamma}$ is chosen large enough as $n\rightarrow\infty$.
\end{proof}

\begin{lemma}\label{lem:lemma2}
For small enough $\underline{\gamma}$ and large enough $\overline{\gamma}$, there exists constant $P_3>0$ such that uniformly in $f_0\in\mathcal{B}^{\boldsymbol{\alpha}}_{\infty,\infty}(R)$ with $\boldsymbol{\alpha}\in\mathbb{A}$ and any $0<R\leq R_0-1/2$,
\begin{align*}
\mathrm{E}_0\sup_{\sigma\in\mathcal{V}_n}\Pi(\mathcal{A}^c\cap\mathcal{C}|\boldsymbol{Y},\sigma)\leq\frac{(\log{n})^d}{n^{P_3}}.
\end{align*}
\end{lemma}

\begin{proof}[\textit{Proof of Lemma \ref{lem:lemma2}}]
By \eqref{eq:abc}, we have $\mathcal{A}^c=\cup_{\{|\theta^0_{\boldsymbol{j},\boldsymbol{k}}|>\underline{\gamma}\sqrt{\log{n}/n}\}}[|\theta_{\boldsymbol{j},\boldsymbol{k}}-\theta_{\boldsymbol{j},\boldsymbol{k}}^{0}|>\overline{\gamma}\sqrt{\log{n}/n}]$. Spilt the union such that $\mathcal{A}^c=\mathcal{A}_1\cup\mathcal{A}_2$ where $\mathcal{A}_1$ is union over $\{\underline{\gamma}\sqrt{\log{n}/n}<|\theta_{\boldsymbol{j},\boldsymbol{k}}^0|\leq\overline{\gamma}\sqrt{\log{n}/n}\}$ and $\mathcal{A}_2$ is over its complement $\{|\theta_{\boldsymbol{j},\boldsymbol{k}}^0|>\overline{\gamma}\sqrt{\log{n}/n}\}$. Then $\mathcal{A}^c\cap\mathcal{C}=(\mathcal{A}_1\cap\mathcal{C})\cup(\mathcal{A}_2\cap\mathcal{C})$. Define $\mathcal{Z}_{\boldsymbol{j},\boldsymbol{k}}:=\{|\theta_{\boldsymbol{j},\boldsymbol{k}}-\theta_{\boldsymbol{j},\boldsymbol{k}}^{0}|>\overline{\gamma}\sqrt{\log{n}/n}\}\cap\{\theta_{\boldsymbol{j},\boldsymbol{k}}\neq0\}$. In view of \eqref{eq:c}, observe that $\mathcal{A}_1\cap\mathcal{C}=\cup_{\{\underline{\gamma}\sqrt{\log{n}/n}<|\theta_{\boldsymbol{j},\boldsymbol{k}}^0|\leq\overline{\gamma}\sqrt{\log{n}/n}\}}\mathcal{Z}_{\boldsymbol{j},\boldsymbol{k}}$ while $\mathcal{A}_2\cap\mathcal{C}=\cup_{\{|\theta_{\boldsymbol{j},\boldsymbol{k}}^0|>\overline{\gamma}\sqrt{\log{n}/n}\}}\mathcal{Z}_{\boldsymbol{j},\boldsymbol{k}}$. Therefore by using a union bound, we can bound $\mathrm{E}_0\sup_{\sigma\in\mathcal{V}_n}\Pi(\mathcal{A}^c\cap\mathcal{C}|\boldsymbol{Y},\sigma)$ from above by
\begin{align}\label{eq:ac}
&\mathrm{E}_0\sup_{\sigma\in\mathcal{V}_n}\sum_{(\boldsymbol{j},\boldsymbol{k})\in\mathcal{J}_n(\underline{\gamma})}
\Pi\left(\theta_{\boldsymbol{j},\boldsymbol{k}}\in\mathcal{Z}_{\boldsymbol{j},\boldsymbol{k}}\middle|\boldsymbol{Y},\sigma\right)\mathbbm{1}_{\Omega_n(1)}
+P_0(\Omega_n(1)^c).
\end{align}
In view of \eqref{eq:omega}, the second term is bounded above by $n^{-1/2}(\log{n})^d$, and it goes to $0$ as $n\rightarrow\infty$. Using the same decomposition as in \eqref{eq:theta0}, we find that $\Pi\left(\theta_{\boldsymbol{j},\boldsymbol{k}}\in\mathcal{Z}_{\boldsymbol{j},\boldsymbol{k}}\middle|\boldsymbol{Y},\sigma\right)\mathbbm{1}_{\Omega_n(1)}$ is bounded above by
\begin{align*}
\frac{\int_{\mathcal{W}_n}I_n(\mathcal{Z}_{\boldsymbol{j},\boldsymbol{k}},\widetilde{\Theta})K_n(\widetilde{\Theta})d\Pi(\widetilde{\Theta})}{\int_{\mathcal{W}_n}I_n([\theta_{\boldsymbol{j},\boldsymbol{k}}\neq0],\widetilde{\Theta})K_n(\widetilde{\Theta})d\Pi(\widetilde{\Theta})}\mathbbm{1}_{\Omega_n(1)}+
\Pi(\mathcal{W}_n^c|\boldsymbol{Y},\sigma),
\end{align*}
where $\Pi(\mathcal{W}_n^c|\boldsymbol{Y},\sigma)=O_{P_0}\left(n^{-B}\right)$ for some constant $B>0$ follows from Lemma \ref{lem:betabound}. Recall that when $\theta_{\boldsymbol{j},\boldsymbol{k}}\neq0$, its prior density $d\Pi(\theta_{\boldsymbol{j},\boldsymbol{k}})\leq\omega_{\boldsymbol{j},n}p_{\mathrm{max}}d\theta_{\boldsymbol{j},\boldsymbol{k}}$. Hence, it follows that $I_n(\mathcal{Z}_{\boldsymbol{j},\boldsymbol{k}},\widetilde{\Theta})$ is bounded above by
\begin{align*}
p_{\mathrm{max}}\omega_{\boldsymbol{j},n}
\int_{|x|>\overline{\gamma}\sqrt{\log{n}/n}}\exp\left\{-\frac{(\boldsymbol{\Psi}_{\boldsymbol{j}}^T\boldsymbol{\Psi}_{\boldsymbol{j}})_{\boldsymbol{k},\boldsymbol{k}}}{2\sigma^2}
\left[x+\frac{\beta_n(\widetilde{\Theta})}{(\boldsymbol{\Psi}_{\boldsymbol{j}}^T\boldsymbol{\Psi}_{\boldsymbol{j}})_{\boldsymbol{k},\boldsymbol{k}}}\right]^2\right\}dx.
\end{align*}
Since $(\boldsymbol{j},\boldsymbol{k})\in\mathcal{J}_n(\underline{\gamma})\subset\mathcal{I}_n(\boldsymbol{\alpha})$ by \eqref{eq:important}, we will have $2^{\sum_{l=1}^dj_l}=o(n)$ and $(\boldsymbol{\Psi}_{\boldsymbol{j}}^T\boldsymbol{\Psi}_{\boldsymbol{j}})^{1/2}_{\boldsymbol{k},\boldsymbol{k}}\geq\sqrt{C_1n}$ by Lemma \ref{lem:phiphi}. Then if $|x|>\overline{\gamma}\sqrt{\log{n}/n}$ and under $\mathcal{W}_n$ and $\Omega_n(1)$, we have by twice application of the reverse triangular inequality that $\left|x+\frac{\beta_n(\widetilde{\Theta})}{(\boldsymbol{\Psi}_{\boldsymbol{j}}^T\boldsymbol{\Psi}_{\boldsymbol{j}})_{\boldsymbol{k},\boldsymbol{k}}}\right|$ is lower bounded by
\begin{align*}
&|x|-\frac{|\widetilde{\beta}_n(\widetilde{\Theta})|}{(\boldsymbol{\Psi}_{\boldsymbol{j}}^T\boldsymbol{\Psi}_{\boldsymbol{j}})_{\boldsymbol{k},\boldsymbol{k}}}-
\left|\frac{\widetilde{\beta}_n(\widetilde{\Theta})}{(\boldsymbol{\Psi}_{\boldsymbol{j}}^T\boldsymbol{\Psi}_{\boldsymbol{j}})_{\boldsymbol{k},\boldsymbol{k}}}-\frac{\beta_n(\widetilde{\Theta})}{(\boldsymbol{\Psi}_{\boldsymbol{j}}^T\boldsymbol{\Psi}_{\boldsymbol{j}})_{\boldsymbol{k},\boldsymbol{k}}}\right|\\
&>\overline{\gamma}\sqrt{\frac{\log{n}}{n}}-\frac{\tau_n}{C_1}\sqrt{\frac{\log{n}}{n}}-\frac{\sigma_0}{\sqrt{C_1n}}\sqrt{2\log{2^{\sum_{l=1}^dj_l}}+\log{n}}>\frac{\overline{\gamma}}{2}\sqrt{\frac{\log{n}}{n}},
\end{align*}
if $\overline{\gamma}$ is chosen large enough since $\tau_n\rightarrow0$ as $n\rightarrow\infty$. As a conclusion,
\begin{align*}
\left\{|x|>\overline{\gamma}\sqrt{\frac{\log{n}}{n}}\right\}\subset\left\{\left|x+\frac{\beta_n(\widetilde{\Theta})}{(\boldsymbol{\Psi}_{\boldsymbol{j}}^T\boldsymbol{\Psi}_{\boldsymbol{j}})_{\boldsymbol{k},\boldsymbol{k}}}\right|
>\frac{1}{2}\overline{\gamma}\sqrt{\frac{\log{n}}{n}}\right\}.
\end{align*}
Therefore, $I_n(\mathcal{Z}_{\boldsymbol{j},\boldsymbol{k}},\widetilde{\Theta})$ is further bounded above by $p_{\mathrm{max}}\omega_{\boldsymbol{j},n}$ times
\begin{align*}
&\exp{\left[\frac{-(\boldsymbol{\Psi}_{\boldsymbol{j}}^T\boldsymbol{\Psi}_{\boldsymbol{j}})_{\boldsymbol{k},\boldsymbol{k}}\overline{\gamma}^2\log{n}}{16n\sigma^2}\right]}
\int_{\mathbb{R}}\exp\left[\frac{-(\boldsymbol{\Psi}_{\boldsymbol{j}}^T\boldsymbol{\Psi}_{\boldsymbol{j}})_{\boldsymbol{k},\boldsymbol{k}}}{4\sigma^2}
\left(x+\frac{\beta_n(\widetilde{\Theta})}{(\boldsymbol{\Psi}_{\boldsymbol{j}}^T\boldsymbol{\Psi}_{\boldsymbol{j}})_{\boldsymbol{k},\boldsymbol{k}}}\right)^2\right]dx\\
&\qquad\lesssim n^{-1/2}\sigma\omega_{\boldsymbol{j},n}\exp\left[-\frac{C_1\overline{\gamma}^2\log{n}}{16\sigma^2}\right],
\end{align*}
again utilizing the bounds in Lemma \ref{lem:phiphi}. Using the upper bound established above and the lower bound for $\int_{\mathcal{W}_n}I_n([\theta_{\boldsymbol{j},\boldsymbol{k}}\neq0],\widetilde{\Theta})K_n(\widetilde{\Theta})d\Pi(\widetilde{\Theta})$ derived in \eqref{eq:lowerboundU}, we can upper bound the first term in \eqref{eq:ac} up to some constant by
\begin{align*}
\sup_{\sigma\in\mathcal{V}_n}\sum_{j_1=N_1}^{J_{n,1}-1}\cdots\sum_{j_d=N_d}^{J_{n,d}-1}\sum_{k_1=0}^{2^{j_1}-1}\cdots\sum_{k_d=0}^{2^{j_d}-1}
n^{-\frac{C_1\overline{\gamma}^2}{16\sigma^2}}\sigma\lesssim(\sigma_0+o(1))
n^{-\left[\frac{C_1}{16(\sigma_0^2+o(1))}\overline{\gamma}^2-1\right]},
\end{align*}
and it will approach $0$ if $\overline{\gamma}$ is large enough as $n\rightarrow\infty$.
\end{proof}

\begin{proof}[\textit{Proof of Corollary \ref{cor:adapt}}]
First note that the loss $f\mapsto\|f-f_0\|_\infty$ is unbounded but convex for any $f_0$. Let $\mathcal{F}:=\{\|D^{\boldsymbol{r}}f-D^{\boldsymbol{r}}f_0\|_\infty>M\epsilon_{n,\boldsymbol{r}}\}$. For $u\in\mathbb{N}$, we decompose $\mathcal{F}=\cup_{u=1}^\infty\mathcal{F}_u$ into slices $\mathcal{F}_u:=\{M\epsilon_{n,\boldsymbol{r}}u<\|D^{\boldsymbol{r}}f-D^{\boldsymbol{r}}f_0\|_\infty\leq M\epsilon_{n,\boldsymbol{r}}(u+1)\}$.

Now, if we introduced an extra $u$ factor in 3 places: the right hand side of the definition of $\mathcal{J}_n(\gamma)$ in \eqref{eq:Jngamma}, $\mathcal{A}$ in \eqref{eq:abc} and in \eqref{eq:sigma2bound} by replacing $c\log{n}$ with $cu^2\log{n}$, we see that by slightly modifying the proof of Theorem \ref{th:adapt},
\begin{align*}
\mathrm{E}_0\Pi(\mathcal{F}_u|\boldsymbol{Y})\leq(\log{n})^d\exp\{-C\log{(n)}u^2\},
\end{align*}
for some universal constant $C>0$. Therefore since $\mathcal{F}=\bigcup_{u=1}^{\infty}\mathcal{F}_u$,
\begin{align*}
\mathrm{E}_0\mathrm{E}(\|D^{\boldsymbol{r}}f-D^{\boldsymbol{r}}f_0\|_\infty|\boldsymbol{Y})&\leq M\epsilon_{n,\boldsymbol{r}}+\sum_{u=1}^\infty\mathrm{E}_0\mathrm{E}\left(\mathbbm{1}_{\mathcal{F}_u}\|D^{\boldsymbol{r}}f-D^{\boldsymbol{r}}f_0\|_\infty\middle|\boldsymbol{Y}\right)\\
&\leq M\epsilon_{n,\boldsymbol{r}}\left(1+2(\log{n})^d\sum_{u=1}^\infty ue^{-C\log{(n)}u^2}\right)\lesssim\epsilon_{n,\boldsymbol{r}},
\end{align*}
where the sum $(\log{n})^d\sum_{u=1}^\infty ue^{-C\log{(n)}u^2}=(\log{n})^d(n^{-C}+2n^{-4C}+3n^{-9C}+\cdots)$ converges when $n$ is large enough. By Jensen's inequality, $\|\mathrm{E}(D^{\boldsymbol{r}}f|\boldsymbol{Y})-D^{\boldsymbol{r}}f_0\|_\infty\leq\mathrm{E}(\|D^{\boldsymbol{r}}f-D^{\boldsymbol{r}}f_0\|_\infty|\boldsymbol{Y})$ and the result follows by taking $\mathrm{E}_0$ on both sides.
\end{proof}

\subsection{Proof of results in Section \ref{sec:test}}\label{sec:testproof}
For the proofs in this subsection, there is no qualitative difference in distinguishing between father and mother wavelet coefficients. Hence for notational simplicity, we combine the father and mother parts into a single sum and write the wavelet projection of $f$ at resolution $\boldsymbol{J}_n$ of \eqref{eq:prior} as $K_{\boldsymbol{J}_n}(f)(\boldsymbol{x})=\sum_{j_1=N_1-1}^{J_{n,1}-1}\cdots\sum_{j_d=N_d-1}^{J_{n,d}-1}\sum_{k_1=0}^{2^{j_1}-1}\cdots\sum_{k_d=0}^{2^{j_d}-1}\theta_{\boldsymbol{j},\boldsymbol{k}}\psi_{\boldsymbol{j},\boldsymbol{k}}(\boldsymbol{x})$, by delegating the father wavelets and their coefficients to the level $j_l=N_l-1,l=1,\dotsc,d$. Let $\|\cdot\|_n$ be the $L_2$-norm with respect to the empirical measure of $\{\boldsymbol{X}_1,\dotsc,\boldsymbol{X}_n\}$ and $\boldsymbol{1}_d$ be the $d$-dimensional vector of ones.

\begin{proof}[\textit{Proof of Proposition \ref{prop:test}}]
Let us choose a $g\in\mathcal{B}^{\boldsymbol{\alpha}}_{\infty,\infty}(R)$ such that $\|g-f_0\|_\infty>M\epsilon_n$ and $\|g-f_0\|_n\leq C\inf_{f\in\mathcal{B}^{\boldsymbol{\alpha}}_{\infty,\infty}(R):\|f-f_0\|_\infty>M\epsilon_n}\|f-f_0\|_n$ for some constant $C\geq1$. It then suffice to lower bound $\mathrm{E}_{f=g}(1-\phi_n)$.

Denote $\phi_{\mathrm{LR}}$ to be the likelihood ratio test for the simple hypotheses $H_0:f=f_0$ versus $H_1:f=g$. By a change of Gaussian measure and using the Cauchy-Schwarz inequality, $\mathrm{E}_{f_0}(1-\phi_{\mathrm{LR}})$ is bounded above by
\begin{align*}
\sqrt{\mathrm{E}_g(1-\phi_{\mathrm{LR}})}\sqrt{\int\left(\frac{dP^n_{f_0}}{dP^n_g}\right)^2dP^n_g}
\leq\sqrt{\mathrm{E}_g(1-\phi_{\mathrm{LR}})}e^{\frac{n\|f_0-g\|_n^2}{2\sigma_0^2}},
\end{align*}
where $P^n_f$ is the $n$-multivariate normal distribution with mean vector $(f(\boldsymbol{X}_1),\dotsc,f(\boldsymbol{X}_n))^T$ and covariance matrix $\sigma_0^2\boldsymbol{I}_n$. This inequality and the fact that $\phi_{\mathrm{LR}}$ is the uniformly most powerful test imply that $\mathrm{E}_{f_0}(\phi_{\mathrm{LR}})\leq\mathrm{E}_{f_0}(\phi_n)\leq\delta$ for any $0<\delta<1$, and $\mathrm{E}_g(1-\phi_n)$ is lower bounded by
\begin{align}\label{eq:alter}
&\mathrm{E}_g(1-\phi_{\mathrm{LR}})\geq[\mathrm{E}_{f_0}(1-\phi_{\mathrm{LR}})]^2e^{-n\|f_0-g\|_n^2/\sigma_0^2}\nonumber\\
&\qquad\geq(1-\delta)^2\exp\left\{-\frac{C^2n}{\sigma_0^2}\inf_{f\in\mathcal{B}^{\boldsymbol{\alpha}}_{\infty,\infty}(R):\|f-f_0\|_\infty>M\epsilon_n}\|f-f_0\|_n^2\right\},
\end{align}
by the definition of $g$. Note that $\boldsymbol{\alpha}\in\mathbb{A}$ implies $1/d+1/(2\alpha^{*})>1/\alpha_l$ and we further deduce $\alpha^{*}>d/2$ by summing both sides across all $l=1\dotsc,d$. Take $2^{J_{n,l}}=n^{1/(2\alpha_l)}$, and since $2^{\sum_{l=1}^dJ_{n,l}}=o(n)$ because of $\alpha^{*}>d/2$, we have by the triangle inequality that $\|f-f_0\|_n$ is bounded above by $\|K_{\boldsymbol{J}_n}(f)-K_{\boldsymbol{J}_n}(f_0)\|_n+\|f-K_{\boldsymbol{J}_n}(f)\|_n+\|f_0-K_{\boldsymbol{J}_n}(f_0)\|_n\lesssim\|\boldsymbol{\theta}-\boldsymbol{\theta}_0\|+\sum_{l=1}^d2^{-\alpha_lJ_{n,l}}$,
where we have used Lemma \ref{lem:2phiphi} to bound the first term, and utilized Proposition \ref{prop:wavelet} for the second and third terms since $f,f_0\in\mathcal{B}^{\boldsymbol{\alpha}}_{\infty,\infty}(R)$. By the continuous embedding of Proposition 4.3.11 in \citep{nickl2016} and Remark \ref{rem:b0}, $L_\infty\subset\mathcal{B}_{\infty,\infty}^{\boldsymbol{0}}$ and hence $\|f-f_0\|_{\mathcal{B}^{\boldsymbol{0}}_{\infty,\infty}}\leq\|f-f_0\|_\infty$. We then conclude that for $\mathcal{G}_n:=\{\boldsymbol{\theta}\in\mathcal{B}^{\boldsymbol{\alpha}}_{\infty,\infty}(R):\max_{\boldsymbol{j}}2^{\sum_{l=1}^dj_l/2}\max_{\boldsymbol{k}}|\theta_{\boldsymbol{j},\boldsymbol{k}}
-\theta^0_{\boldsymbol{j},\boldsymbol{k}}|>M\epsilon_n\}$,
\begin{align*}
\inf_{f\in\mathcal{B}^{\boldsymbol{\alpha}}_{\infty,\infty}(R):\|f-f_0\|_\infty>M\epsilon_n}\|f-f_0\|_n\lesssim\inf_{\boldsymbol{\theta}\in\mathcal{G}_n}\|\boldsymbol{\theta}-\boldsymbol{\theta}_0\|+\frac{d}{\sqrt{n}}.
\end{align*}
Now let $J_{n,l}(\boldsymbol{\alpha})$ be
\begin{align*}
2^{J_{n,l}(\boldsymbol{\alpha})}=\left(\frac{R}{M}\right)^{\left[\sum_{l=1}^d\alpha_l\left(\frac{1}{d}+\frac{1}{2\alpha{*}}-\frac{1}{2\alpha_l}\right)\right]^{-1}}\left(\frac{n}{\log{n}}\right)^{\frac{\alpha^{*}}{\alpha_l(2\alpha^{*}+d)}}.
\end{align*}
Consider a $\boldsymbol{\theta}^{*}\in\mathcal{B}^{\boldsymbol{\alpha}}_{\infty,\infty}(R)$ such that $|\theta_{\boldsymbol{j},\boldsymbol{k}}^{*}-\theta_{\boldsymbol{j},\boldsymbol{k}}^0|=R2^{-\sum_{l=1}^d\alpha_lJ_{n,l}(\boldsymbol{\alpha})\left(\frac{1}{d}+\frac{1}{2\alpha^{*}}\right)}$ for $j_l=J_{n,l}(\boldsymbol{\alpha}),l=1\dotsc,d$ and $\boldsymbol{k}=\boldsymbol{0}$, but $|\theta_{\boldsymbol{j},\boldsymbol{k}}^{*}-\theta_{\boldsymbol{j},\boldsymbol{k}}^0|=0$ for all the other multi-indices. Then,
\begin{align*}
\max_{\boldsymbol{j}}2^{\sum_{l=1}^dj_l/2}\max_{\boldsymbol{k}}|\theta^{*}_{\boldsymbol{j},\boldsymbol{k}}-\theta^0_{\boldsymbol{j},\boldsymbol{k}}|
=R2^{-\sum_{l=1}^d\alpha_lJ_{n,l}(\boldsymbol{\alpha})\left(\frac{1}{d}+\frac{1}{2\alpha^{*}}-\frac{1}{2\alpha_l}\right)}=M\epsilon_n
\end{align*}
and this implies that $\boldsymbol{\theta}^{*}\in\mathcal{G}_n$, but since $J_{n,l}(\boldsymbol{\alpha})\leq J_{n,l}$,
\begin{align*}
\|\boldsymbol{\theta}^{*}-\boldsymbol{\theta}_0\|+d/\sqrt{n}=R2^{-\sum_{l=1}^d\alpha_lJ_{n,l}(\boldsymbol{\alpha})\left(\frac{1}{d}+\frac{1}{2\alpha^{*}}\right)}+d/\sqrt{n}\lesssim\sqrt{\log{n}/n}.
\end{align*}
Therefore, $\inf_{f\in\mathcal{B}^{\boldsymbol{\alpha}}_{\infty,\infty}(R):\|f-f_0\|_\infty>M\epsilon_n}\|f-f_0\|_n^2\lesssim\log{n}/n$ and plugging this back into \eqref{eq:alter} gives the result.
\end{proof}

\begin{proof}[\textit{Proof of Proposition \ref{prop:testing}}]
The proof of the first statement follows closely the steps outlined in the proof of Proposition \ref{prop:test}. By adapting \eqref{eq:alter} to our present setting, we have
\begin{align}\label{eq:alter2}
\mathrm{E}_g(1-\phi_{\mathrm{LR}})\gtrsim\exp\left\{-\frac{C^2n}{\sigma_0^2}\inf_{f\in\mathcal{B}^{\boldsymbol{\alpha}}_{\infty,\infty}(R):\|f-f_0\|_\infty>Mr_n\epsilon_n}\|f-f_0\|_n^2\right\}.
\end{align}
Hence, we need to upper bound the infimum in the exponent. By using the same argument as in the paragraph after \eqref{eq:alter}, the infimum in question becomes
\begin{align*}
\inf_{f\in\mathcal{B}^{\boldsymbol{\alpha}}_{\infty,\infty}(R):\|f-f_0\|_\infty>Mr_n\epsilon_n}\|f-f_0\|_n\lesssim\inf_{\boldsymbol{\theta}\in\mathcal{G}_n(r_n)}\|\boldsymbol{\theta}-\boldsymbol{\theta}_0\|+\frac{d}{\sqrt{n}},
\end{align*}
where $\mathcal{G}_n(r_n):=\{\boldsymbol{\theta}\in\mathcal{B}^{\boldsymbol{\alpha}}_{\infty,\infty}(R):\max_{\boldsymbol{j}}2^{\sum_{l=1}^dj_l/2}\max_{\boldsymbol{k}}|\theta_{\boldsymbol{j},\boldsymbol{k}}
-\theta^0_{\boldsymbol{j},\boldsymbol{k}}|>Mr_n\epsilon_n\}$.

Now let $i_{n,l},l=1,\dotsc,d$ such that $\frac{2\alpha^{*}}{2\alpha^{*}+d}J_{n,l}(\boldsymbol{\alpha})<i_{n,l}<J_{n,l}(\boldsymbol{\alpha})$ where $J_{n,l}(\boldsymbol{\alpha})$ is such that $2^{J_{n,l}(\boldsymbol{\alpha})}=(n/\log{n})^{\alpha^{*}/\{\alpha_l(2\alpha^{*}+d)\}}$. For the present case, let us consider a $\boldsymbol{\theta}^{*}\in\mathcal{B}^{\boldsymbol{\alpha}}_{\infty,\infty}(R)$ so that $|\theta_{\boldsymbol{j},\boldsymbol{k}}^{*}-\theta_{\boldsymbol{j},\boldsymbol{k}}^0|=R2^{-\sum_{l=1}^d\alpha_li_{n,l}\left(\frac{1}{d}+\frac{1}{2\alpha^{*}}\right)\left(1+\frac{d}{2\alpha^{*}}\right)}$ for $j_l=i_{n,l},l=1\dotsc,d$ and $\boldsymbol{k}=\boldsymbol{0}$, but $|\theta_{\boldsymbol{j},\boldsymbol{k}}^{*}-\theta_{\boldsymbol{j},\boldsymbol{k}}^0|=0$ for all the other multi-indices. Now since $r_n\epsilon_n=o(\rho_n\epsilon_n)$ and $1/d+1/(2\alpha^{*})-1/(2\alpha_l)>0$ for all $\boldsymbol{\alpha}\in\mathbb{A}$, we can always find an $i_{n,l}$ so that
\begin{align*}
Mr_n\epsilon_n&<2^{\sum_{l=1}^di_{n,l}}2^{-\sum_{l=1}^d\alpha_li_{n,l}\left(\frac{1}{d}+\frac{1}{2\alpha^{*}}\right)\left(1+\frac{d}{2\alpha^{*}}\right)}
<2^{-\sum_{l=1}^d\alpha_li_{n,l}\left(\frac{1}{d}+\frac{1}{2\alpha^{*}}-\frac{1}{2\alpha_l}\right)}\\
&<2^{-\sum_{l=1}^d\left(\frac{1}{d}+\frac{1}{2\alpha^{*}}-\frac{1}{2\alpha_l}\right)\frac{2\alpha^{*}}{2\alpha^{*}+d}\alpha_lJ_{n,l}(\boldsymbol{\alpha})}=\rho_n\epsilon_n.
\end{align*}
The expression immediately right of $Mr_n\epsilon_n$ is $\max_{\boldsymbol{j}}2^{\sum_{l=1}^dj_l/2}\max_{\boldsymbol{k}}|\theta^{*}_{\boldsymbol{j},\boldsymbol{k}}-\theta^0_{\boldsymbol{j},\boldsymbol{k}}|$, and hence we can conclude that $\boldsymbol{\theta}^{*}\in\mathcal{G}_n(r_n)$. However since $i_{n,l}>\frac{2\alpha^{*}}{2\alpha^{*}+d}J_{n,l}(\boldsymbol{\alpha})$, we will have
\begin{align*}
\|\boldsymbol{\theta}^{*}-\boldsymbol{\theta}_0\|+d/\sqrt{n}&=R2^{-\sum_{l=1}^d\alpha_li_{n,l}\left(\frac{1}{d}+\frac{1}{2\alpha^{*}}\right)\left(1+\frac{d}{2\alpha^{*}}\right)}+d/\sqrt{n}\\
&\lesssim2^{-\sum_{l=1}^d\alpha_l\frac{2\alpha^{*}}{2\alpha^{*}+d}J_{n,l}(\boldsymbol{\alpha})\left(\frac{1}{d}+\frac{1}{2\alpha^{*}}\right)\left(1+\frac{d}{2\alpha^{*}}\right)}+1/\sqrt{n}\lesssim\sqrt{\log{n}/n}
\end{align*}
and hence $\inf_{f\in\mathcal{B}^{\boldsymbol{\alpha}}_{\infty,\infty}(R):\|f-f_0\|_\infty>Mr_n\epsilon_n}\|f-f_0\|_n^2\lesssim\log{n}/n$. The first statement follows by substituting this back to \eqref{eq:alter2}.

For the second statement, choose $h_{n,l}(\boldsymbol{\alpha}),l=1,\dotsc,d$ such that $2^{h_{n,l}(\boldsymbol{\alpha})}=\epsilon_n^{-2\alpha^{*}/\{\alpha_l(2\alpha^{*}+d)\}}$. Collect $\boldsymbol{h}_n(\boldsymbol{\alpha})=(h_{n,1}(\boldsymbol{\alpha}),\dotsc,h_{n,d}(\boldsymbol{\alpha}))^T$ and construct $\boldsymbol{\psi}_{\boldsymbol{h}_n(\boldsymbol{\alpha})}(\boldsymbol{x})$ by concatenating $\psi_{\boldsymbol{j},\boldsymbol{k}}(\boldsymbol{x})$ across $\boldsymbol{N}\leq\boldsymbol{j}\leq\boldsymbol{h}_n(\boldsymbol{\alpha})-\boldsymbol{1}_d$ and all $\boldsymbol{k}$ in lexicographic order. Let us define $\boldsymbol{\Psi}:=(\boldsymbol{\psi}_{\boldsymbol{h}_n(\boldsymbol{\alpha})}(\boldsymbol{X}_1)^T,\cdots,\boldsymbol{\psi}_{\boldsymbol{h}_n(\boldsymbol{\alpha})}(\boldsymbol{X}_n)^T)^T$ to be the wavelet basis matrix. We consider the plug-in test $\Phi_n=\mathbbm{1}\{\|\widehat{f}_{n,\boldsymbol{\alpha}}-f_0\|_{\infty}>M_0\rho_n\epsilon_n\}$ for some constant $0<M_0<M$, by using the least squares estimator $\widehat{f}_{n,\boldsymbol{\alpha}}(\boldsymbol{x}):=\boldsymbol{\psi}_{\boldsymbol{h}_n(\boldsymbol{\alpha})}(\boldsymbol{x})^T(\boldsymbol{\Psi}^T\boldsymbol{\Psi})^{-1}\boldsymbol{\Psi}^T\boldsymbol{Y}$.

For any $f_0\in\mathcal{B}^{\boldsymbol{\alpha}}_{\infty,\infty}(R)$, we know from \eqref{eq:type21} of Lemma \ref{lem:type2} that the least squares bias is $\|\mathrm{E}_0\widehat{f}_{n,\boldsymbol{\alpha}}-f_0\|_\infty\leq Cd\epsilon_n^{2\alpha^{*}/(2\alpha^{*}+d)}$. Recall that $\rho_n=\epsilon_n^{-d/(2\alpha^{*}+d)}$. Hence by the triangle inequality and taking $M_0>Cd$,
\begin{align*}
\mathrm{E}_0\Phi_n&\leq P_0\left(\|\widehat{f}_{n,\boldsymbol{\alpha}}-\mathrm{E}_0\widehat{f}_{n,\boldsymbol{\alpha}}\|_\infty>M_0\rho_n\epsilon_n-\|\mathrm{E}_0\widehat{f}_{n,\boldsymbol{\alpha}}-f_0\|_\infty\right)\\
&\qquad\leq P_0\left(\|\widehat{f}_{n,\boldsymbol{\alpha}}-\mathrm{E}_0\widehat{f}_{n,\boldsymbol{\alpha}}\|_\infty>(M_0-Cd)\epsilon_n^{2\alpha^{*}/(2\alpha^{*}+d)}\right).
\end{align*}
Now apply \eqref{eq:type22} of Lemma \ref{lem:type2} with $x=C_In\epsilon_n^2$ and since $\epsilon_n\gg\sqrt{\log{n}/n}$,
\begin{align*}
Q_1\sqrt{\frac{2^{\sum_{l=1}^dh_{n,l}(\boldsymbol{\alpha})}\log{n}}{n}}+\sqrt{2Q_2C_{I}}2^{\sum_{l=1}^dh_{n,l}(\boldsymbol{\alpha})/2}\epsilon_n\leq2\sqrt{2Q_2C_I}\epsilon_n^{2\alpha^{*}/(2\alpha^{*}+d)},
\end{align*}
when $n$ is large enough. Therefore, take $M_0>2\sqrt{2Q_2C_I}+Cd$ and we will have $\mathrm{E}_0\Phi_n\leq e^{-C_In\epsilon_n^2}$.

For the Type II error with $f\in\{\mathcal{B}^{\boldsymbol{\alpha}}_{\infty,\infty}(R):\|f-f_0\|_\infty>M\rho_n\epsilon_n\}$ such that $M>M_0$, we apply the reverse triangle inequality twice to yield
\begin{align*}
&\mathrm{E}_f(1-\Phi_n)=P_f\left(\|\widehat{f}_{n,\boldsymbol{\alpha}}-f_0\|_\infty\leq M_0\rho_n\epsilon_n\right)\\
&\leq P_f\left(\|\widehat{f}_{n,\boldsymbol{\alpha}}-\mathrm{E}_f\widehat{f}_{n,\boldsymbol{\alpha}}\|_\infty\geq\|f-f_0\|_\infty-M_0\rho_n\epsilon_n-\|\mathrm{E}_f\widehat{f}_{n,\boldsymbol{\alpha}}-f\|_\infty\right).
\end{align*}
Now since $f\in\mathcal{B}^{\boldsymbol{\alpha}}_{\infty,\infty}(R)$, we apply \eqref{eq:type21} again to conclude $\|\mathrm{E}_f\widehat{f}_{n,\boldsymbol{\alpha}}-f\|_\infty\leq Cd\epsilon_n^{2\alpha^{*}/(2\alpha^{*}+d)}$. Hence we are in the same situation as in the Type I error case, and we can use the same argument to conclude that for $M$ and $n$ large enough, we have $\sup_{f\in\mathcal{B}^{\boldsymbol{\alpha}}_{\infty,\infty}(R):\|f-f_0\|_\infty>M\rho_n\epsilon_n}\mathrm{E}_f(1-\Phi_n)\leq e^{-C_{II}n\epsilon_n^2}$ for some constant $C_{II}>0$ when $\rho_n=\epsilon_n^{-d/(2\alpha^{*}+d)}$. The last statement is proved using the master theorem (see Theorem 3 of \citep{converge2007}) once we have tests with exponential errors. The Kullback-Leibler neighborhood and prior complement criteria follow the same steps as in Lemma \ref{lem:l2contract} to prove $L_2$-contraction rate.
\end{proof}

\section{Technical lemmas}\label{sec:proof3}
The lemma below quantifies the error in approximating Riemann's sum with its integral version, and it is useful to give size estimates of various discrete sums found in this paper. Let $\langle f,g\rangle$ be the inner product of two functions $f,g$ in Hilbert space.
\begin{lemma}\label{lem:discrete}
Suppose $\partial^df/(\partial x_1\cdots\partial x_d)\in L_1$, then if the fixed design points are chosen such that \eqref{eq:cdf} holds, we have for some constant $C>0$,
\begin{align*}
\left|\left|\frac{1}{n}\sum_{i=1}^nf(\boldsymbol{X}_i)\right|-\left|\int_{[0,1]^d}f(\boldsymbol{x})d\boldsymbol{x}\right|\right|\leq C\frac{1}{n}\int_{[0,1]^d}\left|\frac{\partial^d}{\partial x_1,\cdots\partial x_d}f(\boldsymbol{x})\right|d\boldsymbol{x}.
\end{align*}
\end{lemma}
\begin{proof}
Note that $n^{-1}\sum_{i=1}^nf(\boldsymbol{X}_i)=\int_{[0,1]^d}f(\boldsymbol{x})dG_n(\boldsymbol{x})$ where $G_n(\boldsymbol{x})$ is the empirical distribution. By the triangle inequality,
\begin{align*}
\left|\int_{[0,1]^d}f(\boldsymbol{x})dG_n(\boldsymbol{x})\right|\leq\left|\int_{[0,1]^d}f(\boldsymbol{x})dU(\boldsymbol{x})\right|+\left|\int_{[0,1]^d}f(\boldsymbol{x})d(G_n-U)(\boldsymbol{x})\right|,
\end{align*}
where $U(\boldsymbol{x})$ is the $\mathrm{Uniform}([0,1]^d)$ cumulative distribution function. Thus, the first term is $|\int_{[0,1]^d}f(\boldsymbol{x})d\boldsymbol{x}|$. To bound the second term, observe that by the multivariate integration by parts, $\int_{[0,1]^d}f(\boldsymbol{x})d(G_n-U)(\boldsymbol{x})$ is
\begin{align*}
f(\boldsymbol{x})(G_n-U)(\boldsymbol{x})|_{\boldsymbol{0}}^{\boldsymbol{1}_d}+\int_{[0,1]^d}\frac{\partial^df(\boldsymbol{x})}{\partial x_1,\cdots\partial x_d}(G_n-U)(\boldsymbol{x})d\boldsymbol{x}.
\end{align*}
Since $(G_n-U)(\boldsymbol{1}_d)=(G_n-U)(\boldsymbol{0})=0$, it follows by assumption \eqref{eq:cdf} that the second term is bounded above by
\begin{align*}
\|G_n-U\|_\infty\int_{[0,1]^d}\left|\frac{\partial^df(\boldsymbol{x})}{\partial x_1,\cdots\partial x_d}\right|d\boldsymbol{x}\lesssim\frac{1}{n}\int_{[0,1]^d}\left|\frac{\partial^df(\boldsymbol{x})}{\partial x_1,\cdots\partial x_d}\right|d\boldsymbol{x}.
\end{align*}
For the other direction, use the reverse triangle inequality. This together with the upper bound established above will then prove the result.
\end{proof}

\begin{lemma}\label{lem:phiphi}
Under the assumption of \eqref{eq:cdf},
\begin{align}
\left||(\boldsymbol{\Psi}_{\boldsymbol{a}}^T\boldsymbol{\Psi}_{\boldsymbol{b}})_{\boldsymbol{c},\boldsymbol{e}}|
-n|\langle\psi_{\boldsymbol{a},\boldsymbol{c}},\psi_{\boldsymbol{b},\boldsymbol{e}}\rangle|\right|&\lesssim\prod_{l=1}^d2^{(a_l+b_l)/2},\nonumber\\
|(\boldsymbol{\Psi}_{\boldsymbol{a}}^T\boldsymbol{B})_{\boldsymbol{c},\boldsymbol{m}}|&
\lesssim\prod_{l=1}^d2^{a_l/2}.\label{eq:phiphi1}
\end{align}
In particular, for $j_l\leq\widetilde{J}_{n,l},l=1,\dotsc,d$ where $\widetilde{J}_{n,l}$ is increasing with $n$, then if $2^{\sum_{l=1}^d\widetilde{J}_{n,l}}=o(n)$, this implies that $(\boldsymbol{\Psi}_{\boldsymbol{j}}^T\boldsymbol{\Psi}_{\boldsymbol{j}})_{\boldsymbol{k},\boldsymbol{k}}\asymp n$ by the orthonormality of $\psi_{\boldsymbol{j},\boldsymbol{k}}$, and
\begin{align}\label{eq:phiphi3}
\sum_{i=1}^n|\psi_{\boldsymbol{j},\boldsymbol{k}}(\boldsymbol{X}_i)|\lesssim n\prod_{l=1}^d2^{-j_l/2}.
\end{align}
\end{lemma}

\begin{proof}
To prove the first statement of \eqref{eq:phiphi1}, apply Lemma \ref{lem:discrete} with $f(\boldsymbol{x})=\psi_{\boldsymbol{a},\boldsymbol{c}}(\boldsymbol{x})\psi_{\boldsymbol{b},\boldsymbol{e}}(\boldsymbol{x})$. Let us denote $I(f):=\int_{[0,1]^d}\left|\frac{\partial^d}{\partial x_1,\cdots\partial x_d}f(\boldsymbol{x})\right|d\boldsymbol{x}$. Then,
\begin{align*}
I(f)=\int_{[0,1]^d}\left|\psi_{\boldsymbol{a},\boldsymbol{c}}(\boldsymbol{x})\frac{\partial^d}{\partial x_1\cdots\partial x_d}\psi_{\boldsymbol{b},\boldsymbol{e}}(\boldsymbol{x})+\psi_{\boldsymbol{b},\boldsymbol{e}}(\boldsymbol{x})\frac{\partial^d}{\partial x_1\cdots\partial x_d}\psi_{\boldsymbol{a},\boldsymbol{c}}(\boldsymbol{x})\right|d\boldsymbol{x}.
\end{align*}
By construction, the support of the CDV wavelet $\psi_{\boldsymbol{a},\boldsymbol{c}}$ is some compact set $\mathcal{I}_{\boldsymbol{a}}$ such that its Lebesgue measure is $O\left(\prod_{l=1}^d2^{-a_l}\right)$. Therefore, we can restrict the domain of integration to $\mathcal{I}_{\boldsymbol{a}}\cap\mathcal{I}_{\boldsymbol{b}}$, which is the intersection of the supports of $\psi_{\boldsymbol{a},\boldsymbol{c}}$ and $\psi_{\boldsymbol{b},\boldsymbol{e}}$. Since the wavelets and their derivatives are uniformly bounded, we can upper bound $I(f)$ up to some constant multiple by
\begin{align*}
\prod_{l=1}^d2^{(a_l+b_l)/2}\left(\int_{\mathcal{I}_{\boldsymbol{b}}}\prod_{l=1}^d2^{b_l}\|\psi_{k_l}\|_{\infty}\|\psi_{k_l}^{'}\|_{\infty}d\boldsymbol{x}
+\int_{\mathcal{I}_{\boldsymbol{a}}}\prod_{l=1}^d2^{a_l}\|\psi_{k_l}\|_{\infty}\|\psi_{k_l}^{'}\|_{\infty}d\boldsymbol{x}\right),
\end{align*}
which is of the order $\prod_{l=1}^d2^{(a_l+b_l)/2}$.

For the second assertion of \eqref{eq:phiphi1}, we take $f(\boldsymbol{x})=\psi_{\boldsymbol{a},\boldsymbol{c}}(\boldsymbol{x})\varphi_{\boldsymbol{N},\boldsymbol{m}}(\boldsymbol{x})$. Then $\langle\psi_{\boldsymbol{a},\boldsymbol{c}},\varphi_{\boldsymbol{N},\boldsymbol{m}}\rangle=0$ by orthonormality, and $I(f)$ is of the order of $\prod_{l=1}^d2^{a_l/2}\left(1+\int_{\mathcal{I}_{\boldsymbol{a}}}\prod_{l=1}^d2^{a_l}d\boldsymbol{x}\right)\lesssim\prod_{l=1}^d2^{a_l/2}$. The assertion then follows by appealing to Lemma \ref{lem:discrete}.

To prove \eqref{eq:phiphi3}, let $f(\boldsymbol{x})=|\psi_{\boldsymbol{j},\boldsymbol{k}}(\boldsymbol{x})|$ and note that $\int_{[0,1]^d}f(\boldsymbol{x})d\boldsymbol{x}=\prod_{l=1}^d\int_0^1|\psi_{j_k,k_l}(x_l)|dx_l$. Then if $\psi_{j_l,k_l}$ is an interior CDV wavelet, we will have by a change of variable
\begin{align*}
\int_0^1|\psi_{j_l,k_l}(x_l)|dx_l=\int_{2^{-j_l}(-N_l+1+k_l)}^{2^{-j_l}(N_l+k_l)}2^{j_l/2}|\psi_{k_l}(2^{j_l}x_l)|dx_l\lesssim2^{-j_l/2},
\end{align*}
where the constant in $\lesssim$ above does not depend on $j_l$. The same argument holds for the boundary corrected case, where the lower limit of the integral is replaced by $0$ if the $0$-boundary is considered and the upper limit by $1$ for the $1$-boundary. As a result, $\int_{[0,1]^d}f(\boldsymbol{x})d\boldsymbol{x}\lesssim\prod_{l=1}^d2^{-j_l/2}$. By restricting the domain of integration to the support of $\psi_{\boldsymbol{j},\boldsymbol{k}}$, i.e., $\mathcal{I}_{\boldsymbol{j}}$,
\begin{align*}
I(f)\lesssim\int_{\mathcal{I}_{\boldsymbol{j}}}\left|\prod_{l=1}^d2^{j_l/2+j_l}\psi^{'}_{k_l}(2^{j_l}x_l)\mathrm{sgn}[\psi_{k_l}(2^{j_l}x_l)]\right|d\boldsymbol{x}\lesssim\prod_{l=1}^d2^{j_l/2},
\end{align*}
with $\mathrm{sgn}(\cdot)$ denoting the sign function, i.e., $\mathrm{sgn}(x)=1$ if $x\geq0$ and is $-1$ if $x<0$. Therefore if $j_l\leq\widetilde{J}_{n,l}$ where $2^{\sum_{l=1}^d\widetilde{J}_{n,l}}=o(n)$, then the above is $o(1)$ and the result follows.
\end{proof}

\begin{remark}
For uniform random design, the stochastic version of Lemma \ref{lem:phiphi} can be deduced from Bernstein's inequality (see (3.24) in Theorem 3.1.7 of \citep{nickl2016}). In this case, \eqref{eq:phiphi1} is:
\begin{align*}
&P\left[\left|(\boldsymbol{\Psi}_{\boldsymbol{a}}^T\boldsymbol{\Psi}_{\boldsymbol{b}})_{\boldsymbol{c},\boldsymbol{e}}-
n\langle\psi_{\boldsymbol{a},\boldsymbol{c}},\psi_{\boldsymbol{b},\boldsymbol{e}}\rangle\right|\lesssim\sqrt{n}\prod_{l=1}^d2^{(a_l+b_l)/2}\right]\geq1-2e^{-2^{-\sum_{l=1}^d\left(\frac{a_l+b_l}{2}\right)}}\\
&P\left[\left|(\boldsymbol{\Psi}_{\boldsymbol{a}}^T\boldsymbol{B})_{\boldsymbol{c},\boldsymbol{m}}\right|\lesssim\sqrt{n}\prod_{l=1}^d2^{a_l/2}\right]\geq1-2e^{-2^{\sum_{l=1}^da_l}}.
\end{align*}
The extra $\sqrt{n}$ is due to the fact that $\|G_n-U\|_\infty=O_P(n^{-1/2})$ by Donsker's theorem in the random case, instead of the rate $O(n^{-1})$ for the fixed design case as in \eqref{eq:cdf}. In particular, for $j_l\leq\widetilde{J}_{n,l},l=1,\dotsc,d$ where $\widetilde{J}_{n,l}$ is increasing with $n$, then if $2^{\sum_{l=1}^d\widetilde{J}_{n,l}}=o(\sqrt{n})$, this implies that $(\boldsymbol{\Psi}_{\boldsymbol{j}}^T\boldsymbol{\Psi}_{\boldsymbol{j}})_{\boldsymbol{k},\boldsymbol{k}}\asymp n$ by the orthonormality of $\psi_{\boldsymbol{j},\boldsymbol{k}}$ and $\sum_{i=1}^n|\psi_{\boldsymbol{j},\boldsymbol{k}}(\boldsymbol{X}_i)|\lesssim n\prod_{l=1}^d2^{-j_l/2}$ with probability at least $1-2e^{-2^{\sum_{l=1}^dj_l}}$.
\end{remark}

The lemma below gives the $L_2$-posterior contraction rate for spike-and-slab prior in nonparametric regression models. It shows in particular that there is an extra logarithmic factor in the rate, and is a reflection of the fact that separable selection rules (coefficient-wise spike-and-slab) will have at least a logarithmic penalty when trying to estimate $f$ adaptively under a global $L_2$-loss.

\begin{lemma}\label{lem:l2contract}
Under the hierarchical spike-and-slab prior in \eqref{eq:prior}, there exist constants $M,P_4>0$ such that for any $0<\alpha_l<\eta+1,l=1,\dotsc,d$ and uniformly over $f_0\in\mathcal{B}^{\boldsymbol{\alpha}}_{\infty,\infty}(R)$,
\begin{align*}
\mathrm{E}_0\Pi\left(\|f-f_0\|_n+|\sigma^2-\sigma_0^2|>M(n/\log{n})^{-\alpha^{*}/(2\alpha^{*}+d)}\middle|\boldsymbol{Y}\right)\leq n^{-P_4}.
\end{align*}
\end{lemma}
\begin{proof}
We will use the master theorem (see Theorem 3 of \cite{converge2007}) by constructing test function with exponential error probabilities, and verifying that the prior gives sufficient mass on Kullback-Leibler neighborhood around $(f_0,\sigma_0^2)$. Let $\epsilon_n\rightarrow0$ and $n\epsilon_n^2\rightarrow\infty$. For $l=1,\dotsc,d$, we choose $J_{n,l}(\boldsymbol{\alpha})$ such that $0.5U(n/\log{n})^{\alpha^{*}/\{\alpha_l(2\alpha^{*}+d)\}}\leq2^{J_{n,l}(\boldsymbol{\alpha})}\leq U(n/\log{n})^{\alpha^{*}/\{\alpha_l(2\alpha^{*}+d)\}}$ for some large enough constant $U>0$. For $\Theta_n:=\{\boldsymbol{\theta}:\theta_{\boldsymbol{j},\boldsymbol{k}}=\theta_{\boldsymbol{j},\boldsymbol{k}}\mathbbm{1}_{\{\boldsymbol{j}\leq\boldsymbol{J}_n(\boldsymbol{\alpha}),\boldsymbol{k}\}}\}$,
define sieves $\mathcal{F}_n:=\{f:\boldsymbol{\theta}\in\Theta_n\}$ consisting of functions with wavelet expansion truncated at levels $J_{n,l}(\boldsymbol{\alpha})-1, l=1,\dotsc,d$.

For any $f\in\mathcal{F}_n$ and by the property of $L_2$-projection, $\|f-K_{\boldsymbol{J}_n(\boldsymbol{\alpha})}(f_0)\|_n\leq\|f-f_0\|_n$. Then by the triangle inequality, $\|f-f_0\|_n\leq\|f-K_{\boldsymbol{J}_n(\boldsymbol{\alpha})}(f_0)\|_n+\|K_{\boldsymbol{J}_n(\boldsymbol{\alpha})}(f_0)-f_0\|_\infty\lesssim \|f-K_{\boldsymbol{J}_n(\boldsymbol{\alpha})}(f_0)\|_n+\sum_{l=1}^d2^{-\alpha_lJ_{n,l}(\boldsymbol{\alpha})}$, where the last inequality follows from Proposition \ref{prop:wavelet} since $f_0\in\mathcal{B}^{\boldsymbol{\alpha}}_{\infty,\infty}(R)$. In view of Lemma \ref{lem:2phiphi}, $\|f-K_{\boldsymbol{J}_n(\boldsymbol{\alpha})}(f_0)\|_n^2\asymp\|\boldsymbol{\vartheta}-\boldsymbol{\vartheta}_0\|^2+\|\widetilde{\boldsymbol{\theta}}-\widetilde{\boldsymbol{\theta}}_0\|^2$, since $2^{\sum_{l=1}^dJ_{n,l}(\boldsymbol{\alpha})}=o(n)$. Here the tilde in $\widetilde{\boldsymbol{\theta}}$ represents the truncated mother wavelet coefficients. We then conclude that there are constants $W_1,W_2>0$ such that for $f\in\mathcal{F}_n$,
\begin{align}\label{eq:coeff}
W_1\|\widetilde{\boldsymbol{\theta}}-\widetilde{\boldsymbol{\theta}}_0\|&\leq\|f-f_0\|_n\nonumber\\
&\leq W_2\left[\|\boldsymbol{\vartheta}-\boldsymbol{\vartheta}_0\|+\|\widetilde{\boldsymbol{\theta}}-\widetilde{\boldsymbol{\theta}}_0\|+(\log{n}/n)^{\alpha^{*}/(2\alpha^{*}+d)}\right]
\end{align}
by the definition of $J_{n,l}(\boldsymbol{\alpha})$. Let us define sieve slices $\mathcal{F}_n^j=\{f\in\mathcal{F}_n:j\epsilon_n<\|f-f_0\|_n+|\sigma^2-\sigma_0^2|\leq(j+1)\epsilon_n\}$ for any integer $j\geq M$. It follows from \eqref{eq:coeff} above that
\begin{align*}
\mathcal{F}_n^j\subset\left\{\|\widetilde{\boldsymbol{\theta}}-\widetilde{\boldsymbol{\theta}}_0\|\leq(2/W_1)j\epsilon_n,\text{ }|\sigma^2-\sigma_0^2|\leq2j\epsilon_n\right\}.
\end{align*}
By calculating the covering number of the Euclidean space on the right hand side, we conclude that $\mathcal{F}_n^j$ has a $\epsilon_n$-net of at most $e^{Cjn\epsilon_n^2}$ points for some constant $C>0$ if $2^{\sum_{l=1}^dJ_{n,l}(\boldsymbol{\alpha})}\lesssim n\epsilon_n^2$. Then by Lemma 1 of \citep{test}, there exists a test $\phi_{n,j}$ with exponentially small error probabilities for testing $f=f_0$ against $f\in\mathcal{F}_n^j$, by maximizing over 3 tests corresponding to the cases where $|\sigma^2-\sigma_0^2|\leq\sigma_0^2/2,\sigma^2>3\sigma_0^2/2$ and $\sigma^2<\sigma_0^2/2$. Then using the arguments outlined in the proof of Theorem 9 in \cite{converge2007}, we conclude that $\phi_n=\sup_{j\geq M}\phi_{n,j}$ is a test with exponentially small Type I and II errors, thus fulfilling the testing requirement of the master theorem.

To characterize prior concentration, let $K(p,q):=\int p\log{(p/q)}d\mu$ be the Kullback-Leibler divergence and $V(p,q):=\int p[\log{(p/q)}-K(p,q)]^2d\mu$, where $\mu$ is the Lesbegue measure. Define the Kullback-Leibler neighborhood $B_n(\epsilon_n):=\{(f,\sigma^2): n^{-1}\sum_{i=1}^nK(p_{f_0,i},p_{f,i})\leq\epsilon_n^2,\text{ }n^{-1}\sum_{i=1}^nV(p_{f_0,i},p_{f,i})\leq\epsilon_n^2\}$ with $p_{g,i}$ being the density of $\mathrm{N}[g(\boldsymbol{X}_i),\sigma^2]$. After some calculations,
\begin{align*}
\frac{1}{n}\sum_{i=1}^nK(p_{f_0,i},p_{f,i})&=\frac{1}{2}\log{\left(\frac{\sigma^2}{\sigma_0^2}\right)}-\frac{1}{2}\left(1-\frac{\sigma_0^2}{\sigma^2}\right)
+\frac{\|f-f_0\|_n^2}{2\sigma^2},\\
\frac{1}{n}\sum_{i=1}^nV(p_{f_0,i},p_{f,i})&=\frac{1}{2}\left(1-\frac{\sigma_0^2}{\sigma^2}\right)^2
+\frac{\sigma_0^2\|f-f_0\|_n^2}{\sigma^4}.
\end{align*}
Hence, there are constants $W_3,\widetilde{W}_3>0$ such that $B_n(\epsilon_n)\supset\{\|f-f_0\|_n\leq W_3\epsilon_n,|\sigma^2-\sigma_0^2|\leq \widetilde{W}_3\epsilon_n\}$. By \eqref{eq:coeff}, take $\epsilon_n\geq(3W_2/W_3)(\log{n}/n)^{\alpha^{*}/(2\alpha^{*}+d)}$ and we have $B_n(\epsilon_n)\supset\{\|\boldsymbol{\vartheta}-\boldsymbol{\vartheta}_0\|\leq W_3\epsilon_n/(3W_2),\|\widetilde{\boldsymbol{\theta}}-\widetilde{\boldsymbol{\theta}}_0\|\leq W_3\epsilon_n/(3W_2),|\sigma^2-\sigma_0^2|\leq \widetilde{W}_3\epsilon_n\}$. Therefore by the assumed independence of the priors, $\Pi[B_n(\epsilon_n)]$ can be lower bounded by
\begin{align}\label{eq:priorbound}
\Pi\left[\|\boldsymbol{\vartheta}-\boldsymbol{\vartheta}_0\|\leq\frac{W_3}{3W_2}\epsilon_n\right]\Pi\left[\|\widetilde{\boldsymbol{\theta}}-\widetilde{\boldsymbol{\theta}}_0\|\leq \frac{W_3}{3W_2}\epsilon_n\right]\Pi\left(|\sigma^2-\sigma_0^2|\leq \widetilde{W}_3\epsilon_n\right).
\end{align}
Since $\pi_{\sigma}$ is continuous and $\pi_{\sigma}(\cdot)>0$ by assumption, we have
\begin{align*}
\Pi(|\sigma^2-\sigma_0^2|\leq \widetilde{W}_3\epsilon_n)\geq2\widetilde{W}_3\epsilon_n\inf_{|u-\sigma_0^2|\leq\widetilde{W_3}\epsilon_n}\pi_{\sigma}(u)=2\widetilde{W}_3\epsilon_n\pi_{\sigma}(\sigma_0^2)[1+o(1)],
\end{align*}
which is greater than $e^{-H_1\log{n}}$ for some constant $H_1>0$, where the last equality follows since $\epsilon_n\gtrsim n^{-1/2}$ by assumption. For a set $\mathcal{A}$ in some Euclidean space, we denote $\mathrm{vol}(\mathcal{A})$ to be the volume of $\mathcal{A}$. Let $\widetilde{N}=\prod_{l=1}^d2^{N_l}$. The first prior factor in \eqref{eq:priorbound} is
\begin{align*}
\Pi\left[\|\boldsymbol{\vartheta}-\boldsymbol{\vartheta}_0\|\leq\frac{W_3}{3W_2}\epsilon_n\right]&=\int_{\|\boldsymbol{\vartheta}-\boldsymbol{\vartheta}_0\|\leq W_3\epsilon_n/(3W_2)}\prod_{m_1=0}^{2^{N_1}-1}\cdots\prod_{m_d=0}^{2^{N_d}-1}p(\vartheta_{\boldsymbol{m}})d\vartheta_{\boldsymbol{m}}\\
&\geq p_{\mathrm{min}}^{\widetilde{N}}\mathrm{vol}\{\|\boldsymbol{\vartheta}-\boldsymbol{\vartheta}_0\|\leq W_3\epsilon_n/(3W_2),\|\boldsymbol{\vartheta}\|_\infty\leq R_0\}\\
&=\left(\frac{p_{\mathrm{min}}W_3}{3W_2}\epsilon_n\right)^{\widetilde{N}}\frac{\pi^{\widetilde{N}/2}}{\Gamma(\widetilde{N}/2+1)}\geq e^{-H_2\log{n}},
\end{align*}
for some constant $H_2>0$. We lower bound the second factor in \eqref{eq:priorbound} by
\begin{align*}
\Pi\left[\|\widetilde{\boldsymbol{\theta}}-\widetilde{\boldsymbol{\theta}}_0\|\leq W_3\epsilon_n/(3W_2)\right]\geq\Pi\left[\|\widetilde{\boldsymbol{\theta}}-\widetilde{\boldsymbol{\theta}}_0\|\leq W_3\epsilon_n/(3W_2)\middle|\widetilde{\mathcal{P}}_n\right]\Pi(\widetilde{\mathcal{P}}_n),
\end{align*}
where $\widetilde{\mathcal{P}}_n=\{(\boldsymbol{j},\boldsymbol{k}):\theta_{\boldsymbol{j},\boldsymbol{k}}\neq0, j_l<J_{n,l}(\boldsymbol{\alpha})\text{ for all }l=1,\dotsc,d\text{ and }\theta_{\boldsymbol{j},\boldsymbol{k}}=0\text{ for some }l=1,\dotsc,d, J_{n,l}(\boldsymbol{\alpha})\leq j_l\leq J_{n,l}-1\text{, with }0\leq k_l\leq2^{j_l}-1\}$. Denote $\mathcal{K}_n(\boldsymbol{\alpha})=\{(j_1,\dotsc,j_d):N_l\leq j_l\leq J_{n,l}(\boldsymbol{\alpha})-1,l=1,\dotsc,d\}$. Recall that $n^{-\lambda}\leq\omega_{\boldsymbol{j},n}\leq\min\{\prod_{l=1}^d2^{-j_l(1+\mu_l)},1/2\}$. Using the fact that $\log{(1-x)}\geq-(2\log{2})x$ for $0\leq x\leq0.5$, we have $\log{\Pi(\widetilde{\mathcal{P}}_n)}$ is
\begin{align}
&\sum_{j_1=N_1}^{J_{n,1}(\boldsymbol{\alpha})-1}\cdots\sum_{j_d=N_d}^{J_{n,d}(\boldsymbol{\alpha})-1}
2^{\sum_{l=1}^dj_l}\log{\omega_{\boldsymbol{j},n}}+\sum_{\boldsymbol{j}\in\mathcal{K}_n(\boldsymbol{\alpha})^c}2^{\sum_{l=1}^dj_l}
\log{(1-\omega_{\boldsymbol{j},n})}\nonumber\\
&\geq-\lambda\log{n}\prod_{l=1}^d\sum_{j_l=N_l}^{J_{n,l}(\boldsymbol{\alpha})-1}2^{j_l}-2\log{2}\sum_{\boldsymbol{j}\in\mathcal{K}_n(\boldsymbol{\alpha})^c}
2^{\sum_{l=1}^dj_l}\omega_{\boldsymbol{j},n}.\label{eq:pi}
\end{align}
Define sets $\mathcal{Q}_l,l=1,\dotsc,d$ where $\mathcal{Q}_l$ can be $\{j_l<J_{n,l}(\boldsymbol{\alpha})\}$ or $\{j_l\geq J_{n,l}(\boldsymbol{\alpha})$\}, but with the constraint that not all $\mathcal{Q}_l$'s are $\{j_l<J_{n,l}(\boldsymbol{\alpha})\}$. Then the summation over $\boldsymbol{j}\in\mathcal{K}_n(\boldsymbol{\alpha})^c$ is such that $\boldsymbol{j}$ takes on all $2^d-1$ possible combinations of the $\mathcal{Q}_l$'s, and each combination has the form
\begin{align*}
\sum_{j_1\in\mathcal{Q}_1}\cdots\sum_{j_d\in\mathcal{Q}_d}2^{\sum_{l=1}^dj_l}\omega_{\boldsymbol{j},n}
\leq\sum_{j_1\in\mathcal{Q}_1}\cdots\sum_{j_d\in\mathcal{Q}_d}2^{-\sum_{l=1}^dj_l\mu_l}.
\end{align*}
Among these $2^d-1$ combinations, the configuration with one $\mathcal{Q}_i=\{j_i\geq J_{n,i}(\boldsymbol{\alpha})\}$ and the rest $\mathcal{Q}_l=\{j_l<J_{n,l}(\boldsymbol{\alpha})\},l\neq i,l=1,\dotsc,d$ will dominate the sum, and they are exactly $d$ such configurations. Thus, the sum over $\boldsymbol{j}\in\mathcal{K}_n(\boldsymbol{\alpha})^c$ in \eqref{eq:pi} is bounded above up to some universal constant by
\begin{align*}
\sum_{i=1}^d\sum_{j_i\geq J_{n,i}(\boldsymbol{\alpha})}2^{-j_i\mu_i}\prod_{l\neq i}^d\sum_{j_l<J_{n,l}(\boldsymbol{\alpha})}2^{-j_l\mu_l}\lesssim\sum_{i=1}^d2^{-J_{n,i}(\boldsymbol{\alpha})/2},
\end{align*}
since $\mu_l>1/2,l=1,\dotsc,d$. Hence, \eqref{eq:pi} is bounded below up to some constant multiple by $-\log{n}2^{\sum_{l=1}^dJ_{n,l}(\boldsymbol{\alpha})}-\sum_{l=1}^d2^{J_{n,l}(\boldsymbol{\alpha})/2}$. We then conclude that $\Pi(\widetilde{\mathcal{P}}_n)\geq e^{-H_3\log{n}2^{\sum_{l=1}^dJ_{n,l}(\boldsymbol{\alpha})}}$ for some constant $H_3>0$.

By the assumption in \eqref{eq:density} and denoting $\widetilde{J}=\prod_{l=1}^d[2^{J_{n,l}(\boldsymbol{\alpha})}-2^{N_l}]$,
\begin{align*}
\Pi\left[\|\widetilde{\boldsymbol{\theta}}-\widetilde{\boldsymbol{\theta}}_0\|\leq\frac{W_3}{3W_2}\epsilon_n\middle|\widetilde{\mathcal{P}}_n\right]&=\int_{\|\widetilde{\boldsymbol{\theta}}-\widetilde{\boldsymbol{\theta}}_0\|\leq W_3\epsilon_n/(3W_2)}\prod_{l=1}^d\prod_{j_l=N_l}^{J_{n,l}(\boldsymbol{\alpha})-1}\prod_{k_l=0}^{2^{j_l}-1}p(\theta_{\boldsymbol{j},\boldsymbol{k}})d\theta_{\boldsymbol{j},\boldsymbol{k}}\\
&\geq p_{\mathrm{min}}^{\widetilde{J}}\mathrm{vol}\{\widetilde{\boldsymbol{\theta}}\in\widetilde{\mathcal{P}}_n:\|\widetilde{\boldsymbol{\theta}}-\widetilde{\boldsymbol{\theta}}_0\|\leq W_3\epsilon_n/(3W_2), \|\widetilde{\boldsymbol{\theta}}\|_\infty\leq R_0\}\\
&=\left(\frac{p_{\mathrm{min}}W_3}{3W_2}\epsilon_n\right)^{\widetilde{J}}\frac{\pi^{\widetilde{J}/2}}{\Gamma{(\widetilde{J}/2+1)}}\geq e^{-H_4\log{n}2^{\sum_{l=1}^dJ_{n,l}(\boldsymbol{\alpha})}},
\end{align*}
for some constant $H_4>0$. Therefore by multiplying all the lower bounds obtained for \eqref{eq:priorbound}, it follows that $\Pi[B_n(\epsilon_n)]\geq e^{-Cn\epsilon_n^2}$ for some constant $C>0$ only if $(\log{n}/n)2^{\sum_{l=1}^dJ_{n,l}(\boldsymbol{\alpha})}\lesssim\epsilon_n^2$. This implies that $\epsilon_n\gtrsim(\log{n}/n)^{\alpha^{*}/(2\alpha^{*}+d)}$ for $2^{J_{n,l}(\boldsymbol{\alpha})}\asymp(n/\log{n})^{\alpha^{*}/\{\alpha_l(2\alpha^{*}+d)\}},l=1,\dotsc,d$.

It now remains to show that $\mathrm{E}_0\Pi(\mathcal{F}_n^c|\boldsymbol{Y})\rightarrow0$. By continuous embedding, this is equivalent to showing that $\mathrm{E}_0\Pi(\Theta_n^c|\boldsymbol{Y})\rightarrow0$. Observe that $\Theta_n^c=\bigcup_{\boldsymbol{j}\geq\boldsymbol{J}_n(\boldsymbol{\alpha})}[\boldsymbol{\theta}_{(\boldsymbol{j})}\neq\boldsymbol{0}]$, with $[\boldsymbol{\theta}_{(\boldsymbol{j})}\neq\boldsymbol{0}]$ representing the set such that $\theta_{\boldsymbol{j},\boldsymbol{k}}\neq0$ for at least one $k_l$ at some $l=1,\dotsc,d$. Define $\mathcal{A}_{\boldsymbol{j}}(m)=\{\boldsymbol{\theta}_{\boldsymbol{j}}:\text{exactly $m$ among all $2^{\sum_{l=1}^dj_l}$ elements are not zero, and the rest are zeroes}\}$. It follows that $[\boldsymbol{\theta}_{(\boldsymbol{j})}\neq\boldsymbol{0}]$ is a union of $\mathcal{A}_{\boldsymbol{j}}(m)$ across $m=1,\dotsc,2^{\sum_{l=1}^dj_l}$ and we have $\Pi(\Theta_n^c|\boldsymbol{Y})\leq\sum_{\boldsymbol{j}\geq\boldsymbol{J}_n(\boldsymbol{\alpha})}\left[\Pi(\mathcal{A}_{\boldsymbol{j}}(1)|\boldsymbol{Y})+\cdots+\Pi\left(\mathcal{A}_{\boldsymbol{j}}\left(2^{\sum_{l=1}^dj_l}\right)\middle|\boldsymbol{Y}\right)\right]$. After some calculations, it turns out that the first sum is $O_{P_0}(e^{-C\log{n}})$ while the rest of the terms are $o_{P_0}(e^{-C\log{n}})$ for some large enough constant $C>0$. We then conclude that $\mathrm{E}_0\Pi(\mathcal{F}_n^c|\boldsymbol{Y})\lesssim e^{-C\log{n}}\rightarrow0$ as $n\rightarrow\infty$.
\end{proof}

\begin{corollary}\label{cor:l2contract}
As a consequence of Lemma \ref{lem:l2contract} above, we have with posterior probability at least $1-n^{-P_4}$ that
\begin{align*}
\|\boldsymbol{\vartheta}-\boldsymbol{\vartheta}_0\|\lesssim(n/\log{n})^{-\alpha^{*}/(2\alpha^{*}+d)},\quad\|\boldsymbol{\theta}-\boldsymbol{\theta}_0\|\lesssim(n/\log{n})^{-\alpha^{*}/(2\alpha^{*}+d)}.
\end{align*}
\end{corollary}
\begin{proof}
If $f$ has wavelet expansion as in \eqref{eq:prior} at resolution $\boldsymbol{J}_n$, then by the property of $L_2$-projection and Lemma \ref{lem:phiphi}, we have $\|f-f_0\|_n\geq\|f-K_{\boldsymbol{J}_n}(f_0)\|_n\gtrsim(\|\boldsymbol{\vartheta}-\boldsymbol{\vartheta}_0\|^2+\|\boldsymbol{\theta}-\boldsymbol{\theta}_0\|^2)^{1/2}$ since $2^{\sum_{l=1}^dJ_{n,l}}=o(n)$ by assumption. The result follows by applying Lemma \ref{lem:l2contract}.
\end{proof}

\begin{lemma}\label{lem:type2}
Project $f$ unto the wavelet bases at resolution $\boldsymbol{J}_n$ and write the regression model in \eqref{eq:model} (assuming known $\sigma=\sigma_0$) as $\boldsymbol{Y}=\boldsymbol{\Psi\theta}+\boldsymbol{\varepsilon}$ with $\boldsymbol{\Psi}$ the wavelet basis matrix and $\boldsymbol{\varepsilon}\sim\mathrm{N}(\boldsymbol{0},\sigma_0^2\boldsymbol{I})$. Let $\widehat{f}_n(\boldsymbol{x}):=\boldsymbol{\psi}_{\boldsymbol{J}_n}(\boldsymbol{x})^T(\boldsymbol{\Psi}^T\boldsymbol{\Psi})^{-1}\boldsymbol{\Psi}^T\boldsymbol{Y}$ be the corresponding least squares estimator with $\boldsymbol{\psi}_{\boldsymbol{J}_n}(\boldsymbol{x})$ being the vector of all wavelet functions at resolution $\boldsymbol{J}_n$ evaluated at $\boldsymbol{x}$. Let $2^{\sum_{l=1}^dJ_{n,l}}=o(n)$, then for any $f\in\mathcal{B}^{\boldsymbol{\alpha}}_{\infty,\infty}(R)$, there exist constants $C,Q_1,Q_2>0$ such that the following hold:
\begin{gather}
\|\mathrm{E}_f\widehat{f}_n-f\|_\infty\leq C\sum_{l=1}^d2^{-\alpha_lJ_{n,l}},\label{eq:type21}\\
P_f\left(\|\widehat{f}_n-\mathrm{E}_f\widehat{f}_n\|_\infty\geq Q_1\sqrt{\frac{2^{\sum_{l=1}^dJ_{n,l}}\log{n}}{n}}+\sqrt{2Q_2\frac{2^{\sum_{l=1}^dJ_{n,l}}}{n}x}\right)\leq e^{-x}.\label{eq:type22}
\end{gather}
\end{lemma}

\begin{proof}
By Proposition \ref{prop:wavelet}, for any $f\in\mathcal{B}^{\boldsymbol{\alpha}}_{\infty,\infty}(R)$, there is a $\boldsymbol{\xi}$ such that $\|f-\boldsymbol{\psi}_{\boldsymbol{J}_n}(\cdot)^T\boldsymbol{\xi}\|_\infty\lesssim\sum_{l=1}^d2^{-\alpha_lJ_{n,l}}$. Therefore by adding and subtracting $\boldsymbol{\Psi}\boldsymbol{\xi}$ and using the triangle inequality, $\|\mathrm{E}_f\widehat{f}_n-f\|_\infty$ is
\begin{align*}
&\|\boldsymbol{\psi}_{\boldsymbol{J}_n}(\cdot)^T(\boldsymbol{\Psi}^T\boldsymbol{\Psi})^{-1}\boldsymbol{\Psi}^T\boldsymbol{F}-f\|_\infty\\
&\qquad\leq\|\boldsymbol{\psi}_{\boldsymbol{J}_n}(\cdot)^T(\boldsymbol{\Psi}^T\boldsymbol{\Psi})^{-1}\boldsymbol{\Psi}^T(\boldsymbol{F}-\boldsymbol{\Psi\xi})\|_\infty
+\|\boldsymbol{\psi}_{\boldsymbol{J}_n}(\cdot)^T\boldsymbol{\xi}-f\|_\infty,
\end{align*}
where the second term is $O(\sum_{l=1}^d2^{-\alpha_lJ_{n,l}})$. For a matrix $\boldsymbol{A}$, let $\|\boldsymbol{A}\|_{(\infty,\infty)}=\max_i\sum_j|a_{ij}|$ (max of absolute row sums) and $\|\boldsymbol{A}\|_{(1,1)}=\max_j\sum_i|a_{ij}|$ (max of absolute column sums). Using H\"{o}lder's inequality $|\boldsymbol{x}^T\boldsymbol{y}|\leq\|\boldsymbol{x}\|_1\|\boldsymbol{y}\|_\infty$ and definition of the induced matrix norm $\|\boldsymbol{Ax}\|_1\leq\|\boldsymbol{A}\|_{(1,1)}\|\boldsymbol{x}\|_1$, the first term is bounded by
\begin{align}\label{eq:normbound}
\sup_{\boldsymbol{x}\in[0,1]^d}\|\boldsymbol{\psi}_{\boldsymbol{J}_n}(\boldsymbol{x})\|_1\|(\boldsymbol{\Psi}^T\boldsymbol{\Psi})^{-1}\|_{(1,1)}\|\boldsymbol{\Psi}^T\|_{(1,1)}
\|\boldsymbol{F}-\boldsymbol{\Psi\xi}\|_\infty.
\end{align}
By \eqref{eq:regular} with $\boldsymbol{r}=\boldsymbol{0}$, it holds that $\|\boldsymbol{\psi}_{\boldsymbol{J}_n}(\boldsymbol{x})\|_1\lesssim2^{\sum_{l=1}^dJ_{n,l}/2}$ uniformly in $\boldsymbol{x}\in[0,1]^d$. Note that since each entry of $\boldsymbol{\Psi}$ is a dilated version of the base CDV wavelet with compact support, it follows that $\boldsymbol{\Psi}^T\boldsymbol{\Psi}$ is banded. Furthermore by choosing $2^{\sum_{l=1}^dJ_{n,l}}=o(n)$, all eigenvalues of $\boldsymbol{\Psi}^T\boldsymbol{\Psi}$ are $\asymp n$ by virtue of Lemma \ref{lem:2phiphi}. Therefore by appealing to Lemma A.4 of \citep{yoo2015}, we conclude $\|(\boldsymbol{\Psi}^T\boldsymbol{\Psi})^{-1}\|_{(\infty,\infty)}\lesssim n^{-1}$. Since $(\boldsymbol{\Psi}^T\boldsymbol{\Psi})^{-1}$ is symmetric, it follows that $\|(\boldsymbol{\Psi}^T\boldsymbol{\Psi})^{-1}\|_{(1,1)}=\|(\boldsymbol{\Psi}^T\boldsymbol{\Psi})^{-1}\|_{(\infty,\infty)}\lesssim n^{-1}$. Now observe that $\|\boldsymbol{\Psi}^T\|_{(1,1)}=\max_{1\leq i\leq n}\sum_{\boldsymbol{j},\boldsymbol{k}}|\psi_{\boldsymbol{j},\boldsymbol{k}}(X_i)|\leq\sup_{\boldsymbol{x}\in[0,1]^d}\|\boldsymbol{\psi}_{\boldsymbol{J}_n}(\boldsymbol{x})\|_1$, and this is $O(2^{\sum_{l=1}^dJ_{n,l}/2})$ as shown above. It now remains to bound $\|\boldsymbol{F}-\boldsymbol{\Psi\xi}\|_\infty$, and we know it is $O(\sum_{l=1}^d2^{-\alpha_lJ_{n,l}})$ by Proposition \ref{prop:wavelet}. Combine everything and use the assumption that $2^{\sum_{l=1}^dJ_{n,l}}\leq n$ to conclude \eqref{eq:type21}.

By construction, $\widehat{f}_n-\mathrm{E}_f\widehat{f}_n\sim\mathrm{GP}(0,\sigma_0^2\Sigma_{\boldsymbol{J}_n})$ where the covariance kernel $\Sigma_{\boldsymbol{J}_n}(\boldsymbol{x},\boldsymbol{y})=\boldsymbol{\psi}_{\boldsymbol{J}_n}(\boldsymbol{x})^T(\boldsymbol{\Psi}^T\boldsymbol{\Psi})^{-1}\boldsymbol{\psi}_{\boldsymbol{J}_n}(\boldsymbol{y})$ for any $\boldsymbol{x},\boldsymbol{y}\in[0,1]^d$. Since the wavelets are uniformly bounded and applying \eqref{eq:regular} with $\boldsymbol{r}=\boldsymbol{0}$, we can deduce that for any $\boldsymbol{x}\in[0,1]^d$, $\|\boldsymbol{\psi}_{\boldsymbol{J}_n}(\boldsymbol{x})\|^2\leq\max_{\boldsymbol{j}\leq\boldsymbol{J}_n-\boldsymbol{1}_d,\boldsymbol{k}}|
\psi_{\boldsymbol{j},\boldsymbol{k}}(\boldsymbol{x})|\sum_{\boldsymbol{j}\leq\boldsymbol{J}_n-\boldsymbol{1}_d,\boldsymbol{k}}|\psi_{\boldsymbol{j},\boldsymbol{k}}(\boldsymbol{x})|\lesssim\prod_{l=1}^d2^{J_{n,l}}$. Then by appealing to Lemma \ref{lem:2phiphi},
\begin{align*}
\sup_{\boldsymbol{x}\in[0,1]^d}\Sigma_{\boldsymbol{J}_n}(\boldsymbol{x},\boldsymbol{x})\lesssim\|(\boldsymbol{\Psi}^T\boldsymbol{\Psi})^{-1}\|_{(2,2)}
\sup_{\boldsymbol{x}\in[0,1]^d}\|\boldsymbol{\psi}_{\boldsymbol{J}_n}(\boldsymbol{x})\|^2\leq Q_2n^{-1}2^{\sum_{l=1}^dJ_{n,l}}
\end{align*}
for some constant $Q_2>0$. By the Borell's inequality (see Proposition A.2.1 from \citep{empirical} or Theorem 2.5.8 in \citep{nickl2016}), we have for any $x\geq0$ that
\begin{align}\label{eq:borell}
P_f\left(\|\widehat{f}_n-\mathrm{E}_f\widehat{f}_n\|_\infty\geq\mathrm{E}_f\|\widehat{f}_n-\mathrm{E}_f\widehat{f}_n\|_\infty+\sqrt{2Q_2n^{-1}2^{\sum_{l=1}^dJ_{n,l}}x}\right)\leq e^{-x}.
\end{align}
Define $\boldsymbol{\eta}:=\boldsymbol{\Psi}^T\boldsymbol{\varepsilon}$. Observe that $\widehat{f}_n-\mathrm{E}_f\widehat{f}_n$ is $\boldsymbol{\psi}_{\boldsymbol{J}_n}(\cdot)^T(\boldsymbol{\Psi}^T\boldsymbol{\Psi})^{-1}\boldsymbol{\eta}$, and by H\"{o}lder's inequality, $\mathrm{E}_f\|\widehat{f}_n-\mathrm{E}_f\widehat{f}_n\|_\infty$ is bounded above by
\begin{align*}
\sup_{\boldsymbol{x}\in[0,1]^d}\|\boldsymbol{\psi}_{\boldsymbol{J}_n}(\boldsymbol{x})\|_1\|(\boldsymbol{\Psi}^T\boldsymbol{\Psi})^{-1}\|_{(\infty,\infty)}\mathrm{E}\|\boldsymbol{\eta}\|_\infty
\lesssim n^{-1}2^{\sum_{l=1}^dJ_{n,l}/2}\mathrm{E}\|\boldsymbol{\eta}\|_\infty,
\end{align*}
in view of the bounds established in \eqref{eq:normbound}. Let us index the rows and columns of $\boldsymbol{\Psi}^T\boldsymbol{\Psi}$ with multi-indices of the form $(\boldsymbol{j},\boldsymbol{k})$. Since $\boldsymbol{\eta}\sim\mathrm{N}(\boldsymbol{0},\sigma_0^2\boldsymbol{\Psi}^T\boldsymbol{\Psi})$, we can apply Lemma 2.3.4 of \citep{nickl2016} to conclude that for $Z_{\boldsymbol{j},\boldsymbol{k}}\sim\mathrm{N}(0,1)$ i.i.d.~with $(\boldsymbol{j},\boldsymbol{k})$ running across all the indices of the wavelet series up to resolution $\boldsymbol{J}_n$,
\begin{align*}
\mathrm{E}\|\boldsymbol{\eta}\|_\infty\leq\max_{\boldsymbol{j}\leq\boldsymbol{J}_n-\boldsymbol{1}_d,\boldsymbol{k}}\sqrt{(\boldsymbol{\Psi}^T\boldsymbol{\Psi})_{(\boldsymbol{j},\boldsymbol{k}),(\boldsymbol{j},\boldsymbol{k})}}
\mathrm{E}\left(\max_{\boldsymbol{j}\leq\boldsymbol{J}_n-\boldsymbol{1}_d,\boldsymbol{k}}|Z_{\boldsymbol{j},\boldsymbol{k}}|\right)\lesssim\sqrt{n\log{n}},
\end{align*}
where we have utilized Lemma \ref{lem:2phiphi} to upper bound the diagonals of $\boldsymbol{\Psi}^T\boldsymbol{\Psi}$, under the assumption that $2^{\sum_{l=1}^dJ_{n,l}}=o(n)$. Combine all these established bounds back into \eqref{eq:borell} to deduce \eqref{eq:type22}.
\end{proof}

\begin{lemma}\label{lem:2phiphi}
For any vector $\boldsymbol{\theta}$, we have
\begin{align*}
\|\boldsymbol{\theta}\|^2\left(n-\prod_{l=1}^d2^{J_{n,l}}\right)\lesssim n\|K_{\boldsymbol{J}_n}(f)\|_n^2=\boldsymbol{\theta}^T\boldsymbol{\Psi}^T\boldsymbol{\Psi}\boldsymbol{\theta}\lesssim \|\boldsymbol{\theta}\|^2\left(n+\prod_{l=1}^d2^{J_{n,l}}\right).
\end{align*}
In particular if $2^{\sum_{l=1}^dJ_{n,l}}=o(n)$, then the maximum eigenvalue of $\boldsymbol{\Psi}^T\boldsymbol{\Psi}$ is $O(n)$, while its minimum eigenvalue is $\gtrsim n$.
\end{lemma}
\begin{proof}
By definition, $n^{-1}\boldsymbol{\theta}^T\boldsymbol{\Psi}^T\boldsymbol{\Psi}\boldsymbol{\theta}=\int_{[0,1]^d}K_{\boldsymbol{J}_n}(f)(\boldsymbol{x})^2dG_n(\boldsymbol{x})$. Let us take $f$ in Lemma \ref{lem:discrete} as $K_{\boldsymbol{J}_n}(f)$. Then $\int_{[0,1]^d}K_{\boldsymbol{J}_n}(f)(\boldsymbol{x})^2d\boldsymbol{x}=\|\boldsymbol{\theta}\|^2$ by orthonormality. By the Cauchy-Schwarz inequality,
\begin{align*}
\int_{[0,1]^d}\left|\frac{\partial^d}{\partial x_1\cdots\partial x_d}K_{\boldsymbol{J}_n}(f)(\boldsymbol{x})^2\right|d\boldsymbol{x}\leq2\|K_{\boldsymbol{J}_n}(f)\|_2\left\|\frac{\partial^d}{\partial x_1\cdots\partial x_d}K_{\boldsymbol{J}_n}(f)\right\|_2,
\end{align*}
where $\|K_{\boldsymbol{J}_n}(f)\|_2=\|\boldsymbol{\theta}\|$ again by orthonormality, while
\begin{align*}
\left\|\frac{\partial^d}{\partial x_1\cdots\partial x_d}K_{\boldsymbol{J}_n}(f)\right\|_2\lesssim\sqrt{\sum_{j_1=N_1-1}^{J_{n,1}-1}\cdots\sum_{j_d=N_d-1}^{J_{n,d}-1}\sum_{\boldsymbol{k}}2^{2\sum_{l=1}^dj_l}\theta_{\boldsymbol{j},\boldsymbol{k}}^2}
\end{align*}
is $O\left(\prod_{l=1}^d2^{J_{n,l}}\|\boldsymbol{\theta}\|\right)$ by applying the third display of Section 5 in \citep{cai2002}. The last statement follows since the maximum or minimum eigenvalue is the maximization or minimization of $\boldsymbol{\theta}^T\boldsymbol{\Psi}^T\boldsymbol{\Psi}\boldsymbol{\theta}/\|\boldsymbol{\theta}\|^2$ over $\boldsymbol{\theta}\neq\boldsymbol{0}$.
\end{proof}

\bibliographystyle{apa}
\bibliography{adaptref}
\end{document}